\documentclass[11pt]{amsart}
\usepackage[bottom=1in, right=1in, left=1in, top=1in]{geometry}

\RequirePackage[OT1]{fontenc}
\RequirePackage{amsthm,amsmath}
\RequirePackage[colorlinks,citecolor=magenta,urlcolor=blue,linkcolor=magenta]{hyperref}
\RequirePackage{hypernat}
\RequirePackage{comment}
\RequirePackage{graphicx,subfigure,latexsym,amssymb}
\RequirePackage{float,epsfig,multirow,rotating,times}
\RequirePackage{upgreek,wrapfig}
\usepackage{amsthm, amssymb}
\usepackage{tikz}
\usepackage{float}
\usetikzlibrary{matrix}
\usetikzlibrary{arrows}
\usetikzlibrary{decorations.pathreplacing,calligraphy}
\usepackage{nicematrix}
\usepackage{arydshln}
\usepackage{blkarray}
\usepackage{mathdots}
\usepackage{multicol}
\usepackage[dvipsnames,table]{xcolor}
\usepackage{amsthm, amssymb}
\usepackage{caption}
\newtheorem{theorem}{Theorem}[section]
\newtheorem{proposition}[theorem]{Proposition}
\newtheorem{corollary}[theorem]{Corollary}
\newtheorem{lemma}[theorem]{Lemma}
\theoremstyle{definition}
\newtheorem{example}[theorem]{Example}
\theoremstyle{definition}
\newtheorem{definition}[theorem]{Definition}
\theoremstyle{definition}
\newtheorem{remark}[theorem]{Remark}
\theoremstyle{definition}

\usepackage{mathtools} % Bonus

\usetikzlibrary{decorations.pathreplacing,angles,quotes}

\renewcommand{\int}{\mathrm{Int}}

\newcommand{\Xw}{\overline{X_w}}
\newcommand{\SW}{\mathrm{SW}}
\newcommand{\Ess}{\mathrm{Ess}}
\newcommand{\dom}{\mathrm{dom}}
\newcommand{\Gvwt}{G_{v,w}}
\newcommand{\Nvw}{\mathcal{N}_{v,w}}
\newcommand{\Dc}{D^{\circ}}
\newcommand{\dash}{\text{---}}
\newcommand{\Zg}{Z^{(v)}_w}
\newcommand{\cn}{\nu}
\newcommand{\rk}{r}
\newcommand{\GC}{G_{\mathcal C}}
\newcommand{\Gw}{G_w}

\newcommand{\sym}{\mathrm{sym}}
\newcommand{\low}{\mathrm{low}}
\newcommand{\init}{\mathrm{in}}
\newcommand{\codim}{\mathrm{codim}}

\newcommand{\indep}{\perp \!\!\! \perp}

\makeatletter
\newcommand{\doublewidetilde}[1]{{%
  \mathpalette\double@widetilde{#1}%
}}
\newcommand{\double@widetilde}[2]{%
  \sbox\z@{$\m@th#1\widetilde{#2}$}%
  \ht\z@=.52\ht\z@
  \widetilde{\box\z@}%
}
\makeatother

\newcommand{\Gvwtt}{\overline{G_{v,w}}}

\author[\tiny Elke Neuhaus]{Elke Neuhaus}
\address{Max Planck Institute for Mathematics in the Sciences}
\email{elke.neuhaus@mis.mpg.de} 

\author[\tiny Niharika Chakrabarty Paul]{Niharika Chakrabarty Paul}
\address{Max Planck Institute for Mathematics in the Sciences}
\email{niharika.paul@mis.mpg.de} 

\author[\tiny Irem Portakal]{Irem Portakal}
\address{Max Planck Institute for Mathematics in the Sciences}
\email{mail@irem-portakal.de}

\begin{document}

\title[\tiny Torus Actions on Matrix Schubert and Kazhdan-Lusztig Varieties,
and their Links to Statistical Models]{Torus Actions on Matrix Schubert and Kazhdan-Lusztig Varieties,\\
and their Links to Statistical Models}

\begin{abstract}
We investigate the toric geometry of two families of generalised determinantal varieties arising from permutations: Matrix Schubert varieties ($\overline{X_w}$) and Kazhdan-Lusztig varieties ($\Nvw$). Matrix Schubert varieties can be written as $\overline{X_w} = Y_w \times \mathbb C^d$, where $d$ is maximal. 
We are especially interested in the structure 
and complexity of these varieties $Y_w$ and $\Nvw$ under the so-called usual torus actions. 
In the case when $Y_w$ is toric, we provide a full characterisation of the simple reflections $s_i$ that render ${Y_{w \cdot s_i}}$ toric, as well as the corresponding changes to the weight cone.
For Kazhdan-Lusztig varieties, we consider how moving one of the two permutations $v,w$ along a chain in the Bruhat poset affects their complexity. 
Additionally, we study the complexity of these varieties, for permutations $v$ and $w$ of a specific structure.
Finally, we consider the links between these determinantal varieties and two classes of statistical models; namely conditional independence and quasi-independence models.
\end{abstract}

\maketitle

\section{Introduction and basic notions}
Matrix Schubert varieties, introduced by Fulton \cite{fulton1992flags}, arise naturally as orbit closures under group actions on spaces of matrices associated to permutations and are pivotal in the study of degeneracy loci flagged vector bundles, and Schubert calculus.  Kazhdan-Lusztig (KL) varieties, which generalise the setting of matrix Schubert varieties, are associated with pairs of elements in a Weyl group, typically pairs of permutations. These varieties were investigated in \cite{woo2008governing} in the study of singularities of Schubert varieties of flag manifolds. The Castelnuovo-Mumford regularity of a matrix Schubert variety can be computed as the difference between the highest-degree and lowest-degree homogeneous components of its $K$-polynomial (\cite{rajchgot2021degrees}). These $K$-polynomials correspond precisely to Grothendieck polynomials, introduced by Lascoux and Schützenberger \cite{lascoux1982structure} as polynomial representatives for the classes of structure sheaves in the $K$-theoretic Schubert calculus of the variety of complete flags. Knutson and Miller \cite{knutson2005geometry} also showed that Schubert polynomials appear as multidegrees of matrix Schubert varieties.

The study of torus actions on algebraic varieties, such as matrix Schubert and Kazhdan-Lusztig varieties, plays a crucial role in understanding their algebro-geometric properties. In this paper, we use the language of affine normal $T$-varieties admitting effective torus actions which generalise toric varieties by allowing positive \emph{complexity}, defined as the difference between the dimension of the variety and the dimension of the acting torus~\cite{altmann2012geometry}. A $T$-variety can be described by its p-divisor, a partly combinatorial object consisting of a geometric part, a special quotient of $X$ by the torus action, and a combinatorial part, a fan of rational polyhedral cones. When the complexity is small, the combinatorial part encodes much of the structure of the $T$-variety. In particular, if the complexity is zero, then $X$ is toric. Matrix Schubert varieties arise as orbit closures under the action of $B \times B$, where $B$ is a Borel subgroup of the general linear group $GL_n$. The matrix Schubert variety $\overline{X_w}$ for a permutation $w \in S_n$ is isomorphic to $Y_w \times \mathbb{C}^d$, where $d$ is maximal. We consider the \emph{usual torus action} (restricted, in particular, to $Y_w$), which refers to the restriction of the Borel subgroup action to its diagonal torus. Similarly, Kazhdan-Lusztig varieties inherit this torus action from their ambient Schubert varieties. Recent work has classified the possible complexities of matrix Schubert varieties and Kazhdan-Lusztig varieties under the usual torus action, showing that in the former case, complexity one does not occur, and providing a combinatorial characterisation of weight cones as edge cones of acyclic directed graphs~\cite{donten2021complexity}. By counting the number of connected components of these graphs, one can readily determine the dimension of the torus action (Lemma~\ref{lem:weightcondim}). These advances contribute to a finer understanding of the orbit structure and affine geometry of these varieties, and highlight the interplay of algebraic, combinatorial, and geometric techniques in the study of $T$-varieties and their applications to matrix Schubert and Kazhdan-Lusztig varieties~\cite{portakal2023rigid}.

This paper studies the usual torus action on matrix Schubert and Kazhdan-Lusztig varieties using a combinatorial framework mostly based on the structure of opposite Rothe diagrams and directed acyclic graphs. Moreover, it investigates applications of these varieties in statistical models, highlighting the interplay between algebraic combinatorics and algebraic statistics. The organisation of this paper is as follows. In Section~\ref{sec: MS varieties}, we begin with basic definitions for matrix Schubert varieties such as (opposite) Rothe diagrams, the usual torus action, and Fulton's determinantal conditions, focusing on the toric class of matrix Schubert varieties under the usual torus action. Here, the weight cone (which is the convex polyhedral cone associated to the toric $Y_w$) is the edge cone of a bipartite directed acyclic graph $G_w$. Building on two characterisations of toric matrix Schubert varieties (\cite[Theorem~1.6]{stelzer2025matrix} and \cite[Theorem~3.4]{escobar2016toric}), we study how the complexity of the usual torus action of a toric matrix Schubert variety $Y_w$ changes when the permutation $w$ is right-multiplied by a simple reflection. Moreover, we analyze how the directed acyclic graph $G_w$ and consequently the weight cone change under such operations, thereby characterising all such possible toric varieties in detail. Let $e_1, \cdots e_n$ and $f_1, \cdots f_n$ denote the standard basis for $\mathbb Z^n \times \mathbb Z^n$.

\begin{theorem}[{Theorem~\ref{thm: all toric mS}, Corollary~\ref{cor: toric cases}, \ref{cor: change of the weight cone}}]
   Let $\overline{X_w} = Y_w \times \mathbb{C}^d$ be a matrix Schubert variety such that $Y_w$ is toric with respect to the usual torus action. All simple reflections $s_i$ for which $Y_{w \cdot s_i}$ is toric are explicitly determined. In particular, if $\dim(Y_w) = \dim(Y_{w \cdot s_i})$, then either $Y_w = Y_{w \cdot s_i}$ or the weight cone undergoes one of the following changes:
   \begin{itemize}
       \item a new ray generator $e_{w(i)} - f_{i}$ is added;
       \item the ray generator $e_{w(i+1)} - f_{i}$ is removed.
   \end{itemize}
\end{theorem}

Section~\ref{sec: KL varieties} is devoted to the introduction of Kazhdan-Lusztig varieties and the weight cone associated with the usual torus action and its related directed acyclic graph $G_{v,w}$. Noting that a refinement of the underlying undirected graph of $G_{v,w}$ is isomorphic to the graph used for the torus action on Richardson varieties in~\cite{lee21toricbruhatintpolytopes}, we characterise certain complexity properties of Kazhdan-Lusztig varieties. Namely, we determine how the complexity changes when one of the permutations $v$ or $w$ is moved along a chain in the Bruhat poset. In particular, the complexity of $\mathcal N_{v,w}$ is either the same or one more than the complexity of $\mathcal N_{v,w'}$ where $w$ is an atom of $w'$ (see Proposition~\ref{prop: complexity change}). In Lemma~\ref{lem:ccs of GC}, we show that if $\mathcal N_{v,w}$ is toric, then the connected components of $G_{v,w}$ are exactly the subsets of vertices for which the values of $v$ and $w$ match cyclically along these vertices. Let $w = w'\cdot t_{a,b}$ be an atom of $w'$ and assume that $\mathcal N_{v,w'}$ is toric. Then, $\mathcal N_{v,w}$ is toric if and only if the vertices $a$ and $b$ are not in the same connected component of $G_{v,w}$ (Proposition~\ref{prop: extend toric interval}). In a similar manner, in Lemma~\ref{lem: gluing two intervals}, we characterise the case when $\mathcal N_{u,w}$ is toric where both $\mathcal N_{u,v}$ and $\mathcal N_{v,w}$ are toric. 

In Section~\ref{sec: toric families of KL}, to better understand the complexity of $\Nvw$ in more general situations, we analyse the connected components of $G_{v,w}$ in more detail. In particular, we focus on cases where there are no actual or unexpected zeros (Definition~\ref{def: unexpected zero}) and relate these findings to the structure of the opposite Rothe diagram $\Dc(v)$.

\begin{theorem}[{Lemma~\ref{lem: connected components for no unexpected zeros}, \ref{lem: isolated vertices for no unexpected zeros}}]
    Let $v,w \in S_n$ be such that there are no unexpected zeros. Then, the set of connected components of $G_{v,w}$ with more than one vertex has the same cardinality as the set of 
    coordinates in $\Dc(v)$ with no elements in the same row west of them and none in the same column south of them.
    Moreover, the number of isolated vertices of $\Gvwt$ is the same as the cardinality of the set of $1$s on the antidiagonal of $v$ that do not share a row or column with the opposite Rothe diagram of $v$.
\end{theorem}
We next investigate the cycles of the underlying undirected graph of $G_{v,w}$ in Section~\ref{subsec: cycles of Gvw}. In particular, when $\mathcal{N}_{v,w}$ is toric, these cycles correspond to the generators of the ideal by \cite[Proposition 4.3]{gitler2010ring}. We compute the cyclomatic number of $G_{v,w}$, using a certain subset of $\Dc(v) \times \Dc(v)$. The difference between the size of this subset and the number of elements in the opposite Rothe diagram of $w$ yields the complexity of $\mathcal{N}_{v,w}$ in the absence of unexpected or actual zeros. Finally, this section addresses the complexity of specific families of Kazhdan-Lusztig varieties $\mathcal{N}_{v,w}$, including the case where the opposite Rothe diagram of $w$ is a rectangle (Proposition~\ref{prop: rectangle comp}) and the case $\mathcal{N}_{v, w_0 t}$, where $t$ is a transposition and $w_0$ is the longest permutation (Proposition~\ref{prop: change w for complexity}).

Section~\ref{sec: statistical models} is devoted to the relation of matrix Schubert varieties and Kazhdan-Lusztig varieties to two statistical models: Gaussian conditional independence (CI) and quasi-independence models. We begin by introducing symmetric and lower triangular matrix Schubert varieties and determine their dimension and complexity in terms of the opposite Rothe diagram in Section~\ref{subsec: ci models}. Unlike classical matrix Schubert varieties, this allows the full range of possible complexities, including complexity one. In previous work~\cite{fink2016matrix}, the CI ideal associated to a CI statement realised as a matrix Schubert variety is studied. In Proposition~\ref{prop: complexity Mschu}, we determine the complexity of matrix Schubert varieties in this case and generalise this to Kazhdan-Lusztig varieties in Lemma~\ref{KL CI}, determining their complexity as well. We end the section by noting that the toric matrix Schubert varieties are quasi-independence models and, in Theorem~\ref{thm: rational MLE}, prove that these have rational maximum likelihood estimate. Alongside the \texttt{MatrixSchubert} package for Macaulay2~\cite{almousa2025matrixschubert}, we utilise the functions available in~\cite{NiharikaCP_KLvarieties} for computations of the examples throughout this paper.

\subsubsection*{Background on $T$-varieties} We use the notation in \cite{altmann2012geometry}, and refer the reader to this paper for further details. Let $T$ be a torus and $\mathrm{M}(T)$ be its associated character lattice. We then denote by $\mathrm{M}(T)_{\mathbb{R}}:= \mathrm{M}(T) \otimes_{\mathbb Z} \mathbb R$, the real vector space that is given as the span of this character lattice. We further say that the affine normal variety $X$ admits an \emph{effective $T$-torus action} if the set $S := \cap_{p\in X}T_p$ is empty, where $T_p = \{t\in T|t\cdot p=p\}$.  
\begin{definition}
    An affine normal variety $X$ is a \emph{$T$-variety of complexity $d$} if it admits an effective $T$-torus action such that $\dim(T) = \dim(X)-d.$
\end{definition}
\noindent Let us assume that $X$ admits an effective $T$-torus action. Now, we wish to study the complexity via convex geometry, in particular via the \emph{weight cone}. The weight cone associated to a torus $T$ is the cone that is generated by all of the weights of the torus action on $X$ in $\mathrm{M}(T)_{\mathbb{R}}$. Let $p\in X$ be generic. Then, $\overline{T\cdot p}$ is the affine normal toric variety that is associated to the weight cone of the $T$-torus action. Thus, $\dim(\overline{T\cdot p})$ is equal to the dimension of the weight cone of the toric action. Furthermore, $\dim(\overline{T\cdot p}) = \dim(T)$, as the action is effective. Hence, we can define the complexity as
$\dim(X) - \dim(\text{associated weight cone}).$

Now, if the $T$-action is not effective, then the set $S$ is non-trivial, but the $T/S$-action on $X$ is effective. Moreover, $tS\cdot p = t\cdot p$, which implies that the weights of the $T/S$-action are the same as those of the $T$-action. Hence, the normal variety $X$ has the same complexity with respect to both actions. In this case we abuse notation and formally mean $T/S$ when talking about $T$.

\subsubsection*{Toric ideals from directed acyclic graphs (DAGs)} \noindent We refer the reader to \cite{gitler2010ring,kara2023rigid} for further details. Let $G$ be a DAG with edge set $E$ and vertex set $V$. Then, each edge $e=(a\rightarrow b)\in E$ goes from the vertex $a$ to the vertex $b$.  
We define the edge cone associated to the graph $G$ as 
$$\sigma_G = \mathrm{Cone}(e_i-e_j|(i\rightarrow j)\in E) \subseteq M(T)_{\mathbb R}.$$
Here, $e_i$ and $e_j$ are the basis vectors that correspond to vertices $i$ and $j$. The affine normal toric variety associated to $G$ is then $\mathrm{Spec}(\mathbb{C}[\sigma_G \cap \mathrm{M}(T)])$. Toric matrix Schubert and toric KL varieties arise as these varieties (see e.g.\ Section~\ref{sec: MS varieties}, \ref{sec: KL varieties}). We can calculate the dimension of the edge cone via the following lemma, which will be used during the study of the complexity of the usual torus action.

\begin{lemma}[\cite{donten2021complexity}]\label{lem:weightcondim}
    Let $G$ be a DAG with $v$ vertices and $c$ connected components. The dimension of the edge cone $\sigma_G \subseteq M(T)_{\mathbb R}$ is given by $v-c$.
\end{lemma}

\section{Matrix Schubert varieties}\label{sec: MS varieties}

\noindent We now introduce some preliminaries on the class of varieties called matrix Schubert varieties. A large section of this paper focuses on this class of variety. Each matrix Schubert variety is specified by a permutation. We, thus, first recall some basic theory on permutations, before continuing on to define matrix Schubert varieties. 

Let $S_n$ be the symmetric group on $n$ elements. A \emph{permutation} $w$ is a bijective map $w: S_n\rightarrow S_n$. Denote by $t_{i,j}$, the \emph{transposition of the elements $i<j$} with $i,j\in[n]$ and by $T_n$, the set of transpositions on $S_n$. Each permutation in $S_n$ can be written as a product of some transpositions of adjacent elements $s_i := t_{i,i+1}$; the so-called \emph{simple reflections}. We also write a permutation in \emph{one-line notation}. That is, we express $w$ as $w(1)\cdots w(n)$. For readability, we sometimes also write $w$ as $(w(1), \ldots, w(n))$. Note that multiplying a transposition $t_{i,j}$ from the left switches the values $i$ and $j$, whereas multiplying it from the right switches the positions $i$ and $j$. Further, we can construct the permutation matrix that is associated to $w$, by placing $1$'s at $(w(i),i)$ for $i\in[n]$ and $0$'s everywhere else. We also call this matrix $w$, in an abuse of notation throughout this paper.
Let us consider an example to visualise the three representations.

\begin{example}
    Let $w=3241 \in S_4$, in one-line notation, or, equivalently, $w = s_1s_2s_1s_3$. The permutation matrix is given by
    $$w = \left(\begin{array}{cccc}
       0 & 0 & 0 & 1 \\
       0 & 1 & 0 & 0 \\
       1 & 0 & 0 & 0 \\
       0 & 0 & 1 & 0 \\       
    \end{array}\right).$$
\end{example}

\subsection{Preliminaries on matrix Schubert Varieties}\label{subsec: prelim mSV} 
We follow the convention for matrix Schubert varieties from \cite{donten2021complexity}. Let $B \subseteq GL_n$ be the group of upper triangular invertible matrices of size $n\times n$. Then, we define the following action of $B\times B$ on $\mathbb{C}^{n\times n}$, by:
\begin{align*}
    (B \times B) \times \mathbb{C}^{n\times n} &\rightarrow \mathbb{C}^{n\times n}\\
    ((X,Y),M) &\mapsto XMY^{-1}
\end{align*}
For some matrix $M \in \mathbb{C}^{n\times n}$ and $a, b \in [n]$, define $M_{[a,b]}$ to be the submatrix of $M$ formed of the rows $a, a+1, \ldots, n$ and columns $1, \ldots, b$; and $r_M(a,b)$ to be the rank of this matrix. Then, we note that the matrix $M$ is in the orbit $BwB$ if and only if $r_M(a,b) = r_w(a,b)$ for all $a, b \in [n]$. Let the $(i,j)$-th element of $M$ be given by $z_{ij}$.

\begin{definition}
    The \textit{matrix Schubert variety} associated to the permutation $w \in S_n$ is $\overline{X_w} \coloneqq \overline{BwB} \subset \mathbb C^{n \times n}$. The closure is taken with respect to the Zariski topology. The variety is defined in the polynomial ring $\mathbb{C}[z_{ij}|i,j\in[n]]$. 
\end{definition}

\noindent While this seems quite abstract , in practice, we can describe these varieties in terms of conditions derived from certain combinatorial structures.

\begin{definition}
    Let $w \in S_n.$
    \begin{itemize}
        \item The \textit{opposite Rothe Diagram} of $w$ is defined as $D^{\circ}(w) \coloneqq \{(i, j) \ | \  w(j)<i, w^{-1}(i)>j\}$.
        \item The \textit{essential set} of $w$ is the set of north-east corners of each connected component of $D^{\circ}(w)$, and is denoted by $\Ess(w)$. 
    \end{itemize}
\end{definition}

\noindent One can draw the opposite Rothe diagram by considering the matrix $w$ as an $n \times n$ grid with $1$s in some positions and ruling out every square that is to the north or to the east of a $1$. The opposite Rothe diagram is then the set of boxes that is not north or east of a $1$. Note that in our work, we follow a certain convention for the construction of matrix Schubert varieties, which is specified by the orientation of opposite Rothe diagrams. The convention used differs throughout the literature. We use opposite Rothe diagrams, instead of Rothe diagrams, which correspond to a $B_- \times B$ action.  Here, $B_-$ denotes the group of lower triangular ${n \times n}$ matrices. 
\begin{example} 
    Let us consider $w=45231$. The following diagram represents $D^\circ(w)$. In this case, $\Ess(w)$ is all of $\Dc(w)$, as each connected component of $\Dc(w)$ is given by one element. \\
\begin{center}
\begin{tikzpicture}[scale = 0.5]

  \draw[step=1.0,black,thin] (0,0) grid (5,5);

   %ones:
    \node at (0.5,1.5) {{$1$}}; 
   \node at (1.5,0.5) {{$1$}}; 
   \node at (2.5,3.5) {{$1$}}; 
   \node at (3.5,2.5) {{$1$}};
   \node at (4.5,4.5) {{$1$}};

    %Rothe diagram
  \draw [line width = 0.0mm, draw=CornflowerBlue, fill=CornflowerBlue, draw opacity = 0.5, fill opacity=0.5]
       (0,0) -- (0,1) -- (1,1) -- (1,0)  --cycle;
   \draw [line width = 0.0mm, draw=CornflowerBlue, fill=CornflowerBlue, draw opacity = 0.5, fill opacity=0.5]
       (2,2) -- (2,3) -- (3,3) -- (3,2)  --cycle;

  %lines to cross out:
  %up
  \draw[line width = 0.7mm, draw=Gray, draw opacity = 0.4]
  (0.5, 2) -- (0.5,5);
  \draw[line width = 0.7mm, draw=Gray, draw opacity = 0.4]
  (1.5, 1) -- (1.5,5);
  \draw[line width = 0.7mm, draw=Gray, draw opacity = 0.4]
  (2.5, 4) -- (2.5,5);
  \draw[line width = 0.7mm, draw=Gray, draw opacity = 0.4]
  (3.5, 3) -- (3.5,5);
  \draw[line width = 0.7mm, draw=Gray, draw opacity = 0.4]
  (4.5, 5) -- (4.5,5);

  %right
  \draw[line width = 0.7mm, draw=Gray, draw opacity = 0.4]
  (2, 0.5) -- (5,0.5);
  \draw[line width = 0.7mm, draw=Gray, draw opacity = 0.4]
  (1, 1.5) -- (5,1.5);
  \draw[line width = 0.7mm, draw=Gray, draw opacity = 0.4]
  (4, 2.5) -- (5,2.5);
  \draw[line width = 0.7mm, draw=Gray, draw opacity = 0.4]
  (3,3.5 ) -- (5,3.5);
  \draw[line width = 0.7mm, draw=Gray, draw opacity = 0.4]
  (5, 4.5) -- (5,4.5);
\end{tikzpicture}
\end{center}

\end{example}
\noindent Now, in practice, we use the following theorem as the definition for generators of the matrix Schubert ideal.
\begin{theorem}[{\cite[Proposition 3.3, Lemma 3.10]{fulton1992flags}}]\label{thm:fultonessential}
    The matrix Schubert Variety $\overline{X_w} \subset \mathbb C^{n \times n}$ is defined as a scheme by the determinants that encode the inequalities $r_M(a,b) \leq r_w(a,b)$ for each $(a,b) \in \Ess(w)$. The defining determinantal conditions are called the \emph{Fulton conditions}.    
    Its dimension is $\dim \overline{X_w} = n^2 - \Dc(w)$.
\end{theorem}
\noindent We illustrate the above theorem via the example introduced previously.
\begin{example}
    Continue with $w$ as $45231$. Then, the essential set is given by $\{(3,3),(5,1)\}$. Thus, Fulton's determinantal conditions look as follows:
    \begin{itemize}
        \item $r_M(5,1) \leq r_w(5,1) = 0$. Thus, $z_{51} =0$.
        \item $r_M(3,3) \leq r_w(3,3) = 2$. This tells us that the determinant of $M_{[3,3]}$ is zero.
    \end{itemize}
    We can write down the ideal associated to the $23$-dimensional matrix Schubert variety $\Xw$ as
    $$\left\langle z_{51}, \left|M_{[3,3]}\right| \right\rangle = 
    \left\langle z_{51}, \left|\begin{array}{ccc}
        z_{31} & z_{32} & z_{33} \\
        z_{41} & z_{42} & z_{43} \\
        z_{51} & z_{52} & z_{53} \\
    \end{array}\right| \right\rangle.$$
\end{example}

\subsection{The usual torus action} In order to interpret matrix Schubert varieties as T-varieties, we define the relevant torus action on the non-trivial part of $\overline{X_w}$. 
Given $w \in S_n$, we can find the affine variety $Y_w$ such that $\overline{X_w} = Y_w \times \mathbb{C}^d$, for as large an integer $d$ as possible. First, we introduce certain combinatorial structures on the opposite Rothe diagram. 

\begin{definition}
    The \textit{dominant piece} of $w$, denoted by $\text{dom}(w)$, is the connected component of $(n,1)$ in the opposite Rothe diagram of $w$. We further define $\text{SW}(w)$ to be the set of $(i,j)$ in $D^{\circ}(w)$ that are south-west of an entry in $\text{Ess}(w)$. Finally, we let $L(w) \coloneqq \text{SW}(w) \setminus \text{dom}(w)$ and $L'(w) \coloneqq \SW(w)\setminus D^{\circ}(w)$. 
\end{definition}

\begin{example}
    \noindent For $w=45231$, the following diagrams depict $\SW(w)$, $L(w)$ and $L'(w)$, from left to right.

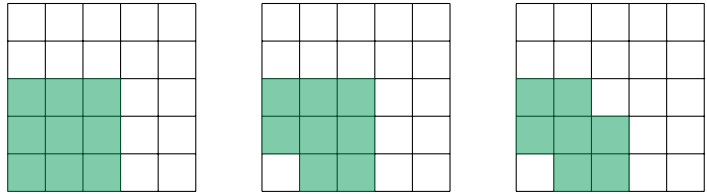
\begin{center}
   \begin{tikzpicture}[scale = 0.5]

  \draw[step=1.0,black,thin] (0,0) grid (5,5);

   \draw [line width = 0.0mm, draw=ForestGreen, fill=ForestGreen, draw opacity = 0.5, fill opacity=0.5]
       (0,0) -- (0,3) -- (3,3) -- (3,0)  --cycle;
\end{tikzpicture}
\qquad
\begin{tikzpicture}[scale = 0.5]

  \draw[step=1.0,black,thin] (0,0) grid (5,5);

   \draw [line width = 0.0mm, draw=ForestGreen, fill=ForestGreen, draw opacity = 0.5, fill opacity=0.5]
       (1,1) -- (0,1) -- (0,3) -- (3,3) -- (3,0)  -- (1,0) -- cycle;
\end{tikzpicture}
\qquad
\begin{tikzpicture}[scale = 0.5]

  \draw[step=1.0,black,thin] (0,0) grid (5,5);

   \draw [line width = 0.0mm, draw=ForestGreen, fill=ForestGreen, draw opacity = 0.5, fill opacity=0.5]
       (1,1) -- (0,1) -- (0,3) -- (2,3) -- (2,2) -- (3,2) -- (3,0)  -- (1,0) -- cycle;
\end{tikzpicture}
\end{center}
    \end{example}

\noindent Now, there are no conditions of $\overline{X_w}$ on a matrix $M$ outside of $\SW(w)$ and hence, $Y_w$ is isomorphic to the subvariety of $\overline{X_w}$ that is formed by setting $z_{ij}$ to $0$, for any $(i,j) \notin \text{SW}(w)$. The $B\times B$-action on $\overline{X_w}$ ``fixes'' this subvariety, and hence induces an action on $Y_w$. The restriction of the $B\times B$-action on $Y_w$ to $T\times T$ is known as the \textit{usual torus action}.

\noindent
Further, we note that $\mathbb{C}^{n^2-|\SW(w)|}$ is the projection of the matrix Schubert variety $\Xw$ onto the free boxes, associated to the entries of the $n\times n$ grid that are not in $\SW(w)$. Similarly, we can define $Y_w$ as the projection onto the boxes of $L(w)$. 
Note, that $(a,b) \in \dom(w)$ if and only if $r_w(a,b)=0$. Thus, $\Xw = Y_w \times \mathbb{C}^{n^2-|\SW(w)|}$ and, therefore,
$$\text{dim}(Y_w) = n^2 - |\Dc(w)| - (n^2 - |\SW(w)|) = |\SW(w)| - |\Dc(w)| = |L'(w)|.$$

\begin{example}
    Continuing with $w=45231$, we note that $\Xw = Y_w \times \mathbb{C}^{16}$ and $\text{dim}(Y_w) = 7$. The defining ideal of $Y_w$ is
    $$\left\langle \left|\begin{array}{ccc}
        z_{31} & z_{32} & z_{33} \\
        z_{41} & z_{42} & z_{43} \\
        0 & z_{52} & z_{53} \\
    \end{array}\right|  \right\rangle  \subset \mathbb{C}[z_{31}, z_{32}, z_{33}, z_{41}, z_{42}, z_{43}, z_{52}, z_{53}].$$
\end{example}

 \noindent In order to define the complexity of the $T \times T$ torus action, we consider the associated weight cone. We recall the definition, as stated in \cite{donten2021complexity}; we refer the reader to this paper for further details of relevant theory. The weight cone is given by 
$$\sigma_w = \mathrm{Cone}(e_i-f_j|(i,j)\in L(w)) \subseteq M(T \times T) \otimes_{\mathbb Z}{\mathbb R},$$
where $e_1, \ldots, e_n,f_1, \ldots, f_n$ denote the standard basis for $\mathbb Z^n \times \mathbb Z^n$.
Furthermore, we can calculate the generators and the dimension of this cone very easily via the edges of a bipartite graph. See \cite{portakal2021rigidity} for more details. The bipartite directed acyclic graph $\Gw$ is defined with $E(\Gw) = \{(a\rightarrow b^* | (a,b)\in L(w)\}$ and $V(\Gw) \subset [n] \cup[n]$ such that there are no isolated vertices. Furthermore, we immediately see that the weight cone of $Y_w$ is the edge cone of $\Gw$. For the sake of simplicity, we abuse notation and say that $\sigma_w$ is the edge cone of $Y_w$. In particular, we have that the $T\times T$-action on $Y_w$ is of complexity $d$ if and only if
$$\mathrm{dim}(\sigma_w) = \mathrm{dim}(Y_w)-d = |L'(w)|-d.$$

\begin{example}
    Again, consider the matrix Schubert variety $\Xw$ for $w=45231$. We can construct the graph $\Gw$, via the procedure described above. Then, the dimension of the edge cone $\sigma_w$ is $6-1=5$. As noted above, the dimension of $Y_w$ is $7$. Thus, $Y_{45231}$ is a $T\times T$-variety of complexity $2$.
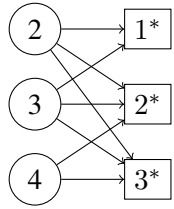
\begin{center}
\begin{tikzpicture}
    \node[shape=circle,draw=black] (2) at (0,3) {2};
    \node[shape=circle,draw=black] (3) at (0,2) {3};
    \node[shape=circle,draw=black] (4) at (0,1) {4};

    \node[shape=rectangle,draw=black] (1b) at (1.5,3) {$1^*$};
    \node[shape=rectangle,draw=black] (2b) at (1.5,2) {$2^*$};
    \node[shape=rectangle,draw=black] (3b) at (1.5,1) {$3^*$};

    \draw[->] (2) -- (1b);
    \draw[->] (2) -- (2b);
    \draw[->] (2) -- (3b);
    \draw[->] (3) -- (1b);
    \draw[->] (3) -- (2b);
    \draw[->] (3) -- (3b);
    \draw[->] (4) -- (2b);
    \draw[->] (4) -- (3b);
\end{tikzpicture}
\end{center}
    
\end{example}

  \noindent In particular, if $Y_w$ has complexity $0$ under the usual torus action, then $Y_w$ is \emph{toric}. We can also characterise the $Y_w$ that are toric with respect to this action,  \cite{escobar2016toric}, via the following theorem:
\begin{theorem}[{\cite[Theorem 3.4]{escobar2016toric}}]\label{thm:toric_desc}
    $Y_w$ is a toric variety with respect to the usual torus action if and only if $L'(w)$ consists of disjoint hooks not sharing a row or a column.
\end{theorem}
 \noindent Here, a \textit{hook} with a corner $(i,j)$ consists of elements $(i',j')$, such that $j=j'$ and $i'< i$, or $i=i'$ and $j < j'$.

\subsection{(Toric) Matrix Schubert varieties under multiplication with simple reflections}\label{subsec: toric ms} 
\noindent In this subsection, we investigate how a toric matrix Schubert variety $\overline{X_w}$ changes when $w$ is right-multiplied by a simple reflection, $s_M$. Recall that $w\cdot s_M$ is the permutation with $w\cdot s_M (M) = w(M+1)$ and $w\cdot s_M (M+1) = w(M)$ and identical to $w$ everywhere else. In terms of the Bruhat order as defined in Section \ref{subsec: Bruhat order}, this operation changes $w$ by a cover relation: either increasing or decreasing the number of inversions by one. Combinatorially, this step is visualised by a local modification of the opposite Rothe diagram; the set of boxes (inversion positions) changes exactly where the swap occurs. 
In contrast to previous work, we examine the opposite Rothe diagram for the toric case in detail. 
We track how it changes under right-multiplication by a simple reflection and this allows us to observe explicitly how the generators of the edge cone associated to a toric matrix Schubert variety are affected. 

For this, we need to focus on the case where $L'(w)$ is formed of disjoint hooks, as specified by Theorem~\ref{thm:toric_desc} or equivalently where $w$ avoids the patterns 4312 and 3412, as stated in Theorem~\ref{thm: pattern avoidance}. 
Thus, we first take a closer look at the combinatorial structures of $w$ and $D^\circ(w)$ that allow 
$L'(w)$ to be formed of hooks.

\begin{definition}
    A \emph{staircase} of dimension $i \times j$, in an $n\times n$ lattice, is a Young diagram in French notation of height $i$ and width $j$. 
\end{definition}
\noindent Staircases play a key role for the structure of $\Dc(w)$ in the toric setting. We now detail this structure, and the logic that we use to arrive at it, below. 
First, let $L'(w)$ be formed of one hook, for $w \in S_n$. Let us say that the hook has height $i$ and width $j$. 
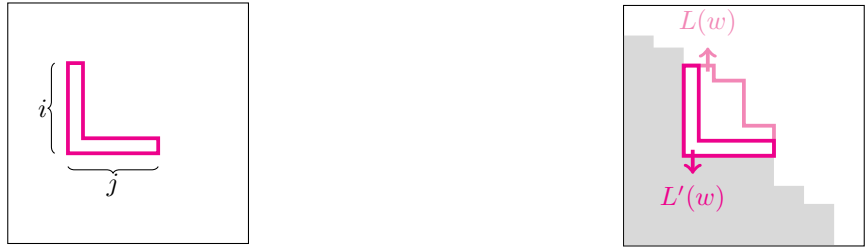
\begin{figure}[H]
\begin{minipage}{0.49\textwidth}
\centering
    \begin{tikzpicture}[scale = 0.8]
    \draw(0,0) rectangle (4,4);
    \draw[line width = 0.5mm, draw=magenta] (1.25,1.75)--(1.25,3)--(1,3)--(1,1.5)--(2.5,1.5)--(2.5,1.75)--cycle;
     \draw[decoration={brace,mirror,raise=5pt},decorate]
  (1,1.5) -- node[below=4pt] {\small $j$} (2.5,1.5);
  \draw[decoration={brace,raise=5pt},decorate]
  (1,1.5) -- node[left=4pt] {\small $i$} (1,3);
\end{tikzpicture}
\end{minipage}
\begin{minipage}{0.49\textwidth}
\centering
\begin{tikzpicture}[scale = 0.8]
    \draw(0,0) rectangle (4,4);
    \fill [gray, fill opacity = 0.3] (0,0) -- (3.5,0) -- (3.5,0.7) -- (3,0.7) -- (3,1) -- (2.5,1) -- (2.5,1.5) -- (1,1.5) -- (1,3.3) -- (0.5,3.3) -- (0.5,3.5) -- (0,3.5) -- cycle;
    \draw[line width = 0.5mm, draw=magenta!60] (1,1.5) -- (1,3) -- (1.5,3) -- (1.5,2.75) -- (2,2.75) -- (2,2)--(2.5,2) -- (2.5,1.5) -- (1,1.5) -- cycle;
    \draw[line width = 0.5mm, draw=magenta] (1.25,1.75)--(1.25,3)--(1,3)--(1,1.5)--(2.5,1.5)--(2.5,1.75)--cycle;
    \draw[->,line width = 0.5mm,magenta!60] (1.4,2.9)--(1.4,3.3) node[above] {\small $L(w)$};
    \draw[->,line width = 0.5mm,magenta] (1.15,1.6)--(1.15,1.2) node[below] {\small $L'(w)$};
\end{tikzpicture}
\end{minipage}
\caption{The figure on the left depicts the hook in an $n\times n$ lattice. The figure on the right illustrates the shapes of $L(w), L'(w)$ and the dominant piece, which is drawn by the grey area.}
\end{figure}

\begin{proposition}\label{prop: one hook}
    Let $L'(w)$ consist of one hook, for $w \in S_n$, of height $i$ and width $j$. The southwest of $L'(w)$ is contained in $\dom(w)$. We further note that the shape of the dominant piece of $L'(w)$ is a staircase, as all connected components of the Rothe diagram are staircases. 
\end{proposition}

\begin{proof}
Recall that $L'(w) \subseteq L(w) \subseteq \SW(w)$, and $L'(w) \cap D^{\circ}(w)=\emptyset$. Hence, everything to the southwest of the hook must be in the opposite Rothe diagram, and is therefore one connected component. In other words, everything to the southwest of the hook must be in $\text{dom}(w)$.  
\end{proof}

\noindent We can then also specify what the structure of $\Dc(w)$ needs to look like when $L'(w)$ is formed of multiple hooks, with no shared row or column. 
Here, we denote the $k$ components of $L'(w)$ by $L_i'(w), i\in[k]$, in order from west to east and similarly denote the components of $L(w)$. Furthermore, the column of the hook in the $i$th connected component is written as $h_i$; the set of columns of $L_i\setminus L_i'(w)$ is given by $\alpha[i]$; and the set of columns before, between, and after the hooks are given by $\beta[i],i\in[k+1]$, as can be seen below. Again, we see that each $L_i(w)$, as well as each $\alpha[i]$ and $\beta[i]$ is a staircase.
When we wish to generically refer to any such staircase, we will refer to this staircase as $s$. The $i$th ``step" of this staircase is given by $s^i$, and the $j$th column of this step is $s^i_j$. Also, note that we call the easternmost step of a staircase $s^{\mathrm{last}}$ and
the easternmost column of a step $s^i_\mathrm{last}$. 

We can associate to each step $s^i$ its width, $\mathrm{width}(s^i)$. Also, to every step $s^i$ that is not the first step of it's staircase, we can associate
its height, which is the difference between the northernmost row of $s^i$ and the northernmost row of $s^{i-1}$.
With the structure of the Rothe diagram in the case that $L'(w)$ consists of disjoint hooks, we may extend the notion of height to first steps in the following way. Here, $(a,b)$ denotes the northwest-most element of $s$.
$$\mathrm{height}(s^i) = \begin{cases}
    a-1 & \textrm{if } s^i = \beta[1]^1 \textrm{ or } s^i = \alpha[1]^1 \textrm{ if } \beta[1] \textrm{ does not exist}\\
    0 & \textrm{if } s^i = \alpha[j]^1, \,\,\, j \neq 1 \textrm{ and } \beta[j-1] \textrm{ does not exist}\\
    \# \{ (c,b-2) \in \Dc(w) \mid c>a \} & \textrm{if } s^i = \alpha[j]^1 \textrm{ and } \beta[j-1]
   \textrm{ exists} \\
   \end{cases}
$$

\noindent We say that the $1$ in column $s^i_k$ is \emph{on the step}, if it lies exactly $k$ positions above it. This is based on the fact that if there is space, the $1$ in a column $s^i_{k+1}$ will always lie one position higher than the $1$ in column $s^i_k$. However, there may not always be space: 
if $\mathrm{width}(s^i) > \mathrm{height}(s^i)$, the $1$s in columns $s^i_k$ with $k > \mathrm{height}(s^i)$ will not lie on the step, but further north of it.

\begin{figure}[H]
\begin{minipage}{0.49\textwidth}
    \centering
    \hspace{1.5cm}
    \begin{tikzpicture}[scale=0.50]
    \draw(0,0) rectangle (8,8);
    
    \draw[line width = 0.5mm, draw=magenta!60] (1,5.5) -- (1,7) -- (1.5,7) -- (1.5,6.75) -- (2,6.75) -- (2,6) -- (2.5,6) -- (2.5,5.5) -- (1,5.5);
    \draw[line width = 0.5mm, draw=magenta] (1.25,5.75)--(1.25,7) -- (1,7)--(1,5.5) -- (2.5,5.5)--(2.5,5.75) -- cycle;
    \draw[line width = 0.5mm, draw=magenta!60] (5.5,4)--(5.5,4.25)--(4.5,4.25)--(4.5,4.5)--(4,4.5)--(4,4.75)--(3.75,4.75);
    \draw[line width = 0.5mm, draw=magenta] (3.5,3.75)--(3.5,4.75)--(3.75,4.75)--(3.75,4)--(5.5,4)--(5.5,3.75)--cycle;
  
    \draw[line width = 0.5mm, draw=magenta!60] (7,2)--(7,2.25)--(6.75,2.25)--(6.75,2.5)--(6.25,2.5);
    \draw[line width = 0.5mm, draw=magenta] (6,1.75)--(6,2.5)--(6,2.5)--(6.25,2.5)--(6.25,2)--(7,2)--(7,1.75)--cycle;

    \fill [gray, fill opacity = 0.3] (0,7.5) -- (0.5,7.5) -- (0.75,7.5) -- (0.75,7) -- (1,7) -- (1,5.5) -- (2.5,5.5) -- (2.5,5) -- (3.25,5) -- (3.25,4.75) -- (3.5,4.75) -- (3.5,3.75) -- (5.5,3.75) -- (5.5,3) -- (5.75,3) -- (5.75,2.5) -- (6,2.5) -- (6,1.75) -- (7,1.75) -- (7,0) -- (0,0) -- cycle;

    \node at (7.5,1.75) {$\ddots$};
   
    \draw[thick,decoration={brace,raise=2pt},decorate] (0,8) -- node[above=4pt] {$\beta[1]$} (1,8);
    \draw[thick,decoration={brace,raise=2pt},decorate] (2.5,8) -- node[above=4pt] {$\beta[2]$} (3.5,8);
    \draw[thick,decoration={brace,raise=2pt},decorate] (5.5,8) -- node[above=4pt] {$\beta[3]$} (6,8);

    \draw[thick,decoration={brace,mirror,raise=2pt},decorate] (1.5,0) -- node[below=4pt] {$\alpha[1]$} (2.5,0);
    \draw[thick,decoration={brace,mirror,raise=2pt},decorate] (3.75,0) -- node[below=4pt] {$\alpha[2]$} (5.5,0);
    \draw[thick,decoration={brace,mirror,raise=2pt},decorate] (6.25,0) -- node[below=4pt] {$\alpha[3]$} (7,0);
    \end{tikzpicture}
\end{minipage}
\begin{minipage}{0.49\textwidth}
    \centering
    \hspace{-3cm}
    \begin{tikzpicture}[scale = 0.5]
  
  \draw [line width = 0.5mm]
       (0,3) -- (0,1) -- (1,1) -- (1,-1) -- (4,-1);
  \node[magenta] at (3.5,2.5) {\small $1$};
  \node[blue] at (0.5,1.5) {\small $1$};
  \node[blue] at (1.5,-0.5) {\small $1$};
  \node[blue] at (2.5,0.5) {\small $1$};
\end{tikzpicture}
\end{minipage}
\caption{The image on the left depicts the form of the opposite Rothe diagram in the case of multiple disjoint hooks. 
The grey area is the dominant piece. 
In the diagram on the right, we have two steps $s^i$ and $s^{i+1}$. The ones in blue lie on the step, since their position on the step is lower than the height of the step. The one in pink lies north of the step, since that is not the case for it.}
\end{figure}
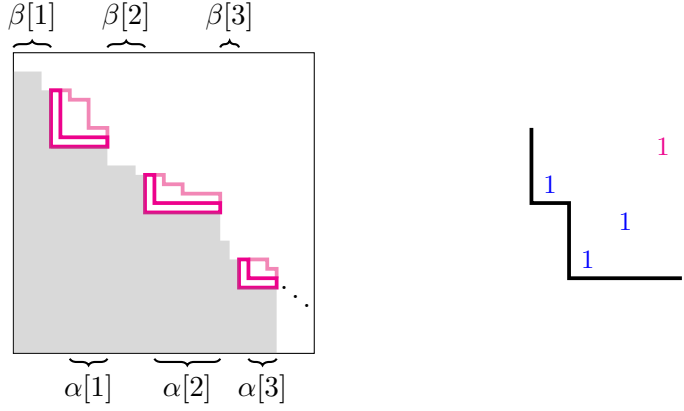

\begin{proposition}
    Let $L'(w)$ consist of $k$ hooks, for $w \in S_n$. The southwest of each connected component of $L'(w)$ is contained in $\dom(w)$. Moreover, the shape of the columns of the dominant piece to the west of $L'_1(w)$; in between $L'_i(w)$ and $L'_{i+1}(w)$, for $i\in[k-1]$; and to the east of $L'_k(w)$, is a staircase. 
\end{proposition}

\begin{proof}
    This follows as an analogue of the proof of Proposition \ref{prop: one hook}.
\end{proof}

 \noindent Another theorem that will prove crucial in our study of toric matrix Schubert varieties was given in \cite{stelzer2025matrix} and is given as follows. Note that in the paper, Stelzer uses a different convention of the Rothe diagram to ours. Thus, we quote the theorem after rewriting the patterns in accordance with our convention. 

\begin{theorem}[{\cite[Theorem 1.6]{stelzer2025matrix}}]\label{thm: pattern avoidance}
    The matrix Schubert ideal $\Xw$ is binomial if and only if $w$ avoids the patterns $4312$ and $3412$. Furthermore, the only toric matrix Schubert varieties are those which are toric under the usual torus action, as in Theorem~\ref{thm:toric_desc}.
\end{theorem}
\noindent Thus, we may equivalently use this criterion to determine when a matrix Schubert variety is toric, or not. Using both of these criteria in its proof, we can then state the following theorem. 

\begin{theorem}\label{thm: all toric mS}
     Let $w\in S_n$ be a permutation such that $Y_w$ is toric. 

     For $M = \alpha[j]^i_k$, the variety ${Y_{w\cdot s_M}}$ is not toric if and only if one of the following holds.

     \begin{enumerate}
         \item $i=1$, $k=1$ and $ \mathrm{width}(\alpha[j]^1) >1$, $\mathrm{height}(\alpha[j]^\mathrm{1}) =0$;
         \item 
         $k = \mathrm{last}$, $i \neq \mathrm{last}$ and $\mathrm{height}(\mathrm{\alpha[j]^{i+1}})\geq2$;
         \item $i,k = \mathrm{last}$ and 
         \begin{enumerate}
            \item $\mathrm{height}(\beta[j+1]^1)=0$ and $\mathrm{width}(\alpha[j]^{\mathrm{last}}) > \min \{ 1, \mathrm{height}(\alpha[j]^{\mathrm{last}}) \}$, or
            \item $\mathrm{height}(\beta[j+1]^1)>0$, or
            \item $\beta[j+1]$ does not exist;
         \end{enumerate}
         \item $1 <k \neq \mathrm{last}$.
     \end{enumerate}

     For $M = \beta[j]^i_k$, the variety ${Y_{w\cdot s_M}}$ is not toric if and only if the following holds.

     \begin{enumerate}
         \setcounter{enumi}{4}
         
         \item $\min \{ \mathrm{height}(\beta[j]^i), 2 \} <k \neq \mathrm{last}$.
         
     \end{enumerate}

     The variety is never not toric if $M = h_j$.

\begin{proof}
    
    We consider all possibilities for $M$, and the structure of $w$ that render $w\cdot s_M$ non-toric. That is, we want to find all combinations of $w$ and $M$ such that $w\cdot s_M$ contains one of $3412$ or $4312$. In other words, we need $w\cdot s_M$ to contain the sub-sequence, in one-line notation, of the form $a\ldots b\ldots c \ldots d$, such that $a,b>d$ and $d>c$. Then, note, that $a$ and $b$ have the same roles in the permutation, and thus swapping them cannot add or remove the patterns that we are interested in.
    We have two cases, relating to ascents and descents of $w$.\\

    \noindent \textbf{1.} $w$ has an ascent at $M$. That is, $w(M) < w(M+1)$. Thus, $w\cdot s_M(M) > w\cdot s_M(M+1)$. 
    
    If the new permutation contains either of the patterns $3412$ or $4312$, then, $w\cdot s_M(M)$ must take the role of $b$ and $w\cdot s_M(M+1)$ takes that of $c$. Thus, $w$ contains $a\ldots bc \ldots d$, with $b,a>d>c$. That is, it contains the pattern $3142$, or $4132$ where the $1$ and $4$, or respectively $1$ and $3$, are in neighbouring columns. 
    
    \noindent \textbf{Case I}: $w$ contains the pattern $3142$. Thus, $w$ contains the subword $a\ldots bc \ldots d$, such that $b<d<a<c$. The $14$ patterns in neighbouring columns can only arise in the following ways:

    \begin{itemize}
    \item The $b$ is associated to the last column of some step $s^i$ and the $c$ is associated to the first column of the next step. Then all rows in between the top of $s^i$ and the row associated to $b$ have $1$s in them that lie to the west of $b$. Hence, say we find the $314$ pattern in $abc$, then the $a$ is one of these $1$s
    and the $d$ to the east of it cannot lie between $a$ and $b$. Therefore, the pattern cannot occur in this case. 
    
    \item The $b$ is associated to the last column of some step $\beta[j]$ and the $c$ is associated to the column of a hook corner. 
    For the same reason as above, $a$ and $d$ cannot be found in this case.

    \item The $b$ is associated to the last column of some hook staircase $\alpha[j]$ and the $c$ is associated to the first column of $\beta[j+1]$ or $h_{j+1}$. This is only possible if $\beta[j+1]^1$ has positive height or if the hooks lie right next to each other. We can choose as $a$ the hook corner $h_j$ and as $d$ the $1$ that lies in the row exactly one above the row of the hook. Hence, $M= \alpha[j]_{\textrm{last}}^{\textrm{last}}$.
    
    \end{itemize}
   \noindent \textbf{Case II}: $w$ contains the pattern $4132$. Thus, $w$ contains the subword $a\ldots bc \ldots d$, such that $b<d<c<a.$ The only way for the $13$ pattern to occur as $bc$, in neighbouring columns, is for the $b$ to correspond to the last column of a step $s^i$, and the $c$ to the first column of the next step or the corner of a hook. However, the $43$ pattern cannot occur as $ac$ if the $a$-column is simply a column in a step that comes to the left of $s^{i+1}$ or the hook. We specifically need the $c$ to be in the first column of $s^{i+1}$ and the $a$ to be in a column $h_j$; that is, in the column of a hook. Then, the $b$ corresponds to $\mathrm{\alpha[j]}^i_{\mathrm{last}}$ and $c$ to $\mathrm{\alpha[j]}^{i+1}_1$. The $d$ must be associated to a column that is somewhere on step $\mathrm{\alpha[j]}^{i+1}$; with the $1$ in this corresponding column to be in between the $1$s in columns $\mathrm{\alpha[j]}^i_{\mathrm{last}}$ and $\mathrm{\alpha[j]}^{i+1}_1$. This is only guaranteed if $\mathrm{height}(\mathrm{\alpha[j]^{i+1}})\geq 2$. Hence, $M=\mathrm{\alpha[j]}^i_{\mathrm{last}}$.\\

    \noindent \textbf{2.} $w$ has an descent at $M$. That is, $w(M) > w(M+1)$. In other words, $w\cdot s_M(M) < w\cdot s_M(M+1)$. Then, if $w\cdot s_M$ contains either $3412$ or $4312$, we see that it must contain, as a subword, $a\ldots b\ldots cd$, such that $c<d<a,b$; the columns representing $c$ and $d$ are adjacent to one another. 

    \noindent \textbf{Case I}: $w$ contains the pattern $3421$. This occurs in the case that $w\cdot s_M$ contains $3412$. Then, $w$ contains a subsequence $a\ldots b\ldots cd$, with $d<c<a<b$. This pattern is found in the following ways:
    \begin{itemize}
    \item
    The columns corresponding to $c$ and $d$ are on the same step, say step $s^i$. For now assume that we can also find the 
     columns corresponding to $a$ and $b$, such that they are in the pattern $34$.
    Then, we just need to ensure that the $1$s in the $c$ and $d$-columns are to the north of that in the $a$-column. This happens if and only if these two $1$s are not on the step $s^i$, but north of it. This, however is guaranteed if and only if $M=s_k^i$ where $\mathrm{height}(s^i) < k < \mathrm{width}(s^i)$. As long as $M \neq \beta[j]^1$, one can see that $a$ and $b$ as we need them actually exist: If $M = \alpha[j]^1_k$, then we may choose $b$ in the hook column and $a$ in the first column of the previous step. Note that $\alpha[j]^1$ cannot be the westernmost step, since otherwise the height of $\alpha[j]^1$ is always greater or equal than its width. If $M = s^i_k$ is any other step, then we may choose $b$ in column $s^i_1$ and $a$ in the first column of the previous step.
    If $M=s^i_k =\beta[j]^1$, we show in another subcase below that the condition $\mathrm{height}(s^i) < k < \mathrm{width}(s^i)$ is sufficient.
    
    \item
    The $c$ is in the easternmost column $\alpha[j]^{\text{last}}_{\text{last}}$ of a staircase and the $d$ on the first step of the dominant piece after the hook, but not on the step. This occurs only when $\mathrm{height}(\beta[j+1]^1)=0.$
    In order to find $a$ and $b$, we also require the $c$ to lie above the step. This occurs when $\mathrm{width}(\alpha[j]^{\text{last}}) > \mathrm{height}(\alpha[j]^{\text{last}})$. In this case, $M = \alpha[j]^{\text{last}}_{\text{last}}$.

    \end{itemize} 

    \noindent \textbf{Case II}: $w$ contains the pattern $4321$. This occurs in the case that $w\cdot s_M$ contains $4312$. Then, $w$ contains a subsequence $a\ldots b\ldots cd$, with $d<c<b<a$. This pattern is found in the following ways:
    \begin{itemize}
        \item All four columns belong to the same step $s^i$ and then we have that $M=s^i_k, k\geq 3$, and $s^i_k$ is not the last column of $s^i$. 
        \item The columns corresponding to $c$ and the $d$ are columns in $s^i$ but not on the step, but are north of it. In a previous subcase we have shown that this is the case if $M= s^i_k$ and $\mathrm{height}(s^i) < k < \mathrm{width}(s^i)$ and that as long as $M \neq \beta[j]^1$, this condition alone already guarantees non-toricness. Thus, now consider the leftover case $M = \beta[j]^1$. Then $\mathrm{width}j<(\beta[j]^1)$ and $\mathrm{height}(\beta[j]^1)=0$. To obtain the $432$ pattern, we can choose $a$ in the hook column and $b$ in column $\alpha[j-1]^\mathrm{last}_1$. Hence, overall, we obtain the case $M= s^i_k$ and $\mathrm{height}(s^i) < k < \mathrm{width}(s^i)$.
        \item The column corresponding to the $a$ is the hook column $h_j$; the $b$ is associated to a column of a step in $\alpha[j]$, namely $\alpha[j]^i_k$, for some $i,k$ and $c,d$ correspond to columns that are either north of the steps in the hook structure; or are given by $\alpha[j]^i_\ell, \alpha[j]^i_{\ell+1}$, for $\ell>k$. The former case is covered by the previous bullet point. Thus, another option is given by $M=\alpha[j]^i_\ell$, such that $\text{width}(\alpha[j]^i) > l >1$. If $\alpha[j]^i$ is the easternmost step of $\alpha[j]$, then we can also have that $M=\alpha[j]^{\text{last}}_{\text{last}}$ and the variety is non-toric if $\mathrm{height}(\beta[j+1]^1)=0$ and the width of $\alpha[j]^\text{last}$ is $>1$. \qedhere
    \end{itemize}
\end{proof}
\end{theorem}

From the above theorem, one can directly deduce for which instances of $w$ and $M$ the variety $Y_{w \cdot s_M}$ is toric. We list them in the corollary below. 

\begin{corollary}\label{cor: toric cases}
Let $w\in S_n$ be a permutation such that $Y_w$ is a toric variety. 

For $M= \alpha[j]^i_k$, the variety $Y_{w\cdot s_M}$ is toric if and only if one of the following holds.
\begin{enumerate}
    \item $i,k=1$ and $\mathrm{width}(\alpha[j]^1)>1$,
        $\mathrm{height}(\alpha[j]^1)>0$; 
    \item $1 = k \neq \mathrm{last}$, $i \neq 1$;
    \item $k = \mathrm{last}, i \neq \mathrm{last}$ and $\mathrm{height}(\alpha[j]^{i+1})=1$;
    \item $i,k= \mathrm{last}$ and
    {$\mathrm{width}(\alpha[j]^{\mathrm{last}}) \leq \min \{1, \mathrm{height}(\alpha[j]^{\mathrm{last}}) \}$}
    , $\mathrm{height}(\beta[j+1]^1)=0$.
\end{enumerate}
For $M= \beta[j]^i_k$, the variety $Y_{w\cdot s_M}$ is toric if and only if one of the following holds.
\begin{enumerate}
    \setcounter{enumi}{4}
    \item $k=1$ and $\mathrm{height}(\beta[j]^i)>0$;
    \item $k=2$ and $\mathrm{height}(\beta[j]^i)>1$;
    \item $k = \mathrm{last}$.
\end{enumerate}

The variety $Y_{w\cdot s_M}$ is also toric if

\begin{enumerate}
    \setcounter{enumi}{7}
    \item $M=h_j$.
\end{enumerate}

\end{corollary}

Even if the variety remains toric in these cases, the opposite Rothe diagram changes, which may cause the graph $G_w$ and the weight cone to change as well. We list below the effects which the different cases from \ref{cor: toric cases} have on $L(w)$, and therefore the associated graph and the weight cone. Here, a hook becoming narrower means that it could potentially disappear.

\begin{table}[H]
    \centering
    \begin{tabular}{ |l|l| }
    Changes in $\Dc(w)$ & cases \\
    \hline
    no change in $L(w)$ & (5) if $k \neq \mathrm{last}$ \\
    new hook & (6) if $k \neq \mathrm{last}$, (7) if $i \neq \mathrm{last}$ \\
    hooks unchanged but weight cone changes & (2), (3) if $k \neq 1$ or $i \neq 1$, (4) if $i \neq 1$\\
    hook becomes shorter & (3) if $i=k=1$\\
    hook becomes taller & (1), (4) if $i=1$\\
    hook becomes narrower & (8)\\
    hook becomes wider & (7) if $i = \mathrm{last}$ \\
    \end{tabular}
\end{table}

    \noindent We observe that in many cases, the graph $G_w$ changes by deleting vertices or adding new connected component which is a complete bipartite graph. Below we study the cases where there are new edge additions $a \rightarrow b^*$ in the bipartite graph $G_w$, while there are simultaneously no changes in either the number of vertices or the number of connected components. In particular, while moving along a chain in the Bruhat poset, by multiplying with $s_M$, one gets a family of non-isomorphic toric varieties of same dimension. 

\begin{corollary}\label{cor: change of the weight cone}
Let $w\in S_n$ such that $Y_w$ is a toric variety. Following the numbering of Corollary~\ref{cor: toric cases}, if $w$ and $M$ are as in (2), as in (4) such that $i \neq 1$, or, if they are as in (3) such that $k \neq 1$ or $i \neq 1$, then $Y_{w \cdot s_M}$ is still toric and $L'(w\cdot s_M)$ does not differ from $L'(w)$. However, the graph $G_w$, and with it the edge cone, change in the following ways:
\begin{itemize}
    \item If $w$ and $M$ are as in (2) or as in (4) such that $i \neq 1$, the graph $G_w$ gains an edge $w(M) \to M^*$.
    \item If $w$ and $M$ are as in (3) and if $k \neq 1$ or $i \neq 1$, then $G_w$ loses an edge $w(M+1) \to M^*$.
\end{itemize}
\end{corollary}

\section{Kazhdan-Lusztig Varieties}\label{sec: KL varieties}

\noindent We now want to introduce another, related type of varieties, the Kazhdan-Lusztig (KL) varieties $\Nvw$. These depend on two permutations $v$ and $w$. In a sense, Kazhdan-Lusztig varieties are generalisations of matrix Schubert varieties: for a specific choice, $v$, the Kazhdan-Lusztig variety $\Nvw$ is actually isomorphic to the matrix Schubert variety $\overline{X_{w'}}$, where $w'$ is an embedding of $w$ into a larger matrix. We refer the reader to \cite{donten2021complexity} for further intricacies of this connection between the two classes of varieties. To understand better how the KL variety $\Nvw$ depends on the two permutations, we first must introduce the partial order on $S_n$. 
 
\subsection{The Bruhat order}\label{subsec: Bruhat order}

\noindent In this section, we freely use definitions and results, as stated in  \cite{bjorner2005combinatorics}.
Any permutation $w\in S_n$ can be written as a product of simple reflections $w= s_{i_1} \cdots s_{i_k}$.
Such a representation of minimal length is called a \emph{reduced word expression} of $w$ and this minimal length is called the \emph{Coxeter length} $\ell(w)$.
We define the following partial order for the permutations in $S_n$, which is relevant for the definition of Kazhdan-Lusztig varieties.

\begin{definition}
    The \emph{Bruhat order} on $S_n$ is the partial order defined by: $v \leq w$ if and only if there exists a reduced word expression of $v$ that is a subword of a reduced word expression of $w$. More precisely, $v \leq w$ if $w= s_{i_1} \cdots s_{i_k}$ is a reduced word expression, then there exists a reduced word expression $v=s_{i_{j_1}} \cdots s_{i_{j_l}}$ of $v$ such that $1 \leq j_1 \leq \ldots \leq j_l \leq k$.
\end{definition}

\noindent This latter definition of the Bruhat order is also called \textit{the subword property}.
Note that if $v\leq w$, the permutation $w$ arises from $v$ by adding $\ell(w)-\ell(v)$ simple reflections. We denote this difference in length by $\ell(v,w)$. If $\ell(v,w)=1$ we say that \emph{$v$ is covered by $w$} in the Bruhat order and denote it as $v \lessdot w$. Equivalently, we can characterise the Bruhat order as follows.

\begin{lemma}[{\cite[Lemma 2.2.1]{bjorner2005combinatorics}}]
    Let $v,w \in S_n$. Then $v < w$ if and only if there exists some transposition $t \in T_n$ such that $w = vt$ and $\ell(v,w) >0$.
\end{lemma}

\noindent The group $S_n$ is a poset with respect to the Bruhat order. On it, we can also define the \emph{Bruhat interval} $[v,w],$ for $v\leq w$, as the set $\{ u \in S_n \mid v \leq u \leq w \}$. The \emph{atoms} of $[v,w]$ are defined to be the elements of the interval with length $\ell(v)+1$. The Coxeter length of the permutation can also be determined via Rothe diagrams that have been introduced in Section~\ref{subsec: prelim mSV}:
$$\ell(w) = \frac{n(n-1)}{2} - |\Dc(w)|.$$

\noindent Also, note how the following theorem relates to the Fulton determinantal conditions from Theorem~\ref{thm:fultonessential}.
\begin{lemma}[{\cite[Theorem 2.1.5]{bjorner2005combinatorics}}]
    Let $v,w \in S_n$. Then $v \leq w$ in Bruhat order if and only if $r_v(a,b) \leq r_w(a,b)$
    for all $a,b \in \{1, \ldots, n \}$.
\end{lemma}

\subsection{Background on Kazhdan-Lusztig varieties} 
\noindent Again, we follow the exposition as laid out in \cite{donten2021complexity}. Denote $G := GL_n(\mathbb C)$. We consider the flag variety $G/B$, where $G$ acts on $G/B$ via left multiplication. The fixed points of the flag variety, under the action of $T$ are of the form $wB$, for a permutation $w \in S_n$. The closure of the orbit $BwB/B$ is the \textit{Schubert variety} $X_w \subseteq G/B$, which has dimension $\ell(w)$. The \textit{opposite Schubert cell} $\Omega_v^\circ$, corresponding to the permutation $v$ is $B_{-}vB/B$, where $B_-$ is the Borel subgroup of lower triangular matrices. 

\begin{definition}
    The \textit{Kazhdan-Lusztig (KL) variety} corresponding to the permutations $v$ and $w$ in $S_n$ is
    $$\mathcal{N}_{v,w} \coloneqq X_w \cap \Omega_v^\circ.$$
\end{definition}
\noindent We note that $\mathrm{dim}(\Nvw) = \ell(w)-\ell(v) = |\Dc(v)| - |\Dc(w)| $, as shown for example, in \cite[Cor 3.3]{woo2008governing}. Moreover, the KL variety $\Nvw$ is non-empty if and only if $v\leq w$. Again, as for the case of matrix Schubert varieties, there exists a determinantal description of the generators of the corresponding KL ideal. We first define the following space, leading on from \cite{woo2008governing} and \cite{woo2012grobner}. 
\begin{definition}
    Given $v\in S_n$, we define $\Sigma_v\subset\mathbb{C}^{n\times n}$ as the space that consists of matrices $Z$ such that:
   $$ \begin{cases}
        Z_{v(i),i} = 1 & \text{ for all } i \in [n],\\
        Z_{v(i),a} = 0 & \text{ for } a > i,\\
        Z_{b,i} = 0 & \text{ for } b<v(i).
    \end{cases}$$\\
    We denote a generic matrix in $\Sigma_v$ by $Z^{(v)}$. The representation above means that it can be constructed via the following procedure.
    \begin{itemize}
        \item Write down $1$s in the same positions as in the permutation matrix of $v$.
        \item Then, place $0$s in all of the entries north or east of a $1$.
        \item The remaining entries are left as free entries, indexed by $z_{ij}$.
    \end{itemize}
\end{definition}

\noindent Now, consider the following proposition, which allows us to infer the determinantal conditions.
\begin{proposition}[{\cite[Chapter~10]{fulton1997algebraic}}]. The map $\pi: G \rightarrow G/B$ that sends a matrix $Z$ to the coset $ZB$ induces an isomorphism between $\Sigma_v$ and $\Omega^\circ_v$.
\end{proposition}
\noindent
Thus, following \cite{woo2008governing}, given $v,w \in S_n$ we have:
    $$\Nvw \cong \Xw \cap \Sigma_v.$$
Hence, the defining ideal of $\Nvw$ is generated by the determinantal equations obtained from imposing the rank conditions associated to $w$ from Theorem~\ref{thm:fultonessential} onto $Z^{(v)}$. We denote a generic element of $\Nvw$ by $\Zg$.

\begin{example}\label{ex: 53412 + 43125}
    Consider the permutations $w = 53412$ and $v=43125$. Note that $w>v$. The opposite Rothe diagrams $\Dc(w)$ and $\Dc(v)$ of $w$ and $v$ look as follows:\\

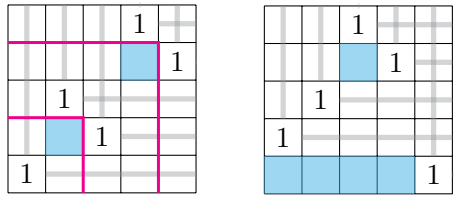
\begin{center}
\begin{tikzpicture}[scale = 0.5]

  \draw[step=1.0,black,thin] (0,0) grid (5,5);

    \node at (0.5,0.5) {{$1$}}; 
   \node at (1.5,2.5) {{$1$}}; 
   \node at (2.5,1.5) {{$1$}}; 
   \node at (3.5,4.5) {{$1$}};
   \node at (4.5,3.5) {{$1$}};

  \draw [line width = 0.0mm, draw=CornflowerBlue, fill=CornflowerBlue, draw opacity = 0.5, fill opacity=0.5]
       (1,1) -- (1,2) -- (2,2) -- (2,1)  -- cycle;
   \draw [line width = 0.0mm, draw=CornflowerBlue, fill=CornflowerBlue, draw opacity = 0.5, fill opacity=0.5]
       (3,3) -- (4,3) -- (4,4) -- (3,4)  --cycle;

  \draw[line width = 0.7mm, draw=Gray, draw opacity = 0.4]
  (0.5, 1) -- (0.5,5);
  \draw[line width = 0.7mm, draw=Gray, draw opacity = 0.4]
  (1.5, 3) -- (1.5,5);
  \draw[line width = 0.7mm, draw=Gray, draw opacity = 0.4]
  (2.5, 2) -- (2.5,5);
  \draw[line width = 0.7mm, draw=Gray, draw opacity = 0.4]
  (3.5, 5) -- (3.5,5);
  \draw[line width = 0.7mm, draw=Gray, draw opacity = 0.4]
  (4.5, 4) -- (4.5,5);

  \draw[line width = 0.7mm, draw=Gray, draw opacity = 0.4]
  (1, 0.5) -- (5,0.5);
  \draw[line width = 0.7mm, draw=Gray, draw opacity = 0.4]
  (3, 1.5) -- (5,1.5);
  \draw[line width = 0.7mm, draw=Gray, draw opacity = 0.4]
  (2, 2.5) -- (5,2.5);
  \draw[line width = 0.7mm, draw=Gray, draw opacity = 0.4]
  (5,3.5 ) -- (5,3.5);
  \draw[line width = 0.7mm, draw=Gray, draw opacity = 0.4]
  (4, 4.5) -- (5,4.5);

  \draw[line width = 0.4mm, draw=magenta]
  (0, 2) -- (2,2) -- (2,0);
  \draw[line width = 0.4mm, draw=magenta]
  (0, 4) -- (4,4) -- (4,0);

\end{tikzpicture}  
\qquad
\begin{tikzpicture}[scale = 0.5]

  \draw[step=1.0,black,thin] (0,0) grid (5,5);

    \node at (0.5,1.5) {{$1$}}; 
   \node at (1.5,2.5) {{$1$}}; 
   \node at (2.5,4.5) {{$1$}}; 
   \node at (3.5,3.5) {{$1$}};
   \node at (4.5,0.5) {{$1$}};

  \draw [line width = 0.0mm, draw=CornflowerBlue, fill=CornflowerBlue, draw opacity = 0.5, fill opacity=0.5]
       (0,0) -- (0,1) -- (4,1) -- (4,0)  -- cycle;
   \draw [line width = 0.0mm, draw=CornflowerBlue, fill=CornflowerBlue, draw opacity = 0.5, fill opacity=0.5]
       (2,3) -- (3,3) -- (3,4) -- (2,4)  --cycle;

  \draw[line width = 0.7mm, draw=Gray, draw opacity = 0.4]
  (0.5, 2) -- (0.5,5);
  \draw[line width = 0.7mm, draw=Gray, draw opacity = 0.4]
  (1.5, 3) -- (1.5,5);
  \draw[line width = 0.7mm, draw=Gray, draw opacity = 0.4]
  (2.5, 5) -- (2.5,5);
  \draw[line width = 0.7mm, draw=Gray, draw opacity = 0.4]
  (3.5, 4) -- (3.5,5);
  \draw[line width = 0.7mm, draw=Gray, draw opacity = 0.4]
  (4.5, 1) -- (4.5,5);

  \draw[line width = 0.7mm, draw=Gray, draw opacity = 0.4]
  (5, 0.5) -- (5,0.5);
  \draw[line width = 0.7mm, draw=Gray, draw opacity = 0.4]
  (1, 1.5) -- (5,1.5);
  \draw[line width = 0.7mm, draw=Gray, draw opacity = 0.4]
  (2, 2.5) -- (5,2.5);
  \draw[line width = 0.7mm, draw=Gray, draw opacity = 0.4]
  (4,3.5 ) -- (5,3.5);
  \draw[line width = 0.7mm, draw=Gray, draw opacity = 0.4]
  (3, 4.5) -- (5,4.5);

\end{tikzpicture}  
.
\end{center}

\noindent The dimension of $\Nvw$ is thus three. Furthermore, $\mathrm{Ess}(w) = \{(4,2),(2,4)\}$. In the diagram, the areas on which the Fulton conditions apply are outlined in pink in $\Dc(w)$. 
A generic element in $\Sigma_v$ can be written as
$$Z^{(v)} = \begin{pmatrix}
    0&0&1&0&0 \\
    0&0& z_{23} & 1&0 \\
    0&1&0&0&0 \\
    1&0&0&0&0 \\
    z_{51}&z_{52}&z_{53}&z_{54}&1
\end{pmatrix}.$$

\noindent Now, we know that $\Nvw$ is the variety given by the determinantal conditions $r_{Z^{(v)}}(a,b)\leq r_w(a,b)$, for $(a,b) \in \mathrm{Ess}(w)$. As we have two elements in the essential set, we consider the following rank conditions. 
    \begin{itemize}
        \item $r_{Z^{(v)}}(4,2) \leq r_{w}(4,2) = 1$. Thus,  $0=\left|\begin{array}{cc}
            1 & 0 \\
            z_{51} & z_{52} 
        \end{array}\right| = z_{52}$.
        \item $r_{Z^{(v)}}(2,4) \leq r_{w}(2,4) = 3$. Thus, $0=\left|  \begin{array}{cccc}
            0 & 0 & z_{23} & 1  \\
            0 & 1 & 0 & 0  \\
            1 & 0 & 0 & 0  \\
            z_{51} & z_{52} & z_{53} & z_{54}
        \end{array}\right| = z_{53} - z_{23} z_{54}$. 
    \end{itemize}
    Hence, the ideal, in this case, is given by $\langle z_{52}, z_{53} - z_{23} z_{54}\rangle$.
\end{example}

\noindent In practice, for simplicity, we will sometimes record information like the one in the previous example in a single diagram as follows:

\begin{center}
\begin{tikzpicture}[scale = 0.5]

  \draw[step=1.0,black,thin] (0,0) grid (5,5);

    \node at (0.5,1.5) {{$1$}}; 
   \node at (1.5,2.5) {{$1$}}; 
   \node at (2.5,4.5) {{$1$}}; 
   \node at (3.5,3.5) {{$1$}};
   \node at (4.5,0.5) {{$1$}};

  \draw [line width = 0.0mm, draw=CornflowerBlue, fill=CornflowerBlue, draw opacity = 0.3, fill opacity=0.3]
       (0,0) -- (0,1) -- (4,1) -- (4,0)  -- cycle;
   \draw [line width = 0.0mm, draw=CornflowerBlue, fill=CornflowerBlue, draw opacity = 0.3, fill opacity=0.3]
       (2,3) -- (3,3) -- (3,4) -- (2,4)  --cycle;

  \draw[line width = 0.7mm, draw=Gray, draw opacity = 0.4]
  (0.5, 2) -- (0.5,5);
  \draw[line width = 0.7mm, draw=Gray, draw opacity = 0.4]
  (1.5, 3) -- (1.5,5);
  \draw[line width = 0.7mm, draw=Gray, draw opacity = 0.4]
  (2.5, 5) -- (2.5,5);
  \draw[line width = 0.7mm, draw=Gray, draw opacity = 0.4]
  (3.5, 4) -- (3.5,5);
  \draw[line width = 0.7mm, draw=Gray, draw opacity = 0.4]
  (4.5, 1) -- (4.5,5);

  \draw[line width = 0.7mm, draw=Gray, draw opacity = 0.4]
  (5, 0.5) -- (5,0.5);
  \draw[line width = 0.7mm, draw=Gray, draw opacity = 0.4]
  (1, 1.5) -- (5,1.5);
  \draw[line width = 0.7mm, draw=Gray, draw opacity = 0.4]
  (2, 2.5) -- (5,2.5);
  \draw[line width = 0.7mm, draw=Gray, draw opacity = 0.4]
  (4,3.5 ) -- (5,3.5);
  \draw[line width = 0.7mm, draw=Gray, draw opacity = 0.4]
  (3, 4.5) -- (5,4.5);

  \draw[line width = 0.4mm, draw=magenta]
  (0, 2) -- (2,2) -- (2,0);
  \draw[line width = 0.4mm, draw=magenta]
  (0, 4) -- (4,4) -- (4,0);

  \node at (1.5,0.5) {\small \textcolor{Blue}{$0$}}; 
  \node at (0.5,0.5) {\small \textcolor{Blue}{$z_{51}$}}; 
  \node at (2.5,0.5) {\small \textcolor{Blue}{$z_{53}$}}; 
  \node at (3.5,0.5) {\small \textcolor{Blue}{$z_{54}$}}; 
  \node at (2.5,3.5) {\small \textcolor{Blue}{$z_{23}$}};
\end{tikzpicture}.
\end{center}

\noindent From this, one can read off the opposite Rothe diagram of $v$; the
structure of $\Zg$ (with possible zero entries); and the area of of $Z^{(v)}$ on which the Fulton conditions apply.

\subsection{The usual torus action on KL varieties} Now, we wish to define the complexity for the usual torus action on KL varieties. This action is the restriction of the  $T$-action on $\Sigma_v$, induced by the action of left multiplication on the opposite Schubert cell via the map $\pi$, onto $\Nvw$. We follow the exposition, as detailed in \cite{donten2021complexity}. We denote the weight cone associated to the KL variety $\Nvw$ as $\sigma_{v,w}$. 
The weights for each coordinate $z_{ij}$ are $e_{v(j)}-e_i$, where $e_1, \ldots, e_n$ again denotes the standard basis of $M(T)_{\mathbb R}$. To highlight which coordinates of $\Zg$ give us such weights, the authors define the notion of unexpected zeros in \cite[Def 4.8]{donten2021complexity}. In practice, they use the definition to refer to all zeros in $\Zg$. In fact, one needs an even smaller set of generators than the ones coming from nonzero coordinates to span the weight cone, which is why we extend the notion of unexpected zeros to the following:

\begin{definition}\label{def: unexpected zero}
    An \textit{unexpected zero} for $\Nvw$, is an entry $z_{ij}$ of $Z^{(v)}$ such that $t_{v(j),i}v \not\leq w$. 
\end{definition} 
\noindent
The coordinates $z_{ij}$ that are not unexpected zeros
are always nonzero coordinates of $\Zg$; but as we will see in Example \ref{ex: actual unexp zeros}, the converse must not hold. Then, the weight cone $\sigma_{v,w}$ is spanned by weights corresponding to the $z_{ij}$ that are not unexpected zeros:
$$\sigma_{v,w} = \mathrm{ Cone } (e_{v(j)}-e_i|z_{ij} \text{ is not an unexpected zero}).$$
In fact, by \cite[Theorem 4.11]{donten2021complexity}, the extremal ray generators are the weights corresponding to $z_{ij}$ such that $v \lessdot t_{v(j),i}v \leq w$. 

\begin{example}\label{ex: actual unexp zeros}[Example~\ref{ex: 53412 + 43125} continued]
    Continuing with $w=53412$ and $v=43125$, we know that a generic matrix in $\Nvw$ is given by
   $$\Zg = \left(\begin{array}{ccccc}
        0 & 0 & 1 & 0 & 0 \\
        0 & 0 & z_{23} & 1 & 0 \\
        0 & 1 & 0 & 0 & 0 \\
        1 & 0 & 0 & 0 & 0 \\
        z_{51} & \textcolor{magenta}{0} & z_{53} & z_{54} & 1 \\
   \end{array}\right).$$
   The zero marked in pink arises from the Fulton conditions. 
   However, there is another unexpected zero, namely $z_{53}$, and indeed, $t_{15}v \not\leq w$. 
   This arises from the fact that one of the Fulton conditions is $z_{53}=z_{23}z_{54}$. Therefore,
   if we know how the torus acts on $z_{23}$ and $z_{54}$, then we also know how it acts on $z_{53}$. Thus, $\sigma_{v,w} = \mathrm{Cone}(e_4-e_5, e_1-e_2, e_2-e_5)$ and it is $3$-dimensional. We also note that $\dim\Nvw = |\Dc(v)| - |\Dc(w)|=3$. So, the $T$-action on $\Nvw$ has a dense orbit. 
\end{example}

\begin{remark}\label{rem: corresp zeros}
It is still true that the weight cone is generated by the weights corresponding to actual non-zero coordinates $z_{ij}$ of $\Zg$. Any such weight coming from $z_{ij} \neq 0$ that is an unexpected zero can then be achieved as a linear combination of weights coming from non unexpected zeros. For example, in Example~\ref{ex: actual unexp zeros}, the weight for $z_{53}$ can be written as the sum of weights for $z_{23}$ and $z_{54}$. More generally, we can also write weights of coordinates $z_{ij}$ that are not unexpected zeros but have $\ell(v,t_{v(j),i}v)>1$ as linear combinations of other weights, since they are not minimal ray generators.
\end{remark}
\noindent We are again able to compute the dimension of the weight cone via an associated DAG, as done in \cite[Definition 4.14]{donten2021complexity} (although note that the notation differs).

\begin{definition}\label{def: graph Gwv}
Let $v, w \in S_n$. We then define the graph $\Gvwt$ to be the graph given by:
    $$V(\Gvwt) = [n], \hspace{0.4cm} E(\Gvwt) = \{(v(j)\rightarrow i)\mid t_{v(j),i}v \leq w, \,\, (i,j) \in \Dc(v)  \}.$$
\end{definition}

\noindent Then, since the weight cone $\sigma_{v,w}$ coincides with the edge cone of $\Gvwt$,

we can use this graph, along with Lemma~\ref{lem:weightcondim} to define the dimension of the weight cone: 
\begin{align*}
\text{dim}(\sigma_{v,w}) & 
= |V(\Gvwt)|-\#(\text{connected components of } \Gvwt). 
 \end{align*}
As we know, the KL variety $\Nvw$ is then of complexity $k$ if and only if
$\text{ dim }(\sigma_{v,w}) = \text{ dim }(\Nvw) - k.$
\begin{example}
    We conclude the running example $w=53412$ and $v=43125$. The graph $\Gvwt$ is drawn below:

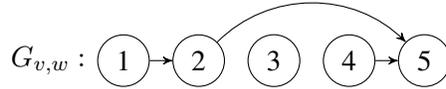
\begin{center}
\begin{tikzpicture}[->,>=stealth',auto,node distance=1cm]
    \node (G) {$\Gvwt:$};
    \node[shape=circle,draw=black] (1) [right of =G] {1};
    \node[shape=circle,draw=black] (2) [right of=1] {2};
    \node[shape=circle,draw=black] (3) [right of=2] {3};
    \node[shape=circle,draw=black] (4) [right of=3] {4};
    \node[shape=circle,draw=black] (5) [right of=4] {5};

    \path[every node/.style={font=\sffamily\small}]
    (1) edge node [right] {} (2)
    (4) edge node [right] {} (5);
    \draw[->] (2) to[out=45, in=135, looseness=1] (5);
  
\end{tikzpicture}
\end{center}
\noindent 
Then, the dimension of the weight cone is $5-3=2$.
\end{example}

\subsection{Glueing toric Bruhat intervals}\label{subsec: glueing toric intervals}
In this section, we are concerned with how the complexity of a KL variety changes when we move one of the permutations $v$ or $w$ along a chain in the Bruhat poset. This question has been posed before in \cite[Section 5.4]{donten2021complexity}, but was left unanswered.
We use many ideas first formulated in \cite{tsukerman2015} and relate them to our work. First, we invoke the following definition. 
\begin{definition}[{\cite[Definition 4.9]{tsukerman2015}}]
Let $v \leq w$ and $\overline{T}(v):= \{ t \in T_n \mid v \lessdot vt \leq w \}$ be the set of transpositions $t$ such that $vt$ is an atom of the interval $[v,w]$. Define the graph \emph{$G^{\textrm{at}}$} with vertex set $[n]$ by adding an edge $(a,b)$ if and only if $t_{a,b} \in \overline{T}(v)$.
\end{definition}
\noindent 
By removing the decomposable edges from $\Gvwt$, one obtains a directed acyclic graph as defined in \cite[Definition 4.14]{donten2021complexity}, which has edges 
$v(j)\rightarrow i$ for 
$(i,j) \in \Dc(v)$ with
$v \lessdot t_{v(j),i}v \leq w$. These correspond to the minimal ray generators of $\sigma_{v,w}$. Clearly, the connected components of this new graph coincide with those of $\Gvwt$. Furthermore, one can see that its underlying undirected graph is isomorphic to the graph $G^{\textrm{at}}$. In \cite[Corollary 4.16]{donten2021complexity}, this is used to show that the complexity of the KL variety $\Nvw$ is the same as that of the Richardson variety $X^v_w$.
We now recall the notion of another graph related to an interval $[v,w]$, as defined in \cite{tsukerman2015}.
\begin{definition}[{\cite[Definition 4.5]{tsukerman2015}}]
    Let $v \leq w$ and let $\mathcal C: v= u_0 \lessdot u_1 \lessdot \ldots \lessdot u_l =w$ be a chain from $v$ to $w$. Define the undirected graph $G_{\mathcal C}$ with vertex set $[n]$ by adding an edge $(a,b)$ if and only if $u_{i+1} = u_i t_{a,b}$ for some $0 \leq i \leq l-1$.
\end{definition}

\noindent Again, this graph $G_{\mathcal C}$ is related to the other graphs, as explained below.

\begin{lemma}[{\cite[Corollary 4.8, Corollary 3.13, Proposition 4.10]{tsukerman2015}}]
    Let $v \leq w$. The connected components of $G_{\mathcal C}$ are independent of the choice of the chain $\mathcal C$. 
    There exists a chain $\mathcal C$ from $v$ to $w$ such that $G_{\mathcal C} = G^{\textrm at}$. In particular, both graphs have the same connected components.
\end{lemma}

\begin{example}\label{ex: different graphs}
Consider the permutations $v=12435$ and $w=41325$. Two possible chains across the interval $[v,w]$ are 
$
    \mathcal C_1: v \lessdot 14235 \lessdot 41235 \lessdot w
$
and 
$
    \mathcal C_2: v \lessdot 13425 \lessdot 31425 \lessdot w
$. 
Also, the set of atoms is $\{14235, 13425, 21435 \}$. Then, the graphs below all have the same connected components.

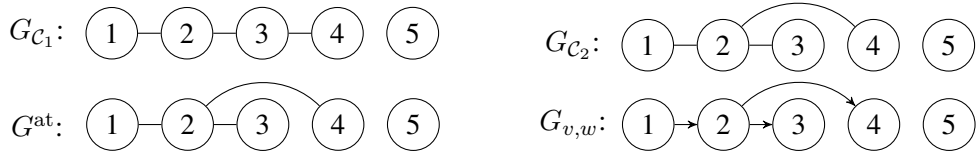
\begin{minipage}{0.45\textwidth}
\hspace{0.55cm}
\begin{tikzpicture}[auto,node distance=1cm, scale = 0.6]
    \node (G) {$G_{\mathcal C_1}$:};
    \node[shape=circle,draw=black] (1) [right of =G] {1};
    \node[shape=circle,draw=black] (2) [right of=1] {2};
    \node[shape=circle,draw=black] (3) [right of=2] {3};
    \node[shape=circle,draw=black] (4) [right of=3] {4};
    \node[shape=circle,draw=black] (5) [right of=4] {5};

    \path[every node/.style={font=\sffamily\small}]
    (1) edge node [right] {} (2)
    (2) edge node [right] {} (3)
    (3) edge node [right] {} (4);
\end{tikzpicture}
\end{minipage}
\begin{minipage}{0.45\textwidth}
\hspace{0.15cm}
\begin{tikzpicture}[auto,node distance=1cm, scale = 0.6]
    \node (G) {$G_{\mathcal C_2}$:};
    \node[shape=circle,draw=black] (1) [right of =G] {1};
    \node[shape=circle,draw=black] (2) [right of=1] {2};
    \node[shape=circle,draw=black] (3) [right of=2] {3};
    \node[shape=circle,draw=black] (4) [right of=3] {4};
    \node[shape=circle,draw=black] (5) [right of=4] {5};

    \path[every node/.style={font=\sffamily\small}]
    (1) edge node [right] {} (2)
    (2) edge node [right] {} (3);
    \draw[-] (2) to[out=45, in=135, looseness=1] (4);
\end{tikzpicture}
\end{minipage} \\
\begin{minipage}{0.45\textwidth}
\hspace{1cm}
\begin{tikzpicture}[auto,node distance=1cm, scale = 0.6]
    \node (G) {$G^{\mathrm{at}}$:};
    \node[shape=circle,draw=black] (1) [right of =G] {1};
    \node[shape=circle,draw=black] (2) [right of=1] {2};
    \node[shape=circle,draw=black] (3) [right of=2] {3};
    \node[shape=circle,draw=black] (4) [right of=3] {4};
    \node[shape=circle,draw=black] (5) [right of=4] {5};

    \path[every node/.style={font=\sffamily\small}]
    (1) edge node [right] {} (2)
    (2) edge node [right] {} (3);
    \draw[-] (2) to[out=45, in=135, looseness=1] (4);
  
\end{tikzpicture}
\end{minipage}
\begin{minipage}{0.45\textwidth}
\hspace{0.5cm}
\begin{tikzpicture}[->,>=stealth',auto,node distance=1cm, scale = 0.6]
    \node (G) {$G_{v,w}$:};
    \node[shape=circle,draw=black] (1) [right of =G] {1};
    \node[shape=circle,draw=black] (2) [right of=1] {2};
    \node[shape=circle,draw=black] (3) [right of=2] {3};
    \node[shape=circle,draw=black] (4) [right of=3] {4};
    \node[shape=circle,draw=black] (5) [right of=4] {5};

    \path[every node/.style={font=\sffamily\small}]
    (1) edge node [right] {} (2)
    (2) edge node [right] {} (3);
    \draw[->] (2) to[out=45, in=135, looseness=1] (4);
  \end{tikzpicture}
\end{minipage}
\end{example}

\noindent Now, in order to relate these concepts to our work, we introduce the following.
\begin{definition}\label{defn:cycnumb}
    For a graph $G$, its \emph{cyclomatic number}
    $$\cn (G) = |E(G)| - |V(G)| + \#(\text{connected components of }G)$$
    is the minimum number of edges that must be removed from $G$ to make it into a forest.
\end{definition}

\noindent As a consequence, 
\begin{align*}
    \dim \sigma_{v,w} &= n- \#(\textrm{connected components of } G_{v,w}) \\
    &= n- \#(\textrm{connected components of } G_{\mathcal C}) \\
    &= |E(G_{\mathcal C})| - \cn(G_{\mathcal C}) \\
    &= \ell(v,w) - \cn(G_{\mathcal C}).
\end{align*}
Thus, we can formalise this thinking as below.
\begin{corollary}
    Let $v \leq w$. The complexity of $\Nvw$ is given by $\cn(\GC)$. In particular, $\Nvw$ is toric if and only if $\GC$ is a forest.
\end{corollary}

\noindent Applying \cite[Corollary 4.16]{donten2021complexity}, this also recovers the result \cite[Proposition 4.12]{tsukerman2015}. Furthermore,  we can use it to gain immediate information about the complexity of short Bruhat intervals.

\begin{corollary}
    For $v,w$ with $\ell(v,w) \leq 2$, the KL variety $\Nvw$ is always toric.
    \begin{proof}
        For $v \lessdot w = vt$ the graph $G_{\mathcal C}$ has only one edge. Adding another edge only makes a cycle if it corresponds to reversing the permutation $t$.
    \end{proof}
\end{corollary}

\noindent We can now address the question of how the complexity of the KL variety $\Nvw$ changes when we move along a chain in the Bruhat poset. 
 \begin{proposition}\label{prop: complexity change}
    When $w' \lessdot w$, the complexity of the KL variety $\Nvw$ is either the same, or one more than the complexity of $\mathcal N_{v,w'}$. 
    \begin{proof}
        Extending the interval adds one edge to the corresponding graph $G_C$, which can reduce the number of connected components by at most one.
    \end{proof}
\end{proposition}
\noindent We are then immediately able to state the following.
\begin{corollary}\label{cor: KL toric then}
If $\Nvw$ is toric, then the KL variety associated to any subinterval of $[v,w]$ is again toric.

\end{corollary}

\noindent As a consequence of Proposition \ref{prop: complexity change}, for any extension of an interval with toric KL variety by one level, the resulting KL variety can either still be toric or of complexity one. One can easily characterise the case in which it remains toric.

\begin{proposition}\label{prop: extend toric interval}
    Let $\mathcal N_{v,w'}$ be toric with a chain $\mathcal C$ in $[v,w']$ and $w' \lessdot w$ with $w=w't_{ab}$. Then, $\Nvw$ is toric if and only if the vertices $a$ and $b$ are not in the same connected component of $G_{\mathcal C}$.
    \begin{proof}
        Since $\mathcal N_{v,w'}$ is toric, $G_{\mathcal C}$ is a forest. In order for $\Nvw$ to not be toric, adding the edge $(a,b)$ to $G_{\mathcal C}$ must create a cycle (or a double edge) in $G_{\mathcal C}$. This happens precisely when $a$ and $b$ are in the same connected component.
    \end{proof}
\end{proposition}

\noindent Analogous results hold for extending an interval $[v',w]$ to $[v,w]$ with $v \lessdot v'$. The last result can also be generalised to the glueing of two intervals with toric KL varieties.

\begin{lemma}\label{lem: gluing two intervals}
    Let $[u,v]$, $[v,w]$ be intervals with chains $\mathcal C_1$, $\mathcal C_2$ be such that $\mathcal N_{u,v}$ and $\mathcal N_{v,w}$ are toric. Then, $\mathcal N_{u,w}$ is toric if and only if there exist no vertices $a_1, \ldots,a_{2k} \in [n]$, components $B_1,\ldots, B_{k}$ of $G_{\mathcal C_1}$ and components $C_1,\ldots, C_{k}$ of $G_{\mathcal C_2}$ such that:
    \begin{align*}
    a_1 \in V(B_1) \cap V(C_1), \, 
      a_2 \in V(C_1) \cap V(B_2), \,
      a_3 \in V(B_2) \cap V(C_2), \, \ldots \, , \\
      a_{2k-1} \in V(B_{k}) \cap V(C_k), \, 
      a_{2k} \in V(C_k) \cap V(B_1).
     \end{align*}
     \end{lemma}
\begin{proof}
  This arises from the fact that combining the graphs $G_{\mathcal C_1}$ and $G_{\mathcal C_2}$ creates a cycle precisely in the scenario given in the theorem statement.
\end{proof}     

\noindent In the final part of this section, we will provide some observations to help apply Proposition~\ref{prop: extend toric interval}. 
In practice, one might not always know what a chain $\mathcal C$ in the given interval looks like. We want to determine the connected components of $G_{\mathcal C}$ for a toric KL variety simply by looking at the one-line notations. 
We say that $a \in [n]$ is \textit{moved in} $[v,w]$, if $u \lessdot u t \leq w$ for some $u \in [v,w]$ and $t \in S_n$ changing the position of $a$. This happens precisely if the vertex $a$ is not isolated $G_{\mathcal C}$.

\begin{lemma}\label{lem:toric move}
    If $\Nvw$ is toric, then $a\in [n]$ is moved in $[v,w]$ if and only if $v(a) \neq w(a)$.
    \begin{proof}
        It is clear that if $v(a) \neq w(a)$ then $a$ is moved in $[v,w]$. For the converse, assume, for contradiction, that $a$ is moved in $[v,w]$ but $v(a)=w(w)$. Then, there must be transpositions $t_{a,b_1},t_{b_1,b_2}, \ldots, t_{b_{l-1},b_l},t_{b_l,a}$ to move $v(a)$ back to its original position. This is a cycle in $G_{\mathcal C}$, contradicting the toricness of $\Nvw$.
    \end{proof}
\end{lemma}

\noindent It is then possible to read off the connected components of an interval with toric KL variety, simply by comparing the one-line notations.

\begin{lemma}\label{lem:ccs of GC}
    If $\Nvw$ is toric, then the connected components of $G_{\mathcal C}$ are the sets of vertices $\{a_1, \ldots, a_k\} \subset [n]$ such that $v(a_i)=w(a_{i+1})$ for all $i=1,\ldots,k-1$ and $v(a_k)=w(a_i)$.
    \begin{proof}
        We conclude this by induction. If $v \lessdot w$, then the claim is obvious. Let $\Nvw$ be toric and $w'<w$, such that $\mathcal N_{v,w'}$ is also toric. By Proposition \ref{prop: extend toric interval}, the new chain $t$ from $w'$ to $w$ connects two connected components that were previously distinct. The vertices of these components have to be a set of the above form and the connection coming from $t$ will create a bigger such set.
    \end{proof}
\end{lemma}

\begin{example}
Consider $[v,w]$ as in Example \ref{ex: different graphs}. We already know what the connected components of $G_{\mathcal C_1}$ look like, but we can check that they are indeed as specified in Lemma \ref{lem:ccs of GC}. 
If we extend $\mathcal C_1$ to $w_1 = 42315 \gtrdot w$, we add the edge $(2,4)$ to $G_{\mathcal C_1}$, creating a circle. The complexity of $\mathcal N_{v,w_1}$ is $1$. Note that although $v(2)= 2= w_1(2)$, the vertex $2$ is moved in $[v,w_1]$.
If we extend $\mathcal C_1$ to $w_2 = 41352 \gtrdot w$, we add the edge $(4,5)$ to $G_{\mathcal C_1}$, connecting two previously distinct components. The variety $\mathcal N_{v,w_2}$ is still toric.
\end{example}

\section{Toric Families of KL varieties}\label{sec: toric families of KL}

\noindent
We are interested in how the permutations $v$ and $w$ influence the complexity of the Kazhdan-Lusztig variety $\Nvw$. While this is a broad field, restricting our attention to special kinds of families of KL varieties allows us to make such statements and even characterise the toric cases.
 
\subsection{KL varieties with no unexpected zeros}

\noindent In this section, we will have a closer look at the structure of the graph $\Gvwt$ via the matrices $Z^{(v)}_w$ in the KL variety. In particular, for KL varieties $\Nvw$ coming from permutations $v$ and $w$ for which there are no unexpected zeros in $\Sigma_v$, the opposite Rothe diagram innately contains useful information about the graph $\Gvwt$ and the complexity of $\Nvw$. 

We know, via Remark \ref{rem: corresp zeros}, that the weight cone $\sigma_{v,w}$ can not only be generated by weights coming from non unexpected zeros, but also by the weights coming from nonzero coordinates $z_{ij}$ of $\Zg$. We define the corresponding graph
$\Gvwtt$ by $V(\Gvwtt) = [n]$ and
$$E(\Gvwtt) = \{ (v(j) \to i) \mid z_{ij} \textrm{ nonzero in } \Zg \}.$$
Then, $\Gvwt$ arises from $\Gvwtt$ by removing some decomposable edges; and the edge cone of $\Gvwtt$ is the weight cone $\sigma_{v,w}$. Whenever $\Zg$ has no unexpected zeros, it also has no coordinates that are zero. In this case, the graphs are the same. Therefore, statements about the complexity of $\Nvw$ in the case that there are no zero $z_{ij}$ are stronger than statements about the complexity in the case that there are no unexpected zeros. In the following section, we will make observations and prove statements about the graph $\Gvwt$ in the latter case, but one can easily see that everything is transferrable to the first setting if one considers $\Gvwtt$ instead. In the following, when we say a coordinate is nonzero, this can either mean it is not an unexpected zero or it is just actually not zero. The precise meaning should be clear from the context.

To understand the structure of $\Gvwt$, it is natural to ask where its edges and paths originate from in $\Zg$. 
In fact, the edges corresponding to coordinates $z_{ij}, z_{kl}$ in $\Zg$ share a vertex in $\Gvwt$ if and only if one of the following cases in the table hold.

\begin{table}[H]
\centering
\begin{tabular}{c|c|c}
     & in $\Gvwt$ & notation $\stackrel{\ast}{\dash}$\\
     \hline
    $i=k$ & $\rightarrow \boldsymbol{\cdot} \leftarrow$ & $\stackrel{r}{\dash}$\\
    $j=l$ & $\leftarrow \boldsymbol{\cdot} \rightarrow$ & $\stackrel{c}{\dash}$\\    
    $v(j) =k$ & $\leftarrow \boldsymbol{\cdot} \leftarrow$ & $\stackrel{_{\nwarrow}}{\dash}$\\
    $v(l) =i$ & $\rightarrow \boldsymbol{\cdot} \rightarrow$ &  $\stackrel{_{\searrow}}{\dash}$  
\end{tabular} \\
\caption{\small{In each of the cases on the left, the middle column shows in which way the edges corresponding to $z_{ij},z_{kl}$ touch the common vertex, also given by the left column. The right column shows how we denote such a connection in a path through the entries of $\Zg$. }}
\label{table}
\end{table}

\noindent The last two cases correspond to there being a $1$ in $\Zg$ in the same row as $z_{kl}$ and the same column as $z_{ij}$ or vice versa. In this way, we can track (an undirected) path in $\Gvwt$ via a path through the nonzero entries in $\Zg$. In this path, each step either moves within the same row or column; it crosses over a one in the same column to move to a lower row and column; or it crosses over a one in the same row to move to a higher row and column. We denote such a path as 
$z_{i_1 j_1} \stackrel{\ast}{\dash} z_{i_2 j_2} \stackrel{\ast}{\dash} \cdots \stackrel{\ast}{\dash} z_{i_r j_r}$, where the $\stackrel{\ast}{\dash}$ are as in the third column of the table and specify what the step looks like. This is also illustrated in Example \ref{ex: paths and cycles}.

\subsubsection{Connected components of $\Gvwt$}

\noindent For now, we are interested in the connected components of $\Gvwt$. Since the dimension of the weight cone depends on the number of connected components, so does the complexity of $\Nvw$. 

If $w$ imposes no unexpected zeros on $\Sigma_v$, then all components of $\Dc(v)$ that are not connected to each other by paths using $\stackrel{r}{\dash}$ or $\stackrel{c}{\dash}$ belong to different connected components of $\Gvwt$. Therefore, sets of elements in $\Dc(v)$ belonging to a certain connected component of $\Gvwt$ group together along the antidiagonal of $v$ and, of course, do not share a row or column. Indeed, if there were connected components $D_1, D_2$ of $\Dc(v)$ belonging to distinct connected components of $\Gvwt$, with one lying south-east of the other, then the coordinates $z_{ij}$, where $i$ is a row in the south-east component and $j$ is a column in the north-west component, are nonzero in $\Zg$ and induce an edge connecting the components in $\Gvwt$. We formalise this below, in order to count the number of connected components of $\Gvwt$.

\begin{lemma}\label{lem: connected components for no unexpected zeros}
    Let $v,w \in S_n$ be such that there are no unexpected zeros in  $\Sigma_v$.
    Then, the set of connected components of $\Gvwt$ with more than one vertex has the same cardinality as the set of south-west corners, written formally as:
    $$C_v := \{ (i,j)\in \Dc(v) \mid \forall i'>i: \, (i',j) \not\in \Dc(v) \,\, \& \,\, \forall j'<j: \, (i,j') \not\in \Dc(v) \} .$$
    \begin{proof}
    The connected components of $\Gvwt$ that are not an isolated vertex contain an edge coming from a nonzero coordinate $z_{ij}$. In this case, such coordinates are exactly the coordinates $(i,j) \in \Dc(v)$, since there are no unexpected zeros.
    We claim that every set of coordinates in $\Dc(v)$ that belong to a connected component of $\Gvwt$ has a unique southwest corner. Assume that $(i,j)$ and $(k,l)$ are both southwest corners of the same component of $\Gvwt$. Then, wlog, $i <k$ and $j <l$ and, therefore, $(k,j) \in \Dc(v)$ lies further to the southwest than the other points and its corresponding edge is obviously connecting points in the given connected component of $\Gvwt$. 
    On the other hand, since the sets of coordinates in $\Dc(v)$ belonging to distinct connected components of $\Gvwt$ share no rows or columns, they also do not share south-west corners.
    \end{proof}
\end{lemma}

\begin{lemma}\label{lem: isolated vertices for no unexpected zeros}
    Let $v,w \in S_n$ be such that there are no unexpected zeros in  $\Sigma_v$. Then, the number of isolated vertices of $\Gvwt$ is the same as the cardinality of the set of $1$s on the antidiagonal of $v$ that do not share a row or column with $\Dc(v)$. This set can be written as:
    $$A_v := \{ i \in [n] \mid v(i) = n-i+1 \,\, \& \,\, \forall j \in [n]: \, (j,i), (n-i+1,j) \not\in \Dc(v)\}.$$
    \begin{proof}
        Any such $1$ on the antidiagonal of $v$, say in position $(i,n-i+1)$, not sharing a row or column with $\Dc(v)$ implies that there is no edge in $\Gvwt$ to or from the vertex $i$.
        On the other hand, consider an isolated vertex $i$ of $\Gvwt$. 
        We know that $v$ has a $1$ in position $(i,v^{-1}(i))$. Wlog assume for contradiction that $v^{-1}(i) >i$ and the $1$ is above the antidiagonal. There must then also be a $1$ below the antidiagonal, say in position $(k,l)$, with $k>i$, $l> v^{-1}(i)$. However, then, $(k,v^{-1}(i)) \in \Dc(v)$, which causes an edge $(i \rightarrow k)$ in $\Gvwt$.
    \end{proof}
\end{lemma}

\noindent We can pose the same statements for the scenario in which there are no zero entries and can also define the dimension and complexity of $\Nvw$ in terms of these sets.

\begin{corollary}\label{cor: conn comp} 
    Let $v,w \in S_n$ be such that there are no (actual) zero coordinates in  $\Sigma_v$. Then there are $|C_v|$ many connected components of $\Gvwtt$ with more than one vertex and $|A_v|$ many isolated vertices of $\Gvwtt$. Thus
    $\dim \sigma_{v,w}= n- |C_v|-|A_v|$ and the complexity of $\Nvw$ is given by
    $\ell(v,w)-n +|C_v|+|A_v|$.
\end{corollary}

\begin{example}
    Consider $v=58672341$ and $w=12345678$. One can check that $w$ causes no (actual) zero entries in $\Zg$. The opposite Rothe diagram $\Dc(v)$ is shown below. The two components of $\Dc(v)$ that are in the outlined squares make up the edges of the two connected components of $\Gvwtt$ with more than one vertex. Their south-west corners are marked with $\times$. The $1$ on the antidiagonal in the upper right outlined square corresponds to the isolated vertex $1$. \\
 
   \begin{minipage}{0.4\textwidth}
  \begin{center}
  \hspace{0.5cm}
  \begin{tikzpicture}[scale=0.5]
  \draw[step=1.0,black,thin] (0,0) grid (8,8);

  \draw[line width = 0.7mm, draw=Gray, draw opacity = 0.4]
  (0.5, 4) -- (0.5,8);
  \draw[line width = 0.7mm, draw=Gray, draw opacity = 0.4]
  (1.5, 1) -- (1.5,8);
  \draw[line width = 0.7mm, draw=Gray, draw opacity = 0.4]
  (2.5, 3) -- (2.5,8);
  \draw[line width = 0.7mm, draw=Gray, draw opacity = 0.4]
  (3.5, 2) -- (3.5,8);
  \draw[line width = 0.7mm, draw=Gray, draw opacity = 0.4]
  (4.5, 7) -- (4.5,8);
  \draw[line width = 0.7mm, draw=Gray, draw opacity = 0.4]
  (5.5, 6) -- (5.5,8);
  \draw[line width = 0.7mm, draw=Gray, draw opacity = 0.4]
  (6.5, 5) -- (6.5,8);  

  \draw[line width = 0.7mm, draw=Gray, draw opacity = 0.4]
  (2, 0.5) -- (8,0.5);
  \draw[line width = 0.7mm, draw=Gray, draw opacity = 0.4]
  (4, 1.5) -- (8,1.5);
  \draw[line width = 0.7mm, draw=Gray, draw opacity = 0.4]
  (3, 2.5) -- (8,2.5);
  \draw[line width = 0.7mm, draw=Gray, draw opacity = 0.4]
  (1, 3.5) -- (8,3.5);
  \draw[line width = 0.7mm, draw=Gray, draw opacity = 0.4]
  (7, 4.5) -- (8,4.5);
  \draw[line width = 0.7mm, draw=Gray, draw opacity = 0.4]
  (6, 5.5) -- (8,5.5);
  \draw[line width = 0.7mm, draw=Gray, draw opacity = 0.4]
  (5, 6.5) -- (8,6.5);
  
  \draw [line width = 0.0mm, draw=CornflowerBlue, fill=CornflowerBlue, draw opacity = 0.5, fill opacity=0.5]
       (0,0) -- (0,3) -- (1,3) -- (1,0) -- cycle;

    \draw [line width = 0.0mm, draw=CornflowerBlue, fill=CornflowerBlue, draw opacity = 0.5, fill opacity=0.5]
       (2,1) -- (3,1) -- (3,2) -- (2,2) -- cycle;

    \draw [line width = 0.0mm, draw=CornflowerBlue, fill=CornflowerBlue, draw opacity = 0.5, fill opacity=0.5]
       (4,4) -- (6,4) -- (6,5) -- (5,5) -- (5,6) -- (4,6) -- cycle;

   \draw[line width = 0.5mm] (4,0) -- (4,7);
   \draw[line width = 0.5mm] (0,4) -- (7,4);
   \draw[line width = 0.5mm] (7,8) -- (7,4);
   \draw[line width = 0.5mm] (8,7) -- (4,7);

   \node at (0.5,3.5) {{$1$}}; 
   \node at (1.5,0.5) {{$1$}}; 
   \node at (2.5,2.5) {{$1$}}; 
   \node at (3.5,1.5) {{$1$}};
   \node at (4.5,6.5) {{$1$}};
   \node at (5.5,5.5) {{$1$}};
   \node at (6.5,4.5) {{$1$}};
   \node at (7.5,7.5) {{$1$}};

   \node at (0.5,0.5) {\large{$\times$}}; 
   \node at (4.5,4.5) {\large{$\times$}};    
   
\end{tikzpicture}
\end{center}
\end{minipage}
\begin{minipage}{0.6\textwidth}
\begin{center}
\hspace{-1cm}
\begin{tikzpicture}[->,>=stealth',auto,node distance=1cm]
    \node[] (G) {$\Gvwtt:$};
    \node[shape=circle,draw=black] (1) [right of=G]{1};
    \node[shape=circle,draw=black] (2) [right of=1] {2};
    \node[shape=circle,draw=black] (3) [right of=2] {3};
    \node[shape=circle,draw=black] (4) [right of=3] {4};
    \node[shape=circle,draw=black] (5) [right of=4] {5};
    \node[shape=circle,draw=black] (6) [right of=5] {6} ;
     \node[shape=circle,draw=black] (7) [right of=6] {7};
    \node[shape=circle,draw=black] (8) [right of=7] {8} ;

    \path[every node/.style={font=\sffamily\small}]
    (2) edge node [right] {} (3)
    (3) edge node [right] {} (4)
    (5) edge node [right] {} (6)
    (6) edge node [right] {} (7);
    \draw[->] (2) to[out=45, in=135, looseness=1] (4);
    \draw[->] (5) to[out=45, in=135, looseness=1] (7);
    \draw[->] (5) to[out=45, in=135, looseness=1] (8);
  
\end{tikzpicture}
\end{center}
\end{minipage}

\end{example}

\subsubsection{Cycles of $\Gvwt$}\label{subsec: cycles of Gvw}
\noindent We now consider another approach to study the complexity of $\Nvw$; and the cycles of the graph $\Gvwt$. From now on, when talking about a cycle of a directed graph, we mean a cycle in the underlying undirected graph. We can then rewrite:
\begin{align*}
\dim \Nvw &= |\Dc(v)| - |\Dc(w)| \\
&= \# (\text{nonzero } z_{ij} ) + \# (\text{unexpected zeros}) - |\Dc(w)|.\\
\dim \sigma_{v,w} &= |V(\Gvwt)| - \#(\text{connected components of } \Gvwt) \\
&= |E(\Gvwt)| - \cn( \Gvwt )\\
& = \# (\text{nonzero } z_{ij} )- \cn( \Gvwt),
\end{align*}
where, $\nu$ refers to the cyclomatic number as defined in Definition~\ref{defn:cycnumb}.
The complexity $\Nvw$ is given by
$$ \cn (\Gvwt) - |\Dc(w)| + \# (\text{unexpected zeros}).$$

Therefore, let us now take a closer look at the structure of cycles in $\Dc(v)$. To obtain a cycle, a path in $\Zg$ must begin and end at the same coordinate. Steps that are not allowed within a cycle are to move within the same row or column of $\Zg$ twice (including crossing over a $1$ from that row or column) or to cross over the same $1$ twice. Namely, every combination of steps except for the ones containing the following eight combinations that starts and ends in the same coordinate makes a cycle. 

\begin{align*}
z_{i_1 j_1} \stackrel{r}{\dash} z_{i_2 j_2} \stackrel{r}{\dash}  z_{i_3 j_3} \hspace{1.7cm}&
z_{i_1 j_1} \stackrel{r}{\dash} z_{i_2 j_2} \stackrel{_{\searrow}}{\dash}  z_{i_3 j_3} \\
z_{i_1 j_1} \stackrel{c}{\dash} z_{i_2 j_2} \stackrel{c}{\dash}  z_{i_3 j_3} \hspace{1.7cm} &
z_{i_1 j_1} \stackrel{_{\nwarrow}}{\dash} z_{i_2 j_2} \stackrel{r}{\dash}  z_{i_3 j_3}\\
z_{i_1 j_1} \stackrel{_{\nwarrow}}{\dash} z_{i_2 j_2} \stackrel{_{\searrow}}{\dash}  z_{i_3 j_3} \hspace{1.7cm} & z_{i_1 j_1} \stackrel{_{\swarrow}}{\dash} z_{i_2 j_2} \stackrel{c}{\dash}  z_{i_3 j_3} \\
z_{i_1 j_1} \stackrel{_{\searrow}}{\dash} z_{i_2 j_2} \stackrel{_{\nwarrow}}{\dash}  z_{i_3 j_3} \hspace{1.7cm} & 
z_{i_1 j_1} \stackrel{c}{\dash} z_{i_2 j_2} \stackrel{_{\nwarrow}}{\dash}  z_{i_3 j_3} 
\end{align*}
Of course, these combinations of steps are also not allowed to appear at the beginning and end of the path.
In addition to studying the complexity of the KL variety, following \cite[Proposition 4.3]{gitler2010ring}, if $\Nvw$ is toric, then the cycles determine a set of binomial generators of the defining ideal.

\begin{example}\label{ex: paths and cycles}
Consider $v=324651$ and $w=563421$. The opposite Rothe diagram $\Dc(v)$ and the graph $\Gvwt$ are shown below. Then, for example, the path
$z_{41} \stackrel{c}{\dash} z_{51} 
\stackrel{r}{\dash} z_{53}
\stackrel{_{\nwarrow}}{\dash} z_{41}$, which is highlighted green in $\Gvwt$, is a cycle, but the path
$z_{42} \stackrel{c}{\dash} z_{52} 
\stackrel{c}{\dash} z_{62}
\stackrel{r}{\dash} z_{63}
\stackrel{_{\nwarrow}}{\dash} z_{42}$, which is highlighted blue, is not. \\

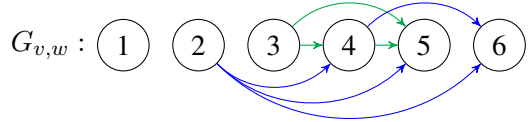
\begin{minipage}{0.35\textwidth}
\begin{center}
\hspace{2cm}
\begin{tikzpicture}[scale = 0.5]
  \draw[step=1.0,black,thin] (0,0) grid (6,6);
 
    \node at (0.5,3.5) {{$1$}}; 
   \node at (1.5,4.5) {{$1$}}; 
   \node at (2.5,2.5) {{$1$}}; 
   \node at (3.5,0.5) {{$1$}};
   \node at (4.5,1.5) {{$1$}};
   \node at (5.5,5.5) {{$1$}};

  \draw [line width = 0.0mm, draw=CornflowerBlue, fill=CornflowerBlue, draw opacity = 0.3, fill opacity=0.3]
       (0,0) -- (0,3) -- (2,3) -- (2,2) -- (3,2) -- (3,0) -- cycle;

  \draw[line width = 0.7mm, draw=Gray, draw opacity = 0.4]
  (0.5, 4) -- (0.5,6);
  \draw[line width = 0.7mm, draw=Gray, draw opacity = 0.4]
  (1.5, 5) -- (1.5,6);
  \draw[line width = 0.7mm, draw=Gray, draw opacity = 0.4]
  (2.5, 3) -- (2.5,6);
  \draw[line width = 0.7mm, draw=Gray, draw opacity = 0.4]
  (3.5, 1) -- (3.5,6);
  \draw[line width = 0.7mm, draw=Gray, draw opacity = 0.4]
  (4.5, 2) -- (4.5,6);
  \draw[line width = 0.7mm, draw=Gray, draw opacity = 0.4]
  (5.5, 6) -- (5.5,6);
  %right
  \draw[line width = 0.7mm, draw=Gray, draw opacity = 0.4]
  (4, 0.5) -- (6,0.5);
  \draw[line width = 0.7mm, draw=Gray, draw opacity = 0.4]
  (5, 1.5) -- (6,1.5);
  \draw[line width = 0.7mm, draw=Gray, draw opacity = 0.4]
  (3, 2.5) -- (6,2.5);
  \draw[line width = 0.7mm, draw=Gray, draw opacity = 0.4]
  (1, 3.5) -- (6,3.5);
  \draw[line width = 0.7mm, draw=Gray, draw opacity = 0.4]
  (2, 4.5) -- (6,4.5);
  \draw[line width = 0.7mm, draw=Gray, draw opacity = 0.4]
  (6, 5.5) -- (6,5.5);

    \node at (0.5,2.5) {\small \textcolor{Blue}{$z_{41}$}}; 
    \node at (0.5,1.5) {\small \textcolor{Blue}{$z_{51}$}};
    \node at (0.5,0.5) {\small \textcolor{Blue}{$0$}}; 
    \node at (1.5,2.5) {\small \textcolor{Blue}{$z_{42}$}}; 
    \node at (1.5,1.5) {\small \textcolor{Blue}{$z_{52}$}};
    \node at (1.5,0.5) {\small \textcolor{Blue}{$z_{62}$}}; 
    \node at (2.5,0.5) {\small \textcolor{Blue}{$z_{63}$}}; 
    \node at (2.5,1.5) {\small \textcolor{Blue}{$z_{53}$}};
 
\end{tikzpicture}
\end{center}
\end{minipage}
\begin{minipage}{0.6\textwidth}
\begin{center}
\hspace{-1.5cm}
\begin{tikzpicture}[->,>=stealth',auto,node distance=1cm]
    \node (G) {$\Gvwt:$};
    \node[shape=circle,draw=black] (1) [right of =G] {1};
    \node[shape=circle,draw=black] (2) [right of=1] {2};
    \node[shape=circle,draw=black] (3) [right of=2] {3};
    \node[shape=circle,draw=black] (4) [right of=3] {4};
    \node[shape=circle,draw=black] (5) [right of=4] {5};
    \node[shape=circle,draw=black] (6) [right of=5] {6} ;

    \path[every node/.style={font=\sffamily\small}];
    \draw (3) edge[Green] (4);
    \draw (4) edge[Green] (5);
    \draw[->, Green] (3) to[out=45, in=135, looseness=1] (5);
    \draw[->, blue] (4) to[out=45, in=135, looseness=1] (6);
    \draw[->, blue] (2) to[out=-45, in=225, looseness=1] (5);
    \draw[->, blue] (2) to[out=-45, in=225, looseness=1] (6);
    \draw[->, blue] (2) to[out=-45, in=225, looseness=1] (4);

\end{tikzpicture}
\end{center}
\end{minipage}

\end{example}

\noindent If there are no unexpected zeros in $\Sigma_v$, 
one knows, a priori, which coordinates $z_{ij}$ appear in $\Zg$ that can contribute to a cycle in $\Gvwt$. Hence, we are able to identify cycles just by looking at the opposite Rothe diagram.

\begin{lemma}\label{contains cycle}
Let $v,w \in S_n$ be such that there are no unexpected zeros in  $\Sigma_v$. Then $\Gvwt$ contains a cycle if and only if there are $(i,j),(k,l) \in D^{\circ}(v)$ such that $i<k$ and $j<l$.  

\begin{proof}

    For the forward implication, assume, for contradiction, that $\Dc(v)$ contains no points $(i,j),(k,l)$ with $i<k,j<l$. Without these, we do not obtain a cycle since the only combinations of steps that a path in $\Zg$ can consist of are $z_{i_1 j_1} \stackrel{r}{\dash} z_{i_2 j_2} \stackrel{c}{\dash}  z_{i_3 j_3}$, where either $j_2 < j_1,i_3 > i_2$ or $j_2>j_1, i_3<i_2$.

    Now for the backward implication, assume that there exist $z_{ij},z_{kl} $ in $\Zg$ such that $i<k$ and $j<l$ of minimal distance. Since there are no unexpected zeros, the matrix $\Zg$ has the nonzero coordinate $z_{kj}$ at position $(k,j)$. Assume for contradiction that $\Zg$ has a zero in position $(i,l)$. By the construction of $\Dc(v)$, this implies a $1$ in $\Zg$ either in position $(q,l)$, where $i<q<k$, or in position $(i,p)$, where $j<p<l$. Wlog, we can assume that there exists a $1$ in such a position $(q,l)$. Again by construction of $\Dc(v)$, this implies that $(q,j)\in \Dc(v)$, which contradicts the minimality of the distance between $(i,j)$ and $(k,l)$. We conclude that $\Zg$ either has a $1$ or the coordinate $z_{il}$ in position $(i,l)$. In the first case, 
    $z_{ij} \stackrel{c}{\dash} z_{kj} \stackrel{r}{\dash} z_{kl}\stackrel{_{\nwarrow}}{\dash} z_{ij}$ is a cycle, in the second case, 
    $z_{ij} \stackrel{c}{\dash} z_{kj} \stackrel{r}{\dash} z_{kl}\stackrel{c}{\dash} z_{il} \stackrel{r}{\dash} z_{ij}$ is.   
\end{proof}
\end{lemma}

It is natural to ask if we can also deduce the cyclomatic number in a similar way, namely by counting the pairs of coordinates as in Lemma \ref{contains cycle} that are of a minimal distance. 
For $(i,j),(k,l)\in \Dc(v)$ with $i<k,j<l$ define the distance 
$d((i,j),(k,l)) := (k-i) + (l-j)$.
Define a set $P_v \subset \Dc(v) \times \Dc(v)$ by adding elements recursively as follows:
\begin{itemize}
\item  $((i,j),(i+1,j+1)) \in P_v$ if $(i,j),(i+1,j+1) \in \Dc(v)$ 
\item $((i,j),(k,l)) \in P_v$ where $d((i,j),(k,l)) = d$ if $(i,j),(k,l) \in \Dc(v)$, $i<k, j<l$ and there is not already $((i,j),(k',l'))\in P$ with $d((i,j),(k',l')) < d$ or $((i',j'),(k,l))\in P_v$ with $d((i',j'),(k,l)) <d$.
\end{itemize}
We first show that this recursion is well-defined. Let $(i,j),(i',j'),(k,l) \in \Dc(v)$ with $i<i'<k$, $j' < j < l$ and $d((i,j),(k,l)) = d((i',j'),(k,l)) = d$. Then neither $((i,j),(k,l))$ nor $((i',j'),(k,l))$ are in $P_v$. This is due to the fact that if $(i',j)\in \Dc(v)$, then $d((i',j),(k,l)) < d$. Otherwise, there exists $j' <q<j$ with $v(q) = i'$. However, then, $(k,q)\in \Dc(v)$ and $d((i',j'),(k,q)) < d$. Similar arguments hold for $(i,j),(k,l),(k',l')$ with $i<k<k'$, $j<l'<l$ and $d((i,j),(k,l))=d((i,j),(k',l'))$.

\begin{theorem} \label{thm: cyclomatic number from pairs}
Let $v,w \in S_n$ be such that there are no unexpected zeros in  $\Sigma_v$. Then $ \cn(\Gvwt) = |P_v|$.
\begin{proof}
    Since the cyclomatic number of a graph is the sum of the cyclomatic numbers of its connected components, consider a single fixed connected component $G$ of $\Gvwt$ with more than one vertex. The edges in $G$ correspond to elements in $\Dc(v)$ that lie in a minimal subset $D$ of $\Dc(v)$ and in a certain minimal set of rows $R$ and minimal set of columns $C$. 
    It is important to note that the connected components of $D$ are {Young diagrams} in French notation and that within the subset of rows in $R$ that a connected component lies in, the further left a column, the ``higher'' it is in $D$. More precisely, if $(i,j)\not\in D$, then there also exists no $(i,k)\in D$ with $k>j$. 
    Define 
    \begin{align*}
        M & := \{ (k,l) \in D \mid \exists (i,j)\in D: ((i,j),(k,l))\in P_v \} ,\\
        L & := \{ (k,l) \in D \mid \forall j < l: (k,j) \not\in D \}, \\
        S & := \{ (k,l) \in D \mid v(l) \not\in R \text{ and } k \text{ minimal s.t. } \exists j<l: (k,j) \in D \}.
    \end{align*}
    The set $M$ describes the lower right elements of pairs in $P_v$, the set $L$ describes the leftmost elements of $D$ of row $k$ and the set $S$ describes the highest points in $D$ that are not the leftmost element in their row for which the $1$ in $Z^{(v)}$ is not in a row in $R$.
    Notice that $L \cap S = \emptyset$ and that $|R| = |L|$. Also, we have that  
    \begin{align*}
    |V(G)| = |R| + |C| - |\{ j \in C \mid v(j) \in R \} 
     = |L| + |S| + 1,    
    \end{align*}    
    Here, the $+1$ comes from the leftmost column $j \in C$ that has $v(j) \not\in R$ but contains no element in $S$.
    We will show that $D \setminus M = L \cup S$. Then,
    \begin{align*}
    |P_v| = |M| = |D| -|L| - |S| 
        = |E(G)| - |V(G)| +1 
         = \cn(G).
    \end{align*} 
\noindent    
\underline{$L \cup S \subset D \setminus M$}: It is clear that $L \subset D \setminus M$. Let $(k,l) \in S$. First of all we notice that then $v(l) >i$ for all $(i,j) \in D$ with $j <l$, since otherwise, if $v(l)<i$, then $(v(l),j)\in D$, contradicting $v(l) \not\in R$. So now we need to show that $(k,l) \not\in M$. If there exists no $(i,j)\in D$ with $i<k$ and $j<l$, then we are done. Assume that it does exist and consider such an $(i,j)$. %We know that it must be $i < v(l)$.
If there exists $q \in C$ with $q < l$, such that $v(q) = i$, then $(k,q) \in D$, $i<k$, $j <q$, and $d((i,j),(k,q)) < d((i,j),(k,l))$. However, if there exists no such $q \in C$, then, since $v(l)>i$, we have that $(i,l)\in D$, which contradicts $(k,l) \in S$.

\noindent
\underline{$D \setminus (L \cup S) \subset M$}: Let $(k,l) \in D \setminus (L \cup S)$. Then, there are two cases that we can differentiate into. \\
\textbf{1}: $v(l) \in R$. We may assume that $k$ is minimal, as otherwise $((k-1,l-1),(k,l))\in M$. Since $v(l)=i$ for some $i \in R$, there also exists $j <l$ with $(i,j)\in D$. Take such a $j$ to be maximal w.r.t. this property. 
We show that for any element in $D$ in between $(i,j)$ and $(k,l)$, we can find another element in $D$ that is north-west of it and closer than $(i,j)$. Then, at the very least, $((i,j),(k,l))\in P_v$.
Suppose that there exist $i<p\leq k$, $j <q \leq l$ with $(p,q)\in D$, such that $d((i,j),(p,q)) < d((i,j),(k,l))$. We have that $q \neq l$, by minimality of $k$; and $p-1 \neq i$, by maximality of $j$ and the fact that $v(l) = i$. We may assume $p$ to be minimal: It is $(p,j)\in D$, so for any non-minimal $p'$, it is $d((p,j),(p',q))<d((i,j),(p',q))$. If $v(q)>i$, then $(i,q)\in D$, contradicting the maximality of $j$. However, if $v(q) <i$, then $(v(q),j)\in D$ and $d((v(q),j),(p,q)) < d((i,j),(p,q))$. \\
\textbf{2}: It is $v(l) \not\in R$ and there exists $i<k$ such that $(i,l)\in D$ and such that there exists $j <l$ with $(i,j)\in D$. Choose $i$ and $j$ maximal with this property. 
\end{proof}
\end{theorem}

\noindent The proof works the same way if instead of considering  unexpected zeros we considered actual zeros. 

\begin{proposition} 
Let $v,w \in S_n$ be such that there are no (actual) zeros in  $\Sigma_v$. Then $ \cn(\Gvwtt) = |P_v|$.
\end{proposition}

\noindent For both cases, we may use the information about cycles in the corresponding graph to deduce the complexity of the KL variety under the usual torus action.

\begin{corollary}\label{cor: toric unexp zeros}
   Let $v,w \in S_n$ be such that there are no (actual or unexpected) zeros in  $\Sigma_v$. Then 
   the complexity of $\Nvw$ is given by $|P_v| - |D^{\circ}(w)|$ and, in particular,
   $\Nvw$ is toric if and only if $|P_v| = |D^{\circ}(w)|$. 
\end{corollary}

\begin{example}
Consider $v= 423516$ and $w=642315$. Then, $\Sigma_v$ has unexpected zeros but no actual zero coordinates. The opposite Rothe diagram $\Dc(v)$ is shown below, with a line between two elements in $P_v$.
We have that $|\Dc(w)|= |P_v|= \cn(\overline{G_{v,w}}) = 4$. Also, $|\Dc(v)|=9$ and $\dim \sigma_{v,w}=6-1=5$. Thus, indeed, $\mathcal N_{v,w}$ is toric. Following \cite[Proposition 4.3]{gitler2010ring}, the binomial generators are $z_{52}- z_{32}z_{53}$,  $z_{62}-z_{52}z_{64}$,  $z_{53}z_{64}-z_{63}$ and  $z_{61}-z_{51}z_{64}$.\\

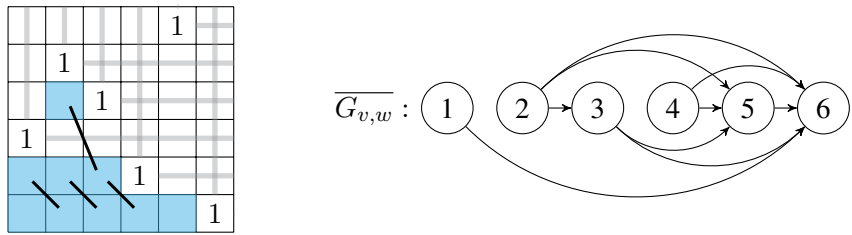
\begin{minipage}{0.35\textwidth}
\begin{center}
\hspace{2cm}
\begin{tikzpicture}[scale = 0.5]

  \draw[step=1.0,black,thin] (0,0) grid (6,6);

    \node at (0.5,2.5) {{$1$}}; 
   \node at (1.5,4.5) {{$1$}}; 
   \node at (2.5,3.5) {{$1$}}; 
   \node at (3.5,1.5) {{$1$}};
   \node at (4.5,5.5) {{$1$}};
   \node at (5.5,0.5) {{$1$}};

  \draw [line width = 0.0mm, draw=CornflowerBlue, fill=CornflowerBlue, draw opacity = 0.5, fill opacity=0.5]
       (0,0) -- (0,2) -- (3,2) -- (3,1)  -- (5,1) -- (5,0) -- cycle;
   \draw [line width = 0.0mm, draw=CornflowerBlue, fill=CornflowerBlue, draw opacity = 0.5, fill opacity=0.5]
       (1,3) -- (2,3) -- (2,4) -- (1,4)  --cycle;

  \draw[line width = 0.7mm, draw=Gray, draw opacity = 0.4]
  (0.5, 3) -- (0.5,6);
  \draw[line width = 0.7mm, draw=Gray, draw opacity = 0.4]
  (1.5, 5) -- (1.5,6);
  \draw[line width = 0.7mm, draw=Gray, draw opacity = 0.4]
  (2.5, 4) -- (2.5,6);
  \draw[line width = 0.7mm, draw=Gray, draw opacity = 0.4]
  (3.5, 2) -- (3.5,6);
  \draw[line width = 0.7mm, draw=Gray, draw opacity = 0.4]
  (4.5, 6) -- (4.5,6);
  \draw[line width = 0.7mm, draw=Gray, draw opacity = 0.4]
  (5.5, 1) -- (5.5,6);

  \draw[line width = 0.7mm, draw=Gray, draw opacity = 0.4]
  (6, 0.5) -- (6,0.5);
  \draw[line width = 0.7mm, draw=Gray, draw opacity = 0.4]
  (4, 1.5) -- (6,1.5);
  \draw[line width = 0.7mm, draw=Gray, draw opacity = 0.4]
  (1, 2.5) -- (6,2.5);
  \draw[line width = 0.7mm, draw=Gray, draw opacity = 0.4]
  (3,3.5 ) -- (6,3.5);
  \draw[line width = 0.7mm, draw=Gray, draw opacity = 0.4]
  (2, 4.5) -- (6,4.5);
  \draw[line width = 0.7mm, draw=Gray, draw opacity = 0.4]
  (5, 5.5) -- (6,5.5);

   \draw[line width = 0.4mm, draw=black]
   (1.35, 0.65) -- (0.65,1.35);
   \draw[line width = 0.4mm, draw=black]
   (2.35, 0.65) -- (1.65,1.35);
   \draw[line width = 0.4mm, draw=black]
   (3.35, 0.65) -- (2.65,1.35);
   \draw[line width = 0.4mm, draw=black]
   (1.65, 3.35) -- (2.35,1.65);
  
\end{tikzpicture}  
\end{center}
\end{minipage}
\begin{minipage}{0.6\textwidth}
\begin{center}
\hspace{-1.5cm}
\begin{tikzpicture}[->,>=stealth',auto,node distance=1cm]
    \node (G) {$\Gvwtt:$};
    \node[shape=circle,draw=black] (1) [right of =G] {1};
    \node[shape=circle,draw=black] (2) [right of=1] {2};
    \node[shape=circle,draw=black] (3) [right of=2] {3};
    \node[shape=circle,draw=black] (4) [right of=3] {4};
    \node[shape=circle,draw=black] (5) [right of=4] {5};
    \node[shape=circle,draw=black] (6) [right of=5] {6} ;

    \path[every node/.style={font=\sffamily\small}]
    (2) edge node [right] {} (3)
    (4) edge node [right] {} (5)
    (5) edge node [right] {} (6);
    \draw[->] (2) to[out=45, in=135, looseness=1] (5);
    \draw[->] (2) to[out=45, in=135, looseness=1] (6);
    \draw[->] (4) to[out=45, in=135, looseness=1] (6);
    \draw[->] (3) to[out=-45, in=225, looseness=1] (5);
    \draw[->] (3) to[out=-45, in=225, looseness=1] (6);
    \draw[->] (1) to[out=-45, in=225, looseness=1] (6);
  
\end{tikzpicture}
\end{center}
\end{minipage}

\end{example}

\noindent In particular, Corollary \ref{cor: toric unexp zeros} tells us the complexity of KL varieties $\Nvw$ for which $\Dc(w) = \emptyset$.
 
\begin{corollary}
    Let $w=w_0$ be the permutation of maximal length in $S_n$. Then $\Nvw$ is toric if and only if $\Dc(v)$ contains no $(i,j),(k,l)$ with $i<k$ and $j<l$.
\end{corollary}

\noindent In general, it is known from 
\cite[Prop 6.4]{lee21toricbruhatintpolytopes}
and
\cite[Cor 4.16]{donten2021complexity}, 
that any positive integer is the complexity of some KL variety.
As an off-topic addition, Proposition \ref{cor: toric unexp zeros} and some results from Section \ref{subsec: glueing toric intervals} help us determine which complexities we can achieve for different KL varieties. 

\begin{proposition}
\begin{enumerate}
    \item Fix an integer $n$. Then, for every integer $k$ between $0$ and ${{n-1} \choose 2}$, there exist permutations $v,w \in S_n$ such that $\Nvw$ has complexity $k$. 
    \item Fix a permutation $v$. Then, for every integer $k$ between $0$ and $|P_v|$, there exists a permutation $w \in S_n$ such that $\Nvw$ has complexity $k$. 
    \item Fix a permutation $w$. Then, for every integer $k$ between $0$ and $\ell(w) - |\{ s_i | s_i \leq w \} |$, there exists a permutation $v \in S_n$ such that $\Nvw$ has complexity $k$. 
\end{enumerate}
    \begin{proof}
        By Proposition \ref{prop: complexity change}, the KL varieties coming from subintervals of an interval $[v,w]$ achieve each complexity between $0$ and the complexity of $\Nvw$. For a fixed $n$, the complexity of $\mathcal N_{id, w_0}$ is $\frac{n(n-1)}{2} - (n-1)= {{n-1} \choose 2}$; this proves the first statement. To prove the second, consider a fixed $v$. Then, the complexity of $\mathcal N_{v,w_0}$ is $|P_v|$ by Theorem \ref{thm: cyclomatic number from pairs}. Finally, for a fixed $w$, the complexity of $\mathcal N_{\textrm{id},w}$ is $\ell(w) - |\{ s_i | s_i \leq w \} |$ by \cite[Theorem 5.8]{donten2021complexity}. 
    \end{proof}
\end{proposition}

\subsection{Complexity of other families}

\noindent We now want to apply similar techniques and provide some statements about the complexity of certain families of KL varieties in terms of the cyclomatic number of $\Gvwt$. We start out by looking at KL varieties where $\Dc(w)$ is a single rectangle i.e., when there exists a unique Fulton determinantal condition.

\begin{proposition}\label{prop: rectangle comp}
    Consider $w\in S_n$, such that $\Dc(w)$ consists of a single rectangle of size $l \times k$. Denote by (a,b) the unique element in $\Ess(w)$ and $m := r_w(a,b)$. Also, denote by $p$ the complexity of $\Nvw$.
    \begin{enumerate}
        \item If $\rk_v(a,b) < m$ then $p= \cn(\Gvwt) - |\Dc(w)| =  |P_v|- kl$. 
        \item If $\rk_v(a,b) =m$ then $p = \cn(\Gvwt)$ and $\Nvw$ is toric if and only if $\Gvwt$ is a forest.
    \end{enumerate}
    \end{proposition}
    \begin{proof}
     First, for $c,d \in [n]$, denote by $\rk_m(c,d)$, the maximal subrank $\rk_u(c,d)$ of matrices $u \in S_n$ with $\rk_u(a,b) \leq m$. We claim that for all such $c,d$, it is $\rk_w(c,d) = \rk_m(c,d)$. Notice that $\rk_m(c,d) \leq \min \{ n-c+1,d \}$. Meanwhile, $w$ must be of the following form:
     $$
    w= (n, \ldots, n-m+1, \, n-m-l, \ldots, n-m-l-k+1, \, n-m, \ldots, n-m-l+1, \, n-m-l-k, \ldots, 1).
$$
     If $c<a$ and $d>b$, then $\rk_w(c,d) = \min \{n-c+1,d \}$. If $c \geq a$ and $d \leq b$, then both $\rk_m(c,d)$ and $\rk_w(c,d)$ are given by $\min \{ m, n-c+1, d \}$.
     If wlog $c \geq a$ and $d>b$, then both $\rk_m(c,d)$ and $\rk_w(c,d)$ are given by $\min \{ m+(d-c), n-c+1 \}$.
     Hence, $w$ dominates all permutations $u$ with $\rk_u(a,b) \leq m$ in the Bruhat order. 
     
    \noindent
    Recall that the complexity of $\Nvw$ is given by $\cn(\Gvwt) - |\Dc(w)| + \# (\textrm{unexpected zeros})$ and that the unexpected zeros of $\Sigma_v$ are the $(i,j) \in \Dc(v)$ such that $t_{v(j),i}v \not\leq w$. We study the change of $\rk_v(a,b)$ under multiplication of $v$ with $t_{v(j),i}$. In the permutation matrix of $v$, left multiplication by $t_{v(j),i}$ swaps rows $v(j)$ and $i$. More precisely, it interchanges the $1$s in the permutation matrix at positions $(v(j),j)$ and $(i, v^{-1}(i))$ with $1$s at positions $(i,j)$ and $(v(j), v^{-1}(i))$. As a result, $\rk_{t_{v(j),i}v}(a,b) = \rk_v(a,b) +1$ if and only if $i \geq a$, $j \leq b$, $v(j)<a$ and $v^{-1}(i) >b$ (or in other words, if $(i,j) \in \SW(w)$ and . 
    the $1$s in that row and column lie outside of $\SW(w)$.
    Otherwise, the subrank stays the same. 

    \begin{enumerate}
        \item Assume $\rk_v(a,b) <m$. Then, for all $(i,j) \in \Dc(v)$, we see that $\rk_{t_{v(j),i}v}(a,b) \leq m$ and therefore $t_{v(j),i}v \leq w$. Hence, in this case, there are no unexpected zeros and the complexity of $\Nvw$ is given by $\cn(\Gvwt) - |\Dc(w)|$.
    \item Assume $\rk_v(a,b) =m$. Then, the only way to obtain $t_{v(j),i}v \not\leq w$ is if $(i,j) \in \SW(w)$, $v(j) <a$ and $v^{-1}(i) >b$. Since $\rk_v(a,b)= m$, there are $l$ rows $i \geq a$ with $v^{-1}(i) >b$ and $k$ columns $j \leq b$ with $v(j) <a$. These tuples $(i,j)$ must lie in the Rothe diagram and they are unexpected zeros. Hence, there are $lk = |\Dc(w)|$ unxepected zeros in $\Sigma_v$ and the complexity of $\Nvw$ is given by $\cn(\Gvwt)$.\qedhere
    \end{enumerate}\end{proof}
\noindent In \cite[Thm 5.14]{donten2021complexity}, the authors address the question of the complexity of $\mathcal N_{s_a,w}$ for some simple reflection~$s_a$. A related question would be to determine the complexity of $\mathcal N_{v,w_0\cdot s_a}$. We generalise this and give a characterisation of the complexity of $\mathcal N_{v,w_0\cdot t}$ for some transposition $t$ in terms of the cyclomatic number of $\Gvwt$.

\begin{proposition}\label{prop: change w for complexity}
Consider $w = w_0\cdot t_{l,k}$ and $v < w$. Denote by $(a_1,b_1)=(n-k+2,l),(a_2,b_2)=(n-l+1,k-1)$ the two elements in $\Ess(w)$ (or the unique element, since these coincide in case of a simple reflection) and by $p$, the complexity of $\Nvw$.
\begin{enumerate}
    \item If $\rk_v(a_1,b_1), \rk_v(a_2,b_2) < l-1$ then $p= \cn(\Gvwt)- |\Dc(w)| = |P_v| - 2(k-l)+1$. 
    \item If $\rk_v(a_1,b_1)= l-1, \rk_v(a_2,b_2) < l-1$ or 
          $\rk_v(a_1,b_1)< l-1, \rk_v(a_2,b_2) =  k-1$
          then $p = \cn(\Gvwt) - |\Dc(w)| + (k-1) = \cn(\Gvwt) +l-k+1$.
    \item If $\rk_v(a_1,b_1) = \rk_v(a_2,b_2) = l-1$ then $p = \cn(\Gvwt)$ and $\Nvw$ is toric if and only if $\Gvwt$ is a forest.
\end{enumerate}
\begin{proof}
    We proceed similarly as in the proof of Proposition \ref{prop: rectangle comp}.
    First, notice that $\rk_w(a_1,b_1)=\rk_w(a_2,b_2) = l-1 =:m$. For $c,d \in [n]$, denote by $\rk_u(c,d)$ the maximal subrank $\rk_u(c,d)$ of matrices $u \in S_n$ with $\rk_u(a_1,b_1), \rk_u(a_2,b_2)\leq m$. Then, for all such $c,d$, it is $\rk_w(c,d)= \rk_m(c,d)$. Denote by $\SW(c,d)$ the set of elements that lie south-west of $(c,d)$.
    Further, $\rk_{t_{v(j),i}v}(a_s,b_s) = \rk_v(a_s,b_s)+1$ if and only if $(i,j)\in \SW(a_s,b_s)$, but $(v(j),j),(i, v^{-1}(i)) \not\in \SW(a_s,b_s)$, otherwise the subrank stays the same.

    \begin{enumerate}
    \item Assume $\rk_v(a_1,b_1), \rk_v(a_2,b_2) < m$. Then still, for all $(i,j) \in \Dc(v)$, $\rk_{t_{v(j),i}\cdot v}(a_s,b_s) \leq m$ and therefore $t_{v(j),i}\cdot v \leq w$. Hence, there are no unexpected zeros in this case.

    \item Wlog, assume $\rk_v(a_1,b_1)= m, \rk_v(a_2,b_2) < m$. We obtain an unexpected zero $(i,j)$ if and only if $(i,j) \in \SW(a_1,b_1)$ but the 1s in $v$ in row $i$ and column $j$ are not south-west. Hence, there are $1 \cdot (n-a_1+1-m) = k-l$ unexpected zeros in this case.

    \item Assume that $\rk_v(a_1,b_1) = \rk_v(a_2,b_2) = m$. Denote $A:= \SW(a_2,b_1)$, $B:= \SW(a_1,b_1) \setminus A$ and $C:= \SW(a_2,b_2) \setminus A$. Say there are $m-q$ 1s in $A$ and $q$ 1s in both $B$ and $C$.
    Inside of $A$ there are 
    \begin{itemize}
    \item one column with 1 outside of $A \cup B$;
    \item $q$ columns with 1 in $B$;
    \item one row with 1 outside of $A \cup C$; and 
    \item $q$ rows with 1 in $C$. 
    \end{itemize}
    Hence, inside $A$ there are 
    \begin{itemize} 
    \item $1 \cdot 1$ element $(i,j)$ with 1 in row $i$ outside $A \cup C$ and 1 in column $j$ outside $A \cup B$; 
    \item $q \cdot 1$ elements $(i,j)$ with 1 in column $j$ outside $A\cup B$ but 1 in row $i$ inside of $C$; and 
    \item $q \cdot 1$ elements $(i,j)$ with 1 in row $i$ ouside of $A \cup C$ but 1 in column $j$ inside of $B$. 
    \end{itemize}
    This sums to $2q+1$ unexpected zeros inside of $A$. Notice that rows in $A$ are also rows in $C$ and columns in $A$ are also columns in $B$.
    Denote by $h_1 = (n-a_1+1)-(m+1)$, the height of $B$ and by $h_2 = b_2 - (m+1)$, the width of $C$. There are $h_1-q$ rows in $B$ with a 1 outside of $A \cup B$ and thus $(h_1-q)\cdot 1$ unexpected zeros in $B$. There are $h_2-q$ columns in $C$ with a 1 outside of $A \cup C$ and thus $(h_2-q)\cdot 1$ unexpected zeros in $C$. In total, we obtain $h_1+h_2+1 = |\Dc(w)|$ unexpected zeros.\qedhere
    \end{enumerate}
\end{proof}
\end{proposition}

\section{Statistical models}\label{sec: statistical models}

\noindent
In this section, we consider applications of matrix Schubert and Kazhdan-Lusztig varieties to two families of
statistical models, each of which is prevalent in the literature of algebraic statistics. 
  
\subsection{Conditional Independence Models}\label{subsec: ci models}
\noindent To begin with, we study the relations between matrix Schubert and Kazhdan-Lustzig varieties  and conditional independence (CI) models. CI models have been well-studied in algebraic staistics and we refer the reader to \cite[Chapter 4]{sullivant2023algebraic} for further details. We say that the random variables $X_A$ and $X_B$ are conditionally independent given $X_C$, if and only if, given a value of $X_C$, the probability distribution of $X_A$ is the same for all values of $X_B$ and the probability distribution of $X_B$ is the same for all values of $X_A$. In other words:
$$\mathbb{P}(X_A,X_B|X_C) = \mathbb{P}(X_A|X_C)\mathbb{P}(X_B|X_C).$$

In particular, in order to draw parallels between CI models and the determinantal varieties we have defined above, we focus on the following case. Let $X$ be an $m$ dimensional Gaussian random vector; $X=(X_1,\ldots,X_m) \sim \mathcal{N}(\mu, \Sigma)$. Here, $\mu$ is the mean vector of the random vector and $\Sigma$ is the covariance matrix, such that the $(i,j)$th entry of $\Sigma$ is the covariance of $X_i$ and $X_j$. Thus, $\Sigma$ is symmetric and positive definite. Let $A$ be any subset of $[m]$. Then, we denote by $X_A$ the subvector of $X$ given by $(X_a)_{a\in A}$. For two sets $A, B \subseteq[m]$, $\Sigma_{A,B}$ is the submatrix of $\Sigma$ with rows indexed by $A$ and columns indexed by $B$. In the remainder of this paper, we will use the following Proposition as the definition of conditional independence. 

\begin{proposition}[{\cite{sullivant2023algebraic}}, Proposition 4.19]
    Let $X \sim \mathcal{N}(\mu, \Sigma)$ and $A,B,C$ be disjoint subsets of $[n]$. Then, the conditional independence statement: $X_A$ is independent of $X_B$ given $X_C$, denoted $X_A\indep X_B | X_C$ or also $A \indep B | C$, holds if and only if
    $$ \mathrm{rank } (\Sigma_{A\cup C, B\cup C})= | C|.$$
\end{proposition}
\noindent Now, since the covariance matrix is positive definite, we can translate this condition into 
$$ \mathrm{rank } (\Sigma_{A\cup C, B\cup C} )\leq | C|.$$
Thus, we naturally associate to the conditional independence statement, the \emph{CI ideal}
$$J_{A\indep B | C} := \langle (|C| + 1) \text{ minors of } \Sigma_{A\cup C, B\cup C} \rangle \subseteq \mathbb{C}[\Sigma].$$
In order to be able to explore the connections between CI varieties and matrix Schubert varieties, we first need to introduce the following subclasses of matrix Schubert varieties.

\subsubsection{Symmetric and lower triangular matrix Schubert varieties}
\noindent In this subsection, we investigate the two subclasses of matrix Schubert varieties as defined below, which were originally introduced in \cite{fink2016matrix}. The definitions we use are analogous to the Fulton conditions in the statement of Theorem \ref{thm:fultonessential}. 

\begin{definition}
    Let $w \in S_n$ be a permutation. 
    \begin{enumerate}
        \item The \emph{symmetric matrix Schubert ideal} $I^{\text{sym}}_w$ is the ideal generated by imposing the Fulton determinantal conditions, as defined by Theorem \ref{thm:fultonessential} on a generic symmetric $n\times n$ matrix. The \emph{symmetric matrix Schubert variety} $\overline{X^{\text{sym}}_w}$ is defined as the variety associated to $I^{\text{sym}}_w$.
         \item The \emph{lower triangular matrix Schubert ideal} $I^{\text{low}}_w$ is the ideal generated by imposing the Fulton determinantal conditions on a generic lower triangular $n\times n$ matrix. The \emph{lower triangular matrix Schubert variety} $\overline{X^{\text{low}}_w}$ is defined as the variety associated to $I^{\text{low}}_w$.
    \end{enumerate}
\end{definition}

\noindent Formally, we can think of these varieties as the usual matrix Schubert variety intersected with the polynomial conditions that make a matrix a symmetric or lower triangular matrix.
Further, we introduce the following notation that will be useful throughout the remainder of this section. We are in the setting of $n\times n$ matrices.
\begin{itemize}
    \item $\Delta^{\textrm{low}} \coloneqq \{ (i,j) \ | \ i,j\in[n], i\geq j \}$,
    \item $\Delta^{\textrm{up}} \coloneqq \{ (i,j) \ | \ i,j\in[n], i< j \}$.
\end{itemize}
\noindent Note that $\Delta^{\textrm{low}}$ contains the elements along the diagonal and anything below it, while $\Delta^{\textrm{up}}$ only contains the squares above the diagonal, and not the diagonal itself.

Similar to the general case, we can write 
$\overline{X_w^{\sym}} = Y_w^{\sym} \times \mathbb C^{d^{\sym}}$, where $Y_w^{\sym}$ is the projection of $\overline{X_w^{\sym}}$ onto the entries of $L(w)$ and its reflection on the diagonal,
and $\mathbb C^{d^{\sym}}$
is isomorphic to the projection onto the entries that do not appear in the Fulton conditions. If $d$ is such that $\overline{X_w} = Y_w \cap \mathbb C^d$, then $d^{\sym} =  |\SW(w)^{\complement} \cap \Delta^{\low}| = d - |\SW(w)^{\complement} \cap \Delta^{\textrm{up}}|$. Here, $\SW(w)^{\complement}$ is the complement of $\SW(w)$ in the whole $n\times n$ grid.
Similarly, we can define a $T \times T$ action of the subvariety $Y_w^{\sym}$ with weight cone
$\sigma_w^{\sym} = \textrm{Cone}(e_i -f_j \mid (i,j) \in L(w) \cap \Delta^{\low})$, where $e_1, \ldots, e_n, f_1, \ldots, f_n$ denote the standard basis of $\mathbb Z^n \times \mathbb Z^n$.
We can also define the corresponding graph $G_w^{\sym}$ on $[n] \amalg [n]$ with an edge $a \rightarrow b^*$ whenever $(a,b) \in L(w) \cap \Delta^{\low}$.
The same can be done for $X_w^{\low}$.

\begin{proposition}
We have that $\dim(\overline{ X^{\low}_w}) = \dim(\overline{ X^{\sym}_w})= \dim(\Xw) - {n\choose2}$.

\begin{proof}
Denote by $G = \{ g_1, \ldots, g_r \}$, the set of determinantal conditions that arise from the essential minors and $H = \{ z_{ij} \mid (i,j) \in \Delta^{\textrm{up}} \}$, such that $\overline {X_w^{\low}} \simeq V( \langle G \rangle + \langle H \rangle )$. Here, $z_{ij}$ denote the entries of a generic $n \times n$-matrix, which are our variables. Fix a diagonal monomial order $<$. Then by \cite[Theorem B]{knutson2005geometry}, $G$ and $H$ are Gröbner bases of the ideals that they span. We claim that $G \cup H$ is a Gröbner basis for the sum of these ideals. All that we now need to show is that the $S$-polynomials $S(g_l, z_{ij})$ have a standard representation with respect to $G \cup H$. This follows from \cite[Section 2.9, Proposition 4]{cox1997ideals}, since the leading monomial of $g_l$ is the product of the diagonal of the corresponding minor, which lies below the overall diagonal and thus contains no $z_{ij}$, for $(i,j) \in \Delta^{\textrm{up}}$. Therefore, $\init_{<}( \langle G \rangle + \langle H \rangle) = \init_{<} \langle G \rangle + \init_{<} \langle H \rangle$ and
\begin{align*}
V( \init_{<} ( \langle G \rangle + \langle H \rangle) = V( \init_{<}(G)) \cap V(z_{ij} \mid (i,j) \in \Delta^{\textrm{up}}) \simeq V_{\low}(\init_{<}(G)).
\end{align*}
Here, $V_{\low}$ refers to the variety in the space of lower triangular matrices associated to $\mathbb C[z_{ij} |i \geq j]$
and the isomorphism comes from the fact that $\init_{<}(G)$ contains no $z_{ij}$ with $(i,j) \in \Delta^{\textrm{up}}$. Then,
$$\codim (V_{\low}(\init_{<}(G))) = \codim (V( \init_{<}(G))) = \codim \langle G \rangle = |\Dc(w)|;$$
and
$$\dim(\overline{ X^{\low}_w}) = \dim V_{\low}(\init_{<}(G)) = n^2 - {n \choose 2} - |\Dc(w)| = \dim X_w - {n \choose 2}. 
$$
To show that the same holds for $\dim(\overline{ X^{\sym}_w})$, one may e.g.\ take the lexicographic term order induced by 
$$z_{11} > z_{12} > \cdots > z_{1n} > z_{21} > z_{22} > \cdots > z_{2n} > \cdots > z_{n1} > \cdots > z_{nn}.$$
Then, for $H' = \{ z_{ji} - z_{ij} \mid (i,j) \in \Delta^{\textrm{up}} \}$, $X_{w}^{\sym} \simeq V( \langle G \rangle + \langle H' \rangle )$ and $\init_{<}(H') = H$, therefore the proof follows just as above.
\end{proof}
\end{proposition}
\noindent This is coherent with \cite[Proposition 3.7]{fink2016matrix}. We immediately conclude the following.
\begin{corollary} 
 $\dim(Y^{\low}_w)=\dim(Y^{\sym}_w)=\dim(Y_w) - \vert \SW(w)\cap\Delta^{\textrm{up}}\vert$.
\begin{proof}
    If $\overline{X_w} = Y_w \cap \mathbb C^d$, then $\dim Y_w^{\sym} = (\dim \overline{X_w} - {n \choose 2}) - (d- |\SW(w)^{\complement} \cap \Delta^{\textrm{up}}|)$. The claim follows since $ {n \choose 2} -|\SW(w)^{\complement} \cap \Delta^{\textrm{up}}| = |\SW(w) \cap \Delta^{\textrm{up}}|$. The same holds for $Y_w^{\low}$.
\end{proof}
\end{corollary}

\noindent Now, our aim is to calculate the complexity of $Y^{\textrm{low}}_w$ and $Y^{\textrm{sym}}_w$. In order to do this, we need to investigate their associated weight cones. 

\begin{theorem}
     The following statement holds for the dimensions of the different weight cones: $$\dim \sigma^{\sym}_{w} = \dim \sigma^{\low}_{w} = \dim \sigma_{w}.$$
    \begin{proof}
        It suffices to show that removing from the weight cone $\sigma_w$ the weights coming from the upper diagonal does not change the number of connected components of $G_w$.
        We note that the connected components of $G_w$ are in one to one correspondence to the connected components of $L(w)$ and that the connected components of $G_w^{\sym}$ and $G_w^{\low}$ are in one to one correspondence to the connected components of $L(w)\cap \Delta^{\textrm{low}}$.  We also note that any two connected components of $L(w)$ are separated by $\dom(w)$. 
        The only way to achieve that a connected piece in $L(w)$ is no longer connected when intersecting with $\Delta^{\low}$ is if the dominant piece $\dom(w)$ intersects the subdiagonal.
        To this end, consider a north-east corner $(a,b)$ of $\dom(w)$. The south-west submatrix $w_{[a,b]}$ of $w$ has height $n-a+1$ and width $b$. Since there are no $1$s in $w_{[a,b]}$, there must be $b$ $1$s north of and $n-a+1$ $1$s east of $w_{[a,b]}$. If $b \geq a$ (i.e.\ if $\dom(w)$ intersects the diagonal), then $(n-a+1)+b >n$ and the placement of $1$s described above is not possible. If $b=a-1$, then $(n-a+1)+b =n$ and there thus are no $1$s strictly north-east of $w_{[a,b]}$. As a consequence, all other elements in $\Dc(w)$ are either strictly more south or strictly more west than $(a,b)$. Hence, any $(c_1,d_1)$ north of and any $(c_2,d_2)$ east of $w_{[a,b]}$ are not in the same connected component of $L(w)$. Therefore, removing $\Delta^{\textrm{up}}$ does not change the number of connected components.
    \end{proof}
    
\end{theorem}

\noindent Thus, the following corollary is immediate. 

\begin{corollary}
If $d$ is the complexity of $Y_w$, then the complexity of $Y^{\textrm{low}}_w$ and $Y^{\textrm{sym}}_w$ are given by $$d- \vert \SW(w)\cap\Delta^{\textrm{up}}\vert.$$
\end{corollary}

\noindent It has been shown in \cite[Theorem 3.14, Theorem 3.15]{donten2021complexity} that there exist $T \times T$-varieties of every complexity $d$ except of complexity $d=1$. For symmetric and lower triangular matrix Schubert varieties, this exception does not exist, as the following Proposition shows.

\begin{proposition}
There exist $T \times T$-varieties $Y_w^{\sym}$ and $Y_w^{\low}$ of complexity $d$ for every $d\in \mathbb N_0$.
\begin{proof}
We know that there exist toric $T \times T$-varieties $Y_{v}$ and
by \cite[Theorem 3.15]{donten2021complexity}, we also know that there exist $Y_{v}$ of complexity $d$ for every $d \geq 2$. If $v \in S_{ m}$, take $n=2m$ and $w = (v_1 + m, \ldots , v_m+m, \,\,\, m, m-1, \ldots, 1)$ such that $\SW(w) \cap \Delta^{\textrm{up}} = \emptyset$ and the complexities of $Y_{v}$ and $Y_w^{\sym}$ are the same. Also, 
$Y^{\sym}_{3412}$ has complexity $1$.
\end{proof}
\end{proposition}

\subsubsection{Matrix Schubert CI models}

Conditional independence ideals that are also matrix Schubert ideals are classified in \cite{fink2016matrix}. We are interested in their complexity, but first state the original result for a single CI statement.

\begin{proposition}[{\cite[Proposition 4.1]{fink2016matrix}}] \label{prop:one CI}
    Let $X \sim \mathcal N (\mu, \Sigma)$ be an $m$-dimensional Gaussian random vector and let $A \indep B \mid C$ be a CI statement. Then the CI ideal $J_{A \indep B \mid C} \subset \mathbb C [\Sigma]$ describes the symmetric matrix Schubert variety $X_w^{\sym}$ if and only if one of the following two conditions is satisfied:
    \begin{enumerate}
        \item $A= [1,i]$, $B= [j,n]$ and $C= \emptyset$ for some $i<j$ and \\
        $w= (j-1, \ldots, j-i, \,\,\, n, \ldots, j, \,\,\, j-i-1, \ldots, 1)$
        \item $A=[1,i]$, $B=[j,n]$ and $C=[i+1,j-1]$ for some $i<j-1$ and \\
        $w = (n, \ldots, n-j+i+2, \,\,\, i , \ldots, 1, \,\,\, n-j+i+1, \ldots, i+1)$
    \end{enumerate}
\end{proposition}
\noindent We can now state the complexity of $Y^{\sym}_w$ that is associated to a CI statement, in terms of the $|A|,|B|$ and $|C|$ given by the statement.
\begin{proposition}\label{prop: complexity Mschu}
    Let $w \in S_n$ be such that $\overline{X^{\sym}_w}$ is a CI variety coming from a CI statement $A \indep B \mid C$, as in Proposition~\ref{prop:one CI}. If $C= \emptyset$, then $Y^{\sym}_w$ has complexity $0$. Otherwise, $Y^{\sym}_w$ has complexity $(|C|-1)(n-1-\frac{|C|}{2}) $.
    \begin{proof}
    If $C = \emptyset$ then $\dim Y^{\sym}_w = |L'(w)| = 0$ and $\dim \sigma^{\sym}_w = 0$. Otherwise $A \cup B \cup C = [n]$ and the rectangle $\SW(w)$ has an overlap over the diagonal of height and width $(|A|+|C|)-|A|=|C|$. Thus
        $$\dim Y^{\sym}_w = |L'(w)| - |\SW(w) \cap \Delta^{\textrm{up}} |
                          = |B||C| + |C|^2 + |A||C| - \frac{(|C|-1)|C|}{2}
                          = n |C| -\frac{(|C|-1)|C|}{2},
                          $$        
        $$\dim \sigma^{\sym}_w = |A \cup C| + |B \cup C| -1 = n+ |C|-1. $$ \qedhere  \end{proof}
\end{proposition}

\begin{corollary}
    Let $w \in S_n$ such that $\overline{X^{\sym}_w}$ is a CI variety coming from a CI statement $A \indep B \mid C$ as in Proposition~\ref{prop:one CI}. Then $Y^{\sym}_w$ is toric if and only if $|C| \in \{0,1\}$.
    
\end{corollary}

\noindent In this case, a single conditional independence statement $A \indep B \mid C$ can be represented in two distinct ways. When $|C|=0$, it arises from the global Markov property of the union of a complete graph on the vertices in $A$ and a complete graph on the vertices in $B$. When $|C|=1$, it can instead be represented by a block graph i.e.\ 1-clique sum of complete graphs in which the two parts are connected through the vertex $C$. The toricness (not necessarily w.r.t. the usual torus action) of the former case is immediate, while in the latter case it follows directly from the irreducibility of $Y^{\sym}_w$ and \cite[Theorem 1.1]{biaggi2025binomiality} and \cite[Theorem 5]{misra2021gaussian}.

\subsubsection{Kazhdan-Lusztig CI models}

\noindent Following the idea behind Proposition \ref{prop:one CI}, one can also find a correspondence between families of KL varieties and families of CI varieties. 

\begin{lemma}\label{KL CI}
    Let $X \sim \mathcal N (\mu, \Sigma)$ be an $m$-dimensional Gaussian random vector and let $A \indep B \mid C$ be a CI statement. Then, the variety of the CI ideal $J_{A \indep B \mid C} \subset \mathbb C [\Sigma]$ is isomorphic to a KL variety $\Nvw$ if
$$
v  = (n-m, \ldots, n, \,\,\, n-m-1, \ldots, 1),
$$
and one of the following two conditions is satisfied:

\begin{enumerate}
\item $w = (n-l, \ldots, n-l-k+1, \,\,\, n, \ldots, n-l+1, \,\,\, n-l-k, \ldots, 1)$
where $l+k \leq m $ and $A=[1,k]$, $B=[m-l+1,m]$, $C= \emptyset$;

\item $w= (n, n-1-t, \ldots, n-t-s , \,\,\, n-1, \ldots, n-t, \,\,\, n-t-s-1, \ldots, 1) $
where $l+k=m+1$ for $k:= 1+s$ and $l:= 1+t$ and 
$A=[1,k-1]$, $B=[k+1,m]$, $C=\{k\}$.
\end{enumerate}
\begin{proof}
The variables $z_{ij}$ of $\Zg$ lie below the diagonal of the southwest $(m+1) \times (m+1)$ submatrix of $\Zg$. We identify this triangular part of $\Zg$ with the lower left part of the symmetric covariance matrix $\Sigma$, where the diagonal of $\Sigma$ comes from the subdiagonal of the submatrix. If $l+k \leq m$, then the position of $\SW(w)$ in $\Sigma$ does not intersect its diagonal. Therefore, for $A$, $B$, and $C$, as given in (1), the CI statement imposes the same rank 0 conditions on $\Sigma$ as $\Ess(w)$. If $l+k = m+1$, then the position of $\SW(w)$ in $\Sigma$ intersects its diagonal with an overlap of one, and therefore for $A$, $B$ and $C$, as given in (2), the CI statement imposes the same rank one conditions on $\Sigma$ as $\Ess(w)$. 
\end{proof}

\begin{figure}[H]
\begin{minipage}[c]{0.33\textwidth}
\begin{tikzpicture}[scale = 0.45]

\node at (-1.5,2.5) {{$v =$}};

    \draw[line width = 0.3mm] (0.2,-0.2) -- (-0.2,-0.2) -- (-0.2,5.2) -- (0.2,5.2);
    \draw[line width = 0.3mm] (4.8,5.2) -- (5.2,5.2) -- (5.2,-0.2) -- (4.8,-0.2);

    \draw[line width = 0.1mm, Gray] (3,0) -- (3,5);
    \draw[line width = 0.1mm, Gray] (0,3) -- (5,3);

   \draw [line width = 0.0mm, draw=magenta!60, fill=magenta!60, draw opacity = 0.5, fill opacity=0.5]
       (0,2) -- (0,2.5) -- (0.5,2.5) -- (0.5,2)  -- cycle;
    \draw [line width = 0.0mm, draw=magenta!60, fill=magenta!60, draw opacity = 0.5, fill opacity=0.5]
       (0.5,1.5) -- (0.5,2) -- (1,2) -- (1,1.5)  -- cycle;

    \draw [line width = 0.0mm, draw=magenta!60, fill=magenta!60, draw opacity = 0.5, fill opacity=0.5]
       (2.5,0.5) -- (2.5,0) -- (2,0) -- (2,0.5)  -- cycle;
    \draw [line width = 0.0mm, draw=magenta!60, fill=magenta!60, draw opacity = 0.5, fill opacity=0.5]
       (1.5,0.5) -- (1.5,1) -- (2,1) -- (2,0.5)  -- cycle;

    \node at (0.25,2.75) {\tiny{$1$}}; 
     \node at (2.75,0.25) {\tiny{$1$}};
     \node at (0.75,2.25) {\tiny{$1$}}; 
     \node at (2.25,0.75) {\tiny{$1$}};
     \node at (3.25,3.25) {\tiny{$1$}};
     \node at (4.75,4.75) {\tiny{$1$}};
     \node at (1.5,1.5) {\tiny{$\ddots$}};
     \node at (4.1,4.2) {\tiny{$\iddots$}};

     \draw[line width = 0.3mm, Green] (2.5,0) -- (2.5,2.5) -- (0,2.5) -- (0,0) -- cycle;

     \draw [decorate, decoration = {calligraphic brace, amplitude=2.5pt, mirror}, line width = 0.3mm] (5.35,0.05) --  (5.35,2.45)
    node[midway, right]{\tiny $m$};
    \draw [decorate, decoration = {calligraphic brace, amplitude=2.5pt, mirror}, line width = 0.3mm] (5.35,3.05) --  (5.35,4.95)
    node[midway, right]{\tiny $n-m-1$};

\end{tikzpicture}

\end{minipage}
\begin{minipage}[c]{0.65\textwidth}
\begin{minipage}[c]{0.1\textwidth}
(1)
\end{minipage}
\begin{minipage}[c]{0.48\textwidth}

\begin{tikzpicture}[scale = 0.45]

\node at (-1.5,2.5) {{$w =$}};

    \draw[line width = 0.3mm] (0.2,-0.2) -- (-0.2,-0.2) -- (-0.2,5.2) -- (0.2,5.2);
    \draw[line width = 0.3mm] (4.8,5.2) -- (5.2,5.2) -- (5.2,-0.2) -- (4.8,-0.2);

    \draw[line width = 0.1mm, Gray] (1.75,0) -- (1.75,5);
    \draw[line width = 0.1mm, Gray] (3.5,0) -- (3.5,5);
    \draw[line width = 0.1mm, Gray] (0,1.75) -- (5,1.75);
    \draw[line width = 0.1mm, Gray] (0,3.5) -- (5,3.5);

    \draw [line width = 0.0mm, draw=CornflowerBlue, fill=CornflowerBlue, draw opacity = 0.3, fill opacity=0.3]
       (0,0) -- (0,1.75) -- (1.75,1.75) -- (1.75,0) -- cycle;

    \draw[line width = 0.4mm, draw=magenta]
     (0, 1.75) -- (1.75,1.75) -- (1.75,0);

    \node at (0.25,2) {\tiny{$1$}}; 
    \node at (1.5,3.25) {\tiny{$1$}};
    \node at (0.9,2.7) {\tiny{$\iddots$}};

    \node at (2,0.25) {\tiny{$1$}}; 
    \node at (3.25,1.5) {\tiny{$1$}};
    \node at (2.7,1) {\tiny{$\iddots$}};

    \node at (3.75,3.75) {\tiny{$1$}}; 
    \node at (4.75,4.75) {\tiny{$1$}};
    \node at (4.25,4.35) {\tiny{$\iddots$}};

    \draw [decorate, decoration = {calligraphic brace, amplitude=2.5pt, mirror}, line width = 0.3mm] (5.35,0.05) --  (5.35,1.7)
    node[midway, right]{\tiny $l$};
    
    \draw [decorate, decoration = {calligraphic brace, amplitude=2.5pt, mirror}, line width = 0.3mm] (5.35,1.8) --  (5.35,3.45)
    node[midway, right]{\tiny $k$};

    \draw [decorate, decoration = {calligraphic brace, amplitude=2.5pt, mirror}, line width = 0.3mm] (5.35,3.55) -- (5.35,4.95)
    node[midway, right]{\tiny $n-k-l$};

\end{tikzpicture}

\end{minipage}
\begin{minipage}[c]{0.4 \textwidth}

\begin{tikzpicture}[scale = 0.45]

\node at (-1.5,1.75) {{$\Sigma =$}};

    \draw[line width = 0.3mm] (0.2,-0.2) -- (-0.2,-0.2) -- (-0.2,3.7) -- (0.2,3.7);
    \draw[line width = 0.3mm] (3.3,3.7) -- (3.7,3.7) -- (3.7,-0.2) -- (3.3,-0.2);

    \draw[line width = 0.1mm, Gray] (1.5,0) -- (1.5,3.5);
    \draw[line width = 0.1mm, Gray] (2,0) -- (2,3.5);
    \draw[line width = 0.1mm, Gray] (0,1.5) -- (3.5,1.5);
    \draw[line width = 0.1mm, Gray] (0,2) -- (3.5,2);

    \draw [line width = 0.0mm, draw=magenta!60, fill=magenta!60, draw opacity = 0.5, fill opacity=0.5]
       (0,3) -- (0,3.5) -- (0.5,3.5) -- (0.5,3)  -- cycle;
   \draw [line width = 0.0mm, draw=magenta!60, fill=magenta!60, draw opacity = 0.5, fill opacity=0.5]
       (1.5,1.5) -- (1.5,2) -- (2,2) -- (2,1.5)  -- cycle;
    \draw [line width = 0.0mm, draw=magenta!60, fill=magenta!60, draw opacity = 0.5, fill opacity=0.5]
       (3,0) -- (3.5,0) -- (3.5,0.5) -- (3,0.5)  -- cycle;
    \draw [line width = 0.0mm, draw=magenta!60, fill=magenta!60, draw opacity = 0.5, fill opacity=0.5]
       (0,3) -- (0,3.5) -- (0.5,3.5) -- (0.5,3)  -- cycle;
    \draw [line width = 0.0mm, draw=magenta!60, fill=magenta!60, draw opacity = 0.5, fill opacity=0.5]
       (1,2.5) -- (1,2) -- (1.5,2) -- (1.5,2.5)  -- cycle;
    \draw [line width = 0.0mm, draw=magenta!60, fill=magenta!60, draw opacity = 0.5, fill opacity=0.5]
    (2,1.5) -- (2,1) -- (2.5,1) -- (2.5,1.5)  -- cycle;

    \node at (0.75,2.95) {\tiny{$\ddots$}};
    \node at (2.75,0.95) {\tiny{$\ddots$}};

    \draw[line width = 0.4mm, draw=magenta]
     (0, 1.5) -- (1.5,1.5) -- (1.5,0);

   \draw [decorate, decoration = {calligraphic brace, amplitude=2.5pt, mirror}, line width = 0.3mm] (3.85,0.05) --  (3.85,1.4)
    node[midway, right]{\tiny $B$};

    \draw [decorate, decoration = {calligraphic brace, amplitude=2.5pt, mirror}, line width = 0.3mm] (3.85,2.1) -- (3.85,3.45)
    node[midway, right]{\tiny $A$};

\end{tikzpicture}

\end{minipage}
\vspace{0.2cm}

\begin{minipage}[c]{0.1\textwidth}
(2)
\end{minipage}
\begin{minipage}[c]{0.48\textwidth}
    
\begin{tikzpicture}[scale = 0.45]

\node at (-1.5,2.5) {{$w =$}};

    \draw[line width = 0.3mm] (0.2,-0.2) -- (-0.2,-0.2) -- (-0.2,5.2) -- (0.2,5.2);
    \draw[line width = 0.3mm] (4.8,5.2) -- (5.2,5.2) -- (5.2,-0.2) -- (4.8,-0.2);

     \draw[line width = 0.1mm, Gray] (0.5,0) -- (0.5,5);
     \draw[line width = 0.1mm, Gray] (0,0.5) -- (5,0.5);
    \draw[line width = 0.1mm, Gray] (2,0) -- (2,5);
    \draw[line width = 0.1mm, Gray] (3.5,0) -- (3.5,5);
    \draw[line width = 0.1mm, Gray] (0,2) -- (5,2);
    \draw[line width = 0.1mm, Gray] (0,3.5) -- (5,3.5);

    \draw [line width = 0.0mm, draw=CornflowerBlue, fill=CornflowerBlue, draw opacity = 0.3, fill opacity=0.3]
       (0.5,0.5) -- (0.5,2) -- (2,2) -- (2,0.5) -- cycle;
    \draw[line width = 0.4mm, draw=magenta]
     (0, 2) -- (2,2) -- (2,0);

     \node at (0.25,0.25) {\tiny{$1$}};
    
    \node at (0.75,2.25) {\tiny{$1$}}; 
    \node at (1.75,3.25) {\tiny{$1$}};
    \node at (1.25,2.8) {\tiny{$\iddots$}};

    \node at (2.25,0.75) {\tiny{$1$}}; 
    \node at (3.25,1.75) {\tiny{$1$}};
    \node at (2.75,1.35) {\tiny{$\iddots$}};

    \node at (3.75,3.75) {\tiny{$1$}}; 
    \node at (4.75,4.75) {\tiny{$1$}};
    \node at (4.25,4.35) {\tiny{$\iddots$}};

    \draw [decorate, decoration = {calligraphic brace, amplitude=2.5pt, mirror}, line width = 0.3mm] (5.35,0.55) --  (5.35,1.95)
    node[midway, right]{\tiny $t$};
    
    \draw [decorate, decoration = {calligraphic brace, amplitude=2.5pt, mirror}, line width = 0.3mm] (5.35,2.05) --  (5.35,3.45)
    node[midway, right]{\tiny $s$};

    \draw [decorate, decoration = {calligraphic brace, amplitude=2.5pt, mirror}, line width = 0.3mm] (5.35,3.55) -- (5.35,4.95)
    node[midway, right]{\tiny $n-l-s$};

\end{tikzpicture}

\end{minipage}
\begin{minipage}[c]{0.4\textwidth}

\begin{tikzpicture}[scale = 0.45]

\node at (-1.5,1.75) {{$\Sigma =$}};

    \draw[line width = 0.3mm] (0.2,-0.2) -- (-0.2,-0.2) -- (-0.2,3.7) -- (0.2,3.7);
    \draw[line width = 0.3mm] (3.3,3.7) -- (3.7,3.7) -- (3.7,-0.2) -- (3.3,-0.2);

    \draw[line width = 0.1mm, Gray] (1.5,0) -- (1.5,3.5);
    \draw[line width = 0.1mm, Gray] (2,0) -- (2,3.5);
    \draw[line width = 0.1mm, Gray] (0,1.5) -- (3.5,1.5);
    \draw[line width = 0.1mm, Gray] (0,2) -- (3.5,2);

    \draw [line width = 0.0mm, draw=magenta!60, fill=magenta!60, draw opacity = 0.5, fill opacity=0.5]
       (0,3) -- (0,3.5) -- (0.5,3.5) -- (0.5,3)  -- cycle;
   \draw [line width = 0.0mm, draw=magenta!60, fill=magenta!60, draw opacity = 0.5, fill opacity=0.5]
       (1.5,1.5) -- (1.5,2) -- (2,2) -- (2,1.5)  -- cycle;
    \draw [line width = 0.0mm, draw=magenta!60, fill=magenta!60, draw opacity = 0.5, fill opacity=0.5]
       (3,0) -- (3.5,0) -- (3.5,0.5) -- (3,0.5)  -- cycle;
    \draw [line width = 0.0mm, draw=magenta!60, fill=magenta!60, draw opacity = 0.5, fill opacity=0.5]
       (0,3) -- (0,3.5) -- (0.5,3.5) -- (0.5,3)  -- cycle;
    \draw [line width = 0.0mm, draw=magenta!60, fill=magenta!60, draw opacity = 0.5, fill opacity=0.5]
       (1,2.5) -- (1,2) -- (1.5,2) -- (1.5,2.5)  -- cycle;
    \draw [line width = 0.0mm, draw=magenta!60, fill=magenta!60, draw opacity = 0.5, fill opacity=0.5]
    (2,1.5) -- (2,1) -- (2.5,1) -- (2.5,1.5)  -- cycle;

    \node at (0.75,2.95) {\tiny{$\ddots$}};
    \node at (2.75,0.95) {\tiny{$\ddots$}};

    \draw[line width = 0.4mm, draw=magenta]
     (0, 2) -- (2,2) -- (2,0);

   \draw [decorate, decoration = {calligraphic brace, amplitude=2.5pt, mirror}, line width = 0.3mm] (3.85,0.05) --  (3.85,1.4)
    node[midway, right]{\tiny $B$};
    
    \draw [decorate, decoration = {calligraphic brace, amplitude=2.5pt, mirror}, line width = 0.3mm] (3.85,1.5) --  (3.85,2)
    node[midway, right]{\tiny $C$};

    \draw [decorate, decoration = {calligraphic brace, amplitude=2.5pt, mirror}, line width = 0.3mm] (3.85,2.1) -- (3.85,3.45)
    node[midway, right]{\tiny $A$};
    
\end{tikzpicture}

\end{minipage}
\end{minipage}

\caption{One can see the structure of $v$ and the structure of $w$ in each of the cases. The area of $v$ that is outlined in green 
corresponds to the covariance matrix $\Sigma$.
The diagonal of $\Sigma$ is highlighted in pink. In both cases, the Fulton conditions coming from $w$ are the same as the conditions from the corresponding CI statement, which are highlighted in the same colour in $\Sigma$.}
\end{figure}

\end{lemma}

\noindent Note that the statement of the previous theorem is not an if and only if. We can also obtain $\mathbb V(J_{A \indep B \mid C}) \cong \Nvw$ in other cases.

\begin{example} 
Consider a $4$-dimensional Gaussian random vector $X \sim \mathcal N (\mu, \Sigma)$ and the CI statement $1 \indep 3 \ |\ 2$. Also, consider the permutations $v=125643$ and $w= 645321$ in $S_6$.
Then, the opposite Rothe diagram and the area of $Z^{(v)}$ to which the rank one Fulton condition is applied are shown in the diagram below. We also draw the covariance matrix $\Sigma$ and outline the area $\Sigma_{A \cup C, B \cup C}$
on which the rank one CI condition applies. Then one can see that 
$\mathbb V(J_{A \indep B \mid C}) \cong \Nvw$.\\

\begin{minipage}{0.4\textwidth}
\begin{center}
\hspace{1.7cm}
\begin{tikzpicture}[scale = 0.5]

  \draw[step=1.0,black,thin] (0,0) grid (6,6);

    \node at (0.5,5.5) {{$1$}}; 
   \node at (1.5,4.5) {{$1$}}; 
   \node at (2.5,1.5) {{$1$}}; 
   \node at (3.5,0.5) {{$1$}};
   \node at (4.5,2.5) {{$1$}};
   \node at (5.5,3.5) {{$1$}};

  \draw [line width = 0.0mm, draw=CornflowerBlue, fill=CornflowerBlue, draw opacity = 0.3, fill opacity=0.3]
       (0,0) -- (0,5) -- (1,5) -- (1,4) -- (2,4) -- (2,1) -- (3,1) -- (3,0) -- cycle;

  \draw[line width = 0.7mm, draw=Gray, draw opacity = 0.4]
  (1.5, 5) -- (1.5,6);
  \draw[line width = 0.7mm, draw=Gray, draw opacity = 0.4]
  (2.5, 2) -- (2.5,6);
  \draw[line width = 0.7mm, draw=Gray, draw opacity = 0.4]
  (3.5, 1) -- (3.5,6);
  \draw[line width = 0.7mm, draw=Gray, draw opacity = 0.4]
  (4.5, 3) -- (4.5,6);
  \draw[line width = 0.7mm, draw=Gray, draw opacity = 0.4]
  (5.5, 4) -- (5.5,6);
  \draw[line width = 0.7mm, draw=Gray, draw opacity = 0.4];

  \draw[line width = 0.7mm, draw=Gray, draw opacity = 0.4]
  (4, 0.5) -- (6,0.5);
  \draw[line width = 0.7mm, draw=Gray, draw opacity = 0.4]
  (3, 1.5) -- (6,1.5);
  \draw[line width = 0.7mm, draw=Gray, draw opacity = 0.4]
  (5, 2.5) -- (6,2.5);
  \draw[line width = 0.7mm, draw=Gray, draw opacity = 0.4]
  (2, 4.5) -- (6,4.5);
  \draw[line width = 0.7mm, draw=Gray, draw opacity = 0.4]
  (1, 5.5) -- (6,5.5);

    \node at (0.5,4.5) {\small \textcolor{Blue}{$z_{21}$}}; 
    \node at (0.5,3.5) {\small \textcolor{Blue}{$z_{31}$}};
    \node at (0.5,2.5) {\small \textcolor{Blue}{$z_{41}$}}; 
    \node at (0.5,1.5) {\small \textcolor{Blue}{$z_{51}$}};
    \node at (0.5,0.5) {\small \textcolor{Blue}{$z_{61}$}}; 
    \node at (1.5,3.5) {\small \textcolor{Blue}{$z_{32}$}};
    \node at (1.5,2.5) {\small \textcolor{Blue}{$z_{42}$}}; 
    \node at (1.5,1.5) {\small \textcolor{Blue}{$z_{52}$}};
    \node at (1.5,0.5) {\small \textcolor{Blue}{$z_{62}$}}; 
    \node at (2.5,0.5) {\small \textcolor{Blue}{$z_{63}$}};  

  \draw[line width = 0.5mm, draw=magenta]
  (0, 2) -- (2,2) -- (2,0);
    
\end{tikzpicture}
\end{center}
\end{minipage}
\begin{minipage}{0.5\textwidth}
\begin{center}
\hspace{-3cm}
\begin{tikzpicture}[scale = 0.5] 
    \draw[step=1.0,Gray,thin] (0,0) grid (4,4);

    \draw[line width = 0.3mm] (0.2,-0.2) -- (-0.2,-0.2) -- (-0.2,4.2) -- (0.2,4.2);
    \draw[line width = 0.3mm] (3.8,4.2) -- (4.2,4.2) -- (4.2,-0.2) -- (3.8,-0.2);

    \draw[line width = 0.5mm, draw=magenta]
    (1,4) -- (1,2) -- (3,2) -- (3,4) -- cycle;

    \node at (0.5,3.5) {\small {$s_{11}$}};
    \node at (0.5,2.5) {\small {$s_{12}$}}; 
    \node at (0.5,1.5) {\small {$s_{13}$}};
    \node at (0.5,0.5) {\small {$s_{14}$}}; 
    \node at (1.5,3.5) {\small {$s_{12}$}};
    \node at (1.5,2.5) {\small {$s_{22}$}}; 
    \node at (1.5,1.5) {\small {$s_{23}$}};
    \node at (1.5,0.5) {\small {$s_{24}$}};
    \node at (2.5,3.5) {\small {$s_{13}$}};
    \node at (2.5,2.5) {\small {$s_{23}$}}; 
    \node at (2.5,1.5) {\small {$s_{33}$}};
    \node at (2.5,0.5) {\small {$s_{34}$}}; 
    \node at (3.5,3.5) {\small {$s_{14}$}};
    \node at (3.5,2.5) {\small {$s_{24}$}}; 
    \node at (3.5,1.5) {\small {$s_{34}$}};
    \node at (3.5,0.5) {\small {$s_{44}$}};

    \node at (-1.5,2) {{$\Sigma =$}};

\end{tikzpicture}
\end{center}
\end{minipage} \\

\end{example}

\noindent 
This family of CI varieties is always toric w.r.t.\ some toric action, since it is generated by linear terms or binomials and by \cite[Theorem 1.1]{biaggi2025binomiality}. 
As we have done for matrix Schubert varieties in Proposition~\ref{prop: complexity Mschu}, we can also state the complexity of these varieties with respect to the usual torus action. 

\begin{proposition}
    In the cases of Lemma \ref{KL CI}, the complexity of $\mathcal N_{v,w}$ is $m (\frac{m-1}{2})-|A||B|$.
    \begin{proof}
    		We know that $\dim \mathcal N_{v,w} = \frac{m(m+1)}{2}- |\Dc(w)|	$.
    		In case (1), $|\Dc(w)|= kl = |A||B|$ and in           		case (2),
    		$|\Dc(w)| = (k-1)(l-1) = |A||B|$.
    		We claim that in both cases, $\dim \sigma_{v,w} = m$: For case (2), this follows from Corollary \ref{cor: conn comp}, since $\Gvwt$ has one connected component coming from the lower left corner of $v$ and $n-m-1$ from the upper right corner. The graph $\Gvwt$ for case (1) is the same, except for removing the edges corresponding to elements of $\dom(v)$. 
            Since in this case, $\dom(v)$ does not intersect the elements of $\Zg$ corresponding to the diagonal of $\Sigma$, removing these edges does not change the connected components of $\Gvwt$.
    \end{proof}
\end{proposition}

\subsection{Quasi-independence Models}
\noindent In this section, we consider when the defining ideal of matrix Schubert variety can be identified with the ideal defined by a ``quasi-independence model'' that is defined on two variables. Quasi-independence models were first introduced by Cassinius in \cite{caussinus1965contribution}. Let us first define the class of quasi-independence models, that we are interested in, which are also referred to as two-way quasi-independence models.

Let $X$ and $Y$ be two discrete random variables, where $X$ can take the states $x_1, \ldots x_m$ and $Y$ takes the states $y_1,\ldots,y_n$. Then, we say that $X$ and $Y$ are in \emph{quasi-independence} if there are certain combinations of states of $X$ and $Y$ that occur with zero probability, but are otherwise independent of one another. 

\begin{definition}

Now, we wish to express this model algebraically.  Let $S\subseteq [m]\times[n]$ 
be the so-called state space of the model we are considering. It is the set that enumerates the ‘allowed’ pairs of states . Now, define $\mathbb{R}^S$ to be the real vector space of dimension $|S|$, with coordinates indexed by elements of $S$. We also index the coordinates of $\mathbb{R}^{m+n}$ by $(s_1,\ldots,s_m,t_1\ldots,t_n)$. 

\noindent
Define the monomial map $\phi^S:\mathbb{R}^{m+n} \rightarrow \mathbb{R}^S$ by
$$\phi^S_{ij}(s,t) = s_it_j.$$

\noindent
Then, we define the associated quasi-independence model by
$$\mathcal{M}_S = \phi^S(\mathbb{R}^{m+n}) \cap \Delta_{|S|-1}.$$
Here, $\Delta_{|S|-1}$ is the standard probability simplex of dimension $|S|-1$. 
\end{definition}

\noindent In particular, it has been shown that for any quasi-independence model $\mathcal{M}_S$, the defining ideal $I(\mathcal{M}_S)$ is a toric ideal - see for instance \cite[Chapter 6]{sullivant2023algebraic}. 

We can also associate bipartite graphs to $2$-way quasi-independence models. This is especially useful in our case. It is done by setting the indices of the states of $X$ as one set of vertices and the indices of the states of $Y$ as the other. We then draw an edge between $i$ and $j$ if $(i,j)$ is in the state space. For illustration, consider the following example.

\begin{example}\label{eg:q-i model}
Let the state space $S$ be equal to $$\{(1,1),(1,2),(2,1),(2,2),(2,3),(3,1),(3,2),(3,3)\}.$$ Then, we can draw the bipartite graph associated to the quasi-independence model $\mathcal{M}_S$ as follows (note that the literature -  \cite{coons2021quasi} - bipartite graphs of quasi-independence models are drawn in the following orientation).

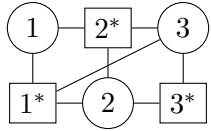
\begin{center}
\begin{tikzpicture}[every node/.style={minimum size=0.5cm}]
\node [shape=circle, draw=black] (A1) at (0,1) {$1$};
\node [shape=rectangle, draw=black] (B1) at (0,0) {$1^*$};
\node [shape=rectangle, draw=black] (B2) at (1,1) {$2^*$};
\node [shape=circle, draw=black] (A2) at (1,0) {$2$};
\node [shape=circle, draw=black] (A3) at (2,1) {$3$};
\node [shape=rectangle, draw=black] (B3) at (2,0) {$3^*$};

\path [-] (A1) edge (B1);
\path [-] (A1) edge (B2);
\path [-] (A2) edge (B2);
\path [-] (A2) edge (B1);
\path [-] (A2) edge (B3);
\path [-] (A3) edge (B2);
\path [-] (A3) edge (B1);
\path [-] (A3) edge (B3);
\end{tikzpicture}
\end{center}

\noindent
The vertices given by circles, refer to the indices of states of the $X$-variable, and those given by rectangles, refer to those associated to the $Y$-variable.

\end{example} 

\subsubsection{Matrix Schubert Ideals as Quasi-independence ideals}
\noindent Now, recall that the graph ${G}_w$, associated to the matrix Schubert variety $\Xw$ is also a bipartite graph. 
If we wish to identify a matrix Schubert ideal with a quasi-independence ideal, we require the former to be toric. Recall that $\overline{X_w}$ is toric if and only if $L'(w)$ is the union of disjoint hooks. In this case, we can then associate $\overline{X_w}$ to a quasi-independence model, as follows. 
\begin{itemize}
    \item For each disjoint hook $L_i'(w)$ of size $m$ by $n$, consider the associated, disjoint, section of $L(w)$, given by the staircase $L_i(w)$.
    \item  Then, each $L_i(w)$ can be associated to a $2$-way quasi-independence model $\mathcal{M}_{S_i}$ on variables $X$ and $Y$, where $X$ takes states in a copy of $[m]$ and $Y$ takes states in a copy of $[n]$.
    \item  We naturally see that $S = \{(x,y)\in L_i(w)\}$.
    \item  The quasi-independence model associated to the entirety of $\overline{X_w}$ is then given by the union of each of the disjoint quasi-independence models. 
\end{itemize}
Let us now consider the following example in the case that $L'(w)$ is formed of exactly one hook, to visualise this identification.

\begin{example}
    Let us consider the matrix Schubert variety $\Xw$, for $w = 251346 \in S_6$. Thus, we find that 
    $$L(w) = \{(3,2),(3,3),(4,2),(4,3),(4,4),(5,2),(5,3),(5,4)\},$$
    after drawing the opposite Rothe diagram. We can further draw the underlying undirected graph associated to $G_w$ as below.

    
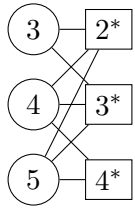
\begin{center}
    \begin{tikzpicture}[every node/.style={minimum size=0.5cm}]
\node [shape=circle, draw=black] (A3) at (0,0.5) {$5$};
\node [shape=circle, draw=black] (A2) at (0,1.5) {$4$};
\node [shape=circle, draw=black] (A1) at (0,2.5) {$3$};
\node [shape=rectangle, draw=black] (B3) at (1,0.5) {$4^*$};
\node [shape=rectangle, draw=black] (B2) at (1,1.5) {$3^*$};
\node [shape=rectangle, draw=black] (B1) at (1,2.5) {$2^*$};

\path [-] (A1) edge (B1);
\path [-] (A1) edge (B2);
\path [-] (A2) edge (B2);
\path [-] (A2) edge (B1);
\path [-] (A2) edge (B3);
\path [-] (A3) edge (B2);
\path [-] (A3) edge (B1);
\path [-] (A3) edge (B3);
\end{tikzpicture}
\end{center}
 \noindent Clearly, this graph is equivalent to the graph associated to the quasi-independence model, as drawn in Example~\ref{eg:q-i model}, after relabelling. We then see that the matrix Schubert ideal associated to $w$ is equivalent to the ideal of this quasi-independence model. 
\end{example}

\noindent
Finally, let us formalise this thinking as follows. Consider a hook $L_i'(w)$ and a staircase $L_i(w)$, each of height $m$ and width $n$. We draw the following diagram and label the diagram as follows. 

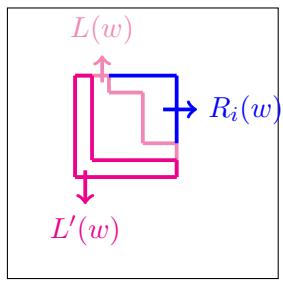
\begin{center}
\begin{tikzpicture}[scale = 0.9]
    \draw(0,0) rectangle (4,4);
    \draw[line width = 0.5mm, draw=magenta!60] (1,1.5) -- (1,3);
    \draw[line width = 0.5mm, draw=magenta!60] (1,3) -- (1.5,3);
    \draw[line width = 0.5mm, draw=magenta!60] (1.5,3) -- (1.5,2.75);
    \draw[line width = 0.5mm, draw=magenta!60] (1.5,2.75) -- (2,2.75);
    \draw[line width = 0.5mm, draw=magenta!60] (2,2.75) -- (2,2);
    \draw[line width = 0.5mm, draw=magenta!60] (2,2) -- (2.5,2);
    \draw[line width = 0.5mm, draw=magenta!60] (2.5,2) -- (2.5,1.5);
    \draw[line width = 0.5mm, draw=magenta!60] (2.5,1.5) -- (1,1.5);
    \draw[line width = 0.5mm, draw=magenta] (1.25,1.75)--(1.25,3);
    \draw[line width = 0.5mm, draw=magenta] (1.25,1.75)--(2.5,1.75);
    \draw[line width = 0.5mm, draw=magenta] (1.25,3)--(1,3);
    \draw[line width = 0.5mm, draw=magenta] (2.5,1.5)--(2.5,1.75);
    \draw[line width = 0.5mm, draw=magenta] (2.5,1.5)--(1,1.5);
    \draw[line width = 0.5mm, draw=magenta] (1,3)--(1,1.5);
    \draw[line width = 0.5mm, draw=blue]   (1.5,3) -- (2.5,3);
    \draw[line width = 0.5mm, draw=blue]   (2.5,2) -- (2.5,3);
    \draw[->,magenta!60,line width = 0.5mm] (1.4,2.9)--(1.4,3.3) node[above] {$L(w)$};
    \draw[->,magenta,line width = 0.5mm] (1.15,1.6)--(1.15,1.1) node[below] {$L'(w)$};
    \draw[->,blue,line width = 0.5mm] (2.3,2.5)--(2.8,2.5) node[right] {$R_i(w)$};
\end{tikzpicture}
\end{center}
\noindent Then, the associated quasi-independence model is given by $S = \mathcal{K}_{mn} \setminus\{(x,y)\in R_i(w)\}$. Here, $\mathcal{K}_{mn}$ denotes $[m]\times[n]$.

\begin{remark}
    With the quasi-independence model defined as above, $S$ must contain the following:
    \begin{itemize}
    \item $\{(x,1)|x\in [m]\}$,
    \item $\{(m,y)|y\in [n]\}$,
    \item $\{(x,2)|x\in [m]\}$,
    \item $\{(m-1,y)|y\in [n]\}$.
\end{itemize}

\noindent
Note that for simplicity of notation, here, we assume $S \subset [m] \times [n]$. The first two bullet points above refer to the elements of the hook. The latter two refer to the second column, from the left, and the second row, from the bottom, of the staircase. These must always be in $L(w)$, in order to guarantee a hook formation appearing in $\SW(w)\setminus D^{\circ}(w)$, and hence in $L'(w)$. 
\end{remark}

\subsubsection{Rational MLE}\noindent We now want to prove that all quasi-independence models that arise from toric matrix Schubert varieties have rational maximum likelihood estimate (MLE). Let us begin by recalling the concepts of likelihood probability and MLE. Given a parametric model (which in our case is a quasi-independence model) and some data in hand, we often wish to determine the model that ``best fits'' the data. One way to do this is to maximise the likelihood probability, which is defined as follows. 

\begin{definition}
    Let $D$ be some data from a discrete parametric model $\mathcal{M}$. The \emph{likelihood function} $L(p|D)$ is defined as the probability of observing the data $D$, under the distribution $p\in \mathcal{M}$. We say that $\hat{p}$ is the value of the distribution that maximises the likelihood function, if such a maximum does exist. Then, $\hat{p}$ is called the \emph{maximum likelihood estimate} (MLE). The \emph{log-likelihood function} $l(p|D)$ is given as the natural logarithm of the likelihood function. 
\end{definition}

\noindent We are especially interested in the case where the MLE can be written as a rational function of the data. In the case of quasi-independence models, there is 
a full classification of models with rational MLE. 

\begin{theorem}[\cite{coons2021quasi} Theorem 1.3]\label{thm: rat mle}
Let $S \subseteq [m] \times [n]$ and let $\mathcal{M}_S$ be the associated quasi-independence model. Let $\mathcal{G}_S$ be the bipartite graph associated to S. Then $\mathcal{M}_S$ has rational maximum likelihood estimate if and only if $\mathcal{G}_S$ is doubly chordal bipartite.
\end{theorem}
\noindent Note that a doubly chordal bipartite graph is a bipartite graph in which every cycle of length $\geq 6$ has at least two chords. Every bipartite graph that is not doubly chordal contains either a cycle of length greater than $2k,$ for $k\geq 3$, or a double square. The latter has the following structure:
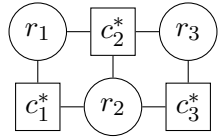
\begin{center}
\begin{tikzpicture} \label{fig:doublesquare}[every node/.style={minimum size=0.5cm}]
\node [shape=circle, draw=black] (A1) at (0,1) {$r_1$};
\node [shape=rectangle, draw=black] (B1) at (0,0) {$c_1^*$};
\node [shape=rectangle, draw=black] (B2) at (1,1) {$c_2^*$};
\node [shape=circle, draw=black] (A2) at (1,0) {$r_2$};
\node [shape=circle, draw=black] (A3) at (2,1) {$r_3$};
\node [shape=rectangle, draw=black] (B3) at (2,0) {$c_3^*$};

\path [-] (A1) edge (B1);
\path [-] (A1) edge (B2);
\path [-] (A2) edge (B2);
\path [-] (A2) edge (B1);
\path [-] (A2) edge (B3);
\path [-] (A3) edge (B2);
\path [-] (A3) edge (B3);
\end{tikzpicture}
\end{center}

\noindent Now, we will use Theorem \ref{thm: rat mle} to prove the following theorem. 

\begin{theorem}\label{thm: rational MLE}
    Every quasi-independence model that arises from a toric matrix Schubert variety has rational maximum likelihood estimate. 
\end{theorem}
\begin{proof}
We prove the theorem by showing that $\mathcal{G}_S$ is doubly chordal for a corresponding toric matrix Schubert variety, $\Xw$. Let $S\subseteq [m]\times[n]$. Without loss of generality, we may assume that $\mathcal{G}_S$ is connected. We differentiate cases based on the number of elements in the essential set $\Ess(w)$. For the connection between essential sets and independent sets we refer to \cite[Section 4.2]{portakal2023rigid}. If $|\Ess(w)| = 1$, then $\mathcal{G}_S$ is the complete bipartite graph $\mathcal K_{m,n}$. If $|\Ess(w)| = 2$, then there exists a unique two-sided maximal independent (stable) set and thus $\mathcal{G}_S$ is doubly chordal. Now, suppose that $\mathcal{G}_S$ contains a cycle $C_{2k} = (r_1,c^*_1, r_2, c^*_2, \cdots, r_k,c^*_k,r_1)$ for $k \geq 3$ with exactly one chord and assume that $\{r_1, c^*_i\}$ is the unique chord for $i \in [k-1] \backslash \{1\}$. For $|\Ess(w)| \geq 2$, by \cite[Lemma 4.6]{portakal2023rigid}, the structure of the maximal two-sided independent (stable) sets of $\mathcal{G}_S$ is known. Namely, let $C = C_1 \sqcup C_2$ and $C' = C'_1 \sqcup C'_2$ be two such independent sets. Then $C_1 \subsetneq C'_1 \subseteq [m]$ and $C'_2 \subsetneq C_2 \subseteq [n]$. 
If $k = 3$, then there are exactly two maximal two-sided independent sets that do not fit the desired structure, leading to a contradiction. For $k \geq 4$, take $C = \{r_1, r_2\} \sqcup X_1 \sqcup \{c^*_3, \cdots, c^*_{i-1}, c^*_{i+1}, c^*_{k-1}\} \sqcup Y_1$ and $C' = \{r_2, r_3\} \sqcup X_2 \sqcup \{c^*_4, \cdots, c^*_k\} \sqcup Y_2$, for some vertex sets $X_1, X_2 \subset [m]$ and $Y_1, Y_2 \subset [n]$. Then $c^*_k \in Y_1$ which is a contradiction to $C$ being an independent set by the definition of the cycle $C_{2k}$. The case where the cycle does not admit a chord follows similarly.  
\end{proof}

\subsection*{Acknowledgements}
We thank Pratik Misra, Colleen Robichaux, and Mahrud Sayrafi for the helpful discussion and useful comments.

\bibliographystyle{abbrv}
\bibliography{biblio}

\end{document}